\documentclass[11pt]{article}

\usepackage[text={17.1cm,24.6cm},centering]{geometry} 
\usepackage[utf8]{inputenc}
\usepackage{authblk}
\usepackage{blkarray}
\usepackage{cite}

\usepackage{amsfonts,amsmath,amssymb,amsthm}
\usepackage{latexsym,mathrsfs,mathtools,bm}
\usepackage{graphicx,subcaption,epsfig,caption,float,xcolor}
\usepackage{enumitem}
\usepackage{chngcntr}

\usepackage{tikz}
\usetikzlibrary{calc}

\usepackage{bookmark}

\newtheorem{theo}{Theorem}[section]

\newtheorem*{theo*}{Theorem}

\newtheorem{prop}[theo]{Proposition}

\newtheorem{defi}[theo]{Definition}
\newtheorem{example}[theo]{Example}

\newtheorem{remax}[theo]{Remark}

\theoremstyle{definition}

\newtheorem*{term*}{Notation/Terminology}

\newcommand{\cR}{\mathscr{R}}
\newcommand{\cS}{\mathscr{S}}

\def\rr{{\boldsymbol{\rho}}}
\def\orr{\overline{{\boldsymbol{\rho}}}}
\def\ox{{\overline{x}}}
\def\ogamma{{\overline{\gamma}}}

\def\tA{{{A}}}
\def\tB{{{Y}}}
\def\tC{{{C}}}

\def\tX{{X}}

\def\oalpha{\overline{{\alpha}}}
\def\obeta{\overline{{\beta}}}
\def\oN{\overline{{N}}}

\newcommand{\rouge}[1]{\textcolor{red}{#1}}

\usepackage{geometry}
\numberwithin{equation}{section}

\usepackage{ulem}
\usepackage{xcolor}

\newcommand\redsout{\bgroup\markoverwith{\textcolor{red}{\rule[0.5ex]{2pt}{0.4pt}}}\ULon}

\title{\bf Contiguity relations for finite families of\\ orthogonal polynomials in the Askey scheme}
\renewcommand*{\Affilfont}{\normalsize\small}

\author[1,a]{Nicolas Cramp\'e\,}
\author[2,b]{Lucia Morey\,}
\author[2,3,c]{Luc Vinet\,}
\author[4,d]{Meri Zaimi\,\vspace{.5em}}
\affil[1]{\textit{CNRS -- Universit\'e de Montr\'eal CRM - CNRS, Montr\'eal (Qu\'ebec), H3C 3J7, Canada}\vspace{.9em}}
\affil[2]{\textit{Centre de Recherches Math\'ematiques, Universit\'e de Montr\'eal, P.O. Box 6128,
\newline\vspace{.9em}
 Centre-ville Station, Montr\'eal (Qu\'ebec), H3C 3J7, Canada}}
\affil[3]{\textit{Insitut de Valorisation des Donn\'ees (IVADO), Montr\'eal (Qu\'ebec), H2S 3H1, Canada \vspace{.9em}}}
\affil[4]{\textit{Perimeter Institute for Theoretical Physics, Waterloo (Ontario), N2L 2Y5, Canada. \newline\vspace{.9em}}}

{
 \makeatletter
 \renewcommand\AB@affilsepx{: \protect\Affilfont}
 \makeatother
 \makeatletter
 \renewcommand\AB@affilsepx{, \protect\Affilfont}
 \makeatother
 \affil[a]{crampe1977@gmail.com}
 \affil[b]{lucia.morey@umontreal.ca}
  \affil[c]{luc.vinet@umontreal.ca}
   \affil[d]{mzaimi@perimeterinstitute.ca}
}
\date{\today}

\begin{document}

\maketitle

\bigskip\bigskip 

\begin{center}
\begin{minipage}{14cm}
\begin{center}
{\bf Abstract} 
\\
\end{center}
This paper classifies the contiguity relations for finite families of polynomials within the ($q$-)Askey scheme. The necessary and sufficient conditions for the existence of these contiguity relations are presented first. These conditions are then solved, yielding a comprehensive list of contiguity relations for these various families of polynomials. Furthermore, we demonstrate that all contiguity relations correspond to spectral transforms.
    \end{minipage}
\end{center}

\medskip

\begin{center}
\begin{minipage}{13cm}

\textbf{MSC2020 database:} 33C80; 33C45; 16G60
\end{minipage}
\end{center}

\clearpage
%\tableofcontents
\newpage

\section{Introduction}

Hypergeometric orthogonal polynomials and their $q$-analogs \cite{AskeyWilson, GasperRahman2004, Chiara,KoekoekLeskyetal2010 } form an important class of special functions. This class comprises various families of polynomials which are organized in the  ($q$-)Askey scheme and satisfy a three-term recurrence relation in view of Favard's theorem. They are bispectral and even called classical since they also verify a second-order
differential or a three-term ($q$-)difference equation. It is of interest to study in detail the properties of special functions to facilitate their use in various ways and contexts. 

In this paper, we restrict ourselves to the finite families of polynomials of the ($q$-)Askey scheme, the most general being the ($q$-)Racah polynomials. 
We denote by $R_i(x; \rr)$ ($i=0,1,\dots,N$ and $N$ is a non-negative integer) the polynomials of such a family. These polynomials are of degree $i$ with respect to a variable $\lambda_{x,\rr}$. We denote the list of parameters of these polynomials by $\rr=\rho_1,\rho_2,\dots$, which includes the integer $N$.
The polynomials $R_i(x,\rr)$ satisfy a three-term recurrence relation of the following type 
\begin{align}\label{eq:recuR}
& \lambda_{x,\rr}\; R_i(x; \rr)=A_{i,\rr}\; R_{i+1}(x; \rr)
-(A_{i,\rr}+C_{i,\rr})\;R_i(x; \rr)+C_{i,\rr}\; R_{i-1}(x; \rr)\,,
\end{align}
with $C_{0,\rr}=0$, $A_{N,\rr}=0$ and $R_0(x,\rr)=R_i(0,\rr)=1$.

The purpose of this paper is to classify the contiguity relations of the polynomials $R_i(x; \rr)$, that is, equations of the form
\begin{align}
\lambda^{+}_{x; \rr}\; R_i(x; \rr)&=\sum_{\epsilon \in \cS} \Phi^{\epsilon,+}_{i}\ R_{i+\epsilon}(\ox; \orr)\,,\label{eq:cons1}
\end{align}
where $\ox$ is a shift of $x$: $\ox=x+\eta$ with $\eta\in\{0,+1,-1\}$, $\orr$ the ensemble of modified parameters and $\cS$ a set of one of the following types:
\begin{itemize}
    \item type ${A_2}$ is   $\cS=\{0,-1\}$;
    \item type ${B_2}$ is  $\cS=\{0,-1,1\}$;
    \item type ${B'_2}$ is $\cS=\{0,-1,-2\}$.
\end{itemize} 
The names for these different types of sets come from the use of the associated contiguity relations to compute the recurrence relations for certain bivariate polynomials which are defined with the root systems $A_2$ or $B_2$ (see \cite{IT06,crampe2023FLP,crampe2023lambdagriffiths,crampe2023griffithsRacah,crampe2024griffithsqRacah}). The case $B'_2$ is a transformation of the case obtained from the root system $B_2$. The main result of this paper is a complete classification of these relations in the cases where $\oN$ belonging to $\orr$ is chosen among $\{N,N+1,N-1,N-2,N+2\}$. For some families of polynomials, the classification is achieved without restricting the values of $\oN$, and all the solutions that are found are such that $\oN$ belongs to the set written above.

Contiguity relations are important features of orthogonal polynomials. Indeed, these relations are at the origin of the  ``doubling'' procedure  \cite{MARCELLAN1997,Jafarov2015,Oste2015Doubling,Oste2016}, which allows to introduce new families of ``classical'' orthogonal polynomials such as the big $-1$ Jacobi polynomials \cite{Vinet2010q1Jacobi}.
These constructions are generalizations of the Chihara method for obtaining new orthogonal polynomials from orthogonal polynomials and their kernel partner \cite{Chiara}.
They also appear in the study of new exactly solvable quantum dynamical systems \cite{Stoilova2011,Jafarov2014,Oste2017}, in the theory of the representation of the rank 2 Racah algebra \cite{crampe2023Racah4}, 
in the study of the bispectral properties of some multivariate polynomials \cite{crampe2023lambdagriffiths,crampe2023griffithsRacah,crampe2024griffithsqRacah}, in the construction of factorized Leonard pairs \cite{crampe2023FLP}, and in the analysis of the bivariate $P$- and $Q$-polynomial structures for certain association schemes \cite{crampe2023nonbinaryJohnson, bernard2024attenuated}. 

We show that, for each contiguity relation \eqref{eq:cons1}, there exists the following other contiguity relation
\begin{align}
\lambda^{-}_{x; \rr}\;  R_i(\ox; \orr)&=\sum_{-\epsilon \in \cS} \Phi^{\epsilon,-}_{i\ }\ R_{i+\epsilon}(x; \rr)\,.\label{eq:cons3}
\end{align}
Let us remark that, using the Wilson duality, \textit{i.e.}  the fact that for any polynomials $R_i(x,\rr)$, there exist other polynomials $\cR_{i}(x,\rr')$ such that $\cR_{i}(x,\rr')=R_{x}(i,\rr)$, the contiguity relations of the type \eqref{eq:cons1} and \eqref{eq:cons3} for $\cR_{i}(x,\rr')$ lead to  contiguity  relations of a difference nature for $R_i(x,\rr)$ since the duality connection leads then to shifts in the variable $i$ of $R_{x}(i,\rr)$. Some of these equations already appear in the literature, see for instance \cite{KoekoekLeskyetal2010,Jafarov2014,Stoilova2011}. We note that the case $A_2$ leads to forward and backward shift operators.

One of the results shown in this paper is that the whole list of $B_2$ and $B_2'$ contiguity relations can be obtained from the $A_2$-contiguity relations. Therefore, we provide the list of $A_2$-contiguity relations for all the families of polynomials, and those of type $B_2$ and $B'_2$ for the $q$-Racah polynomials only.

Furthermore, we show that all the $A_2$-contiguity relations \eqref{eq:cons1} are Christoffel transforms \cite{Szego, Chiara}, and the associated relations \eqref{eq:cons3} are Geronimus transforms \cite{Geronimus,Geronimus2}. We may point out that the compatibility between the Christoffel and Geronimus transforms leads to the discrete Toda equations \cite{Zhedanov-dToda,Satoshi-dToda}. The study in this paper can equivalently be presented as the cataloguing of the Christoffel transforms of polynomials of the ($q$-)Askey scheme for which the kernel polynomials are also polynomials of the ($q$-)Askey scheme.

We also treat the case of the Bannai--Ito polynomials \cite{BI84}. We show that there are no $A_2$-contiguity relations for them. However, the limit $q\to -1$ of the $A_2$-contiguity relations suggests looking for relations of the type \eqref{eq:cons1} where the l.h.s. involves Bannai--Ito polynomials and the r.h.s. involves complementary Bannai--Ito polynomials (defined in \cite{VinetCBI,Vinet2013}). We provide a list of such relations as well as of the associated relations of the form \eqref{eq:cons1}. We show that this list suffices to obtain all the $B_2$ and $B'_2$ contiguity relations for the Bannai--Ito polynomials.

Finally, we show that $A_2$-contiguity relations can also exist between terminating generalized ${}_4\phi_3$ $q$-hypergeometric functions which are not necessarily balanced. One recovers contiguity relations for the $q$-Racah polynomials when imposing that the ${}_4\phi_3$ functions are $1$-balanced.

This paper is organized as follows. Section \ref{sec:cont} gives the sufficient and necessary conditions between $\rr$, $\orr$, $x$ and $\ox$ for contiguity relations to exist. The results are presented in different propositions depending on the type $A_2$, $B_2$ or $B'_2$. 
The connections with the Christoffel and Geronimus transforms  are discussed in Section \ref{sec:CG} while Section \ref{sec:cons} is concerned with the resolution of the constraints and provides the complete list of the contiguity relations (of the classes considered) for the various families of polynomials. Section \ref{sec:cons}, together with Appendices \ref{app:B2} and \ref{app:B2p}, can be used independently as a compendium of the different contiguity relations.
The Bannai--Ito polynomials are studied in Section \ref{sec:BI}, where a list of relations between Bannai--Ito and complementary Bannai--Ito polynomials is given. Finally, Section \ref{sec:gen} provides various formulas involving  generalized $q$-hypergeometric functions of the ${}_4\phi_3$ type which are not necessarily balanced.

%%%%%%%%%%%%%%%%%%%%%%%%%%%%%%%%%
%%%%%%%%%%%%%%%%%%%%%%%%%%%%%%%%%
%%%%%%%%%%%%%%%%%%%%%%%%%%%%%%%%%
%%%%%%%%%%%%%%%%%%%%%%%%%%%%%%%%%

\section{Contiguity relations\label{sec:cont} }

This section is devoted to providing the sufficient and necessary conditions for the contiguity relations of types $A_2$, $B_2$ and $B'_2$ to exist. We demand that 
\begin{align}
\label{eq:cond-lambda}
\lambda_{x,\rr}=\zeta\lambda_{\ox,\orr}-\xi \,,
\end{align}
for some parameters $\zeta$ and $\xi$. This constraint corresponds to asking that $R_i(x,\rr)$ and $R_i(\ox,\orr)$ are both polynomials of degree $i$ of the same variable.
For convenience, we will use from now on the notations
\begin{equation}\label{eq:XY}
    \tX_{i,\rr}:=\tA_{i-1,\rr}\tC_{i,\rr}\,, \qquad \tB_{i,\rr}:=-(\tA_{i,\rr}+\tC_{i,\rr}), 
\end{equation}
where $\tA_{i,\rr}$ and $\tC_{i,\rr}$ are the coefficients appearing in the recurrence relation \eqref{eq:recuR} for the polynomials $R_i(x;\rr)$.

%%%%%%%%%%%%%%%%%%%%%%%%%%%%%%%%%
%%%%%%%%%%%%%%%%%%%%%%%%%%%%%%%%%
%%%%%%%%%%%%%%%%%%%%%%%%%%%%%%%%%
%%%%%%%%%%%%%%%%%%%%%%%%%%%%%%%%%
\subsection{Contiguity relations of type $A_2$ \label{sec:A2crr}}

Contiguity relations \eqref{eq:cons1} and \eqref{eq:cons3} of type $A_2$ are studied in this subsection. We start with a definition of some constraints which will allow the existence of such relations. 

\begin{defi}\label{def:CA2}
The constraints $(\mathfrak{C}_{A_2})$ are defined by the requirement that the following expressions 
\begin{align}
   & \zeta^{2i} \frac{\zeta\tB_{i,\orr}-\tB_{i,\rr}-\xi}{\tX_{i+1,\rr}-\zeta^2\tX_{i,\orr}} 
    \prod_{k=1}^{i}
    \frac{\tX_{k,\orr}}
    {\tX_{k,\rr}} \,,\label{eq:indei1}\\
    &\zeta^{2i}\frac{ \tX_{i+1,\rr}-\zeta^{2}\tX_{i+1,\orr} }
    {(\zeta\tB_{i,\orr}-\tB_{i+1,\rr}-\xi)\tX_{i+1,\rr}}
    \prod_{k=1}^{i} \frac{\tX_{k,\orr}}
    {\tX_{k,\rr}} \,,
    \label{eq:indei2}
\end{align} 
   are equal and independent of $i$ for $i\geq 0$, where the notations \eqref{eq:XY} have been used. 
\end{defi}

\begin{prop}\label{pro:cr1}
The polynomials $R_{i}(x;\rr)$ chosen such that $\lambda_{x,\rr}=\zeta\lambda_{\ox,\orr}-\xi$ satisfy the contiguity relation \eqref{eq:cons1} of type $A_2$
\begin{equation}\label{eq:cons1a}
\lambda^+_{x,\rr}R_{i}(x;\rr)=\Phi_{i}^{0,+}R_{i}(\ox;\orr) +\Phi_{i}^{-1,+}R_{i-1}(\ox;\orr)\,, \qquad i\geq 0\,, 
\end{equation}
(with the convention $\Phi_{0}^{-1,+} = 0$ and with $\Phi_{0}^{0,+} \neq 0$) 
if and only if the coefficients are given (up to a global normalization) by
\begin{align}
& \lambda^+_{x,\rr}=1, \label{eq:lam+}\\
& \Phi_{i}^{0,+}  =\zeta^i
 \prod_{k=0}^{i-1} \frac{\tA_{k,\orr}}{\tA_{k,\rr}}\,, & &i\geq 0, \label{eq:Phii0} \\
&\Phi_{i}^{-1,+}  =\zeta^{-i+1}
 \left(1-\frac{\zeta\tA_{0,\orr}+\xi}{\tA_{0,\rr}}\right)
 \prod_{k=1}^{i-1} \frac{\tC_{k+1,\rr}}{\tC_{k,\orr}}\, , & &i\geq 1, \label{eq:Phii1}
\end{align}
%with $\Phi_{i,j}^{-1,1}\neq 0$ for $i\geq 1$, 
and the constraints $(\mathfrak{C}_{A_2})$ of Definition \ref{def:CA2} are satisfied.
\end{prop}
\proof
Relation \eqref{eq:cons1a} for $i=0$ together with $R_0(x; \rr)=R_0(\ox; \orr)=1$ implies
\begin{align}\label{eq:temp1}
\lambda^+_{x,\rr}=\Phi_{0}^{0,+}\,.\end{align}

Using the recurrence relation \eqref{eq:recuR} satisfied by $R_i(x,\rr)$, one obtains constraints between the coefficients of the recurrence relation and those of the contiguity relation. Indeed,
\begin{align}
\tA_{i,\rr}\lambda^+_{x,\rr}R_{i+1}(x; \rr) 
=&\lambda^+_{x,\rr}\Big( 
\big(\lambda_{x,\rr}  -\tB_{i,\rr}\big)\; R_i(x; \rr)-\tC_{i,\rr}\; R_{i-1}(x; \rr)\Big)\\
=& 
\big(\zeta\lambda_{\overline{x},\orr} -\xi -\tB_{i,\rr}\big)\;\Big(\Phi^{0,+}_{i} R_i(\ox; \orr)+\Phi^{-1,+}_{i} R_{i-1}(\ox; \orr)  \Big)\\
&\hspace{0.5cm} -\tC_{i,\rr}\; \Big(\Phi^{0,+}_{i-1} R_{i-1}(\ox; \orr)+\Phi^{-1,+}_{i-1} R_{i-2}(\ox; \orr)  \Big)\,.\nonumber
\end{align}
The last line is arrived at using the contiguity relation and condition \eqref{eq:cond-lambda} for $\lambda_{x,\rr}$. It can be transformed using the recurrence relation as follows
\begin{align}
 & \Phi^{0,+}_{i}
\Big(\zeta\tA_{i,\orr}\; R_{i+1}(\ox; \orr)
+(\zeta\tB_{i,\orr}-\tB_{i,\rr}-\xi)\; R_{i}(\ox; \orr)
+\zeta\tC_{i,\orr}\;R_{i-1}(\ox; \orr)\Big)
\\
+&\Phi^{-1,+}_{i} 
\Big(\zeta\tA_{i-1,\orr}\; R_{i}(\ox; \orr)
+(\zeta\tB_{i-1,\orr}-\tB_{i,\rr}-\xi)\; R_{i-1}(\ox; \orr)
+\zeta\tC_{i-1,\orr}\;R_{i-2}(\ox; \orr)\Big)\nonumber\\
-&\tC_{i,\rr}\; \Big(\Phi^{0,+}_{i-1} \;R_{i-1}(\ox; \orr)+\Phi^{-1,+}_{i-1}\; R_{i-2}(\ox; \orr)  \Big)\,.\nonumber   
\end{align}
Comparing with \eqref{eq:cons1a} with $i$ replaced by $i+1$, one obtains the following constraints
\begin{subequations}
\begin{align}
& \tA_{i,\rr}\Phi_{i+1}^{0,+}  =
\zeta\tA_{i,\orr} \Phi^{0,+}_{i} \,, & & i \geq 0 \,,\label{eq:cont1}\\
 & \tA_{i,\rr}\Phi_{i+1}^{-1,+} =
(\zeta\tB_{i,\orr}-\tB_{i,\rr}-\xi)\Phi^{0,+}_{i}+ 
\zeta\tA_{i-1,\orr}\Phi^{-1,+}_{i}  \,, & & i \geq 0 \,, \label{eq:cont2}\\
&0=\zeta\tC_{i,\orr}\Phi^{0,+}_{i}+ 
(\zeta\tB_{i-1,\orr}-\tB_{i,\rr}-\xi)\Phi^{-1,+}_{i}
-\tC_{i,\rr} \Phi^{0,+}_{i-1}\,, & & i \geq 1 \,,\label{eq:cont3}\\
&0= \zeta\tC_{i-1,\orr}\Phi^{-1,+}_{i}
-\tC_{i,\rr}\Phi^{-1,+}_{i-1} \, , & & i \geq 2 \,.\label{eq:cont4}
\end{align}
\end{subequations}
Using recursively  \eqref{eq:cont1} and \eqref{eq:cont4}, one finds
\begin{align}
\Phi_{i}^{0,+}  =
 \Phi_{0}^{0,+}\zeta^i\prod_{k=0}^{i-1} \frac{\tA_{k,\orr}}{\tA_{k,\rr}} \quad (i\geq 0)\,, \qquad 
  \Phi_{i}^{-1,+}  =
 \Phi^{-1,+}_{1}\zeta^{-i+1}
 \prod_{k=1}^{i-1} \frac{\tC_{k+1,\rr}}{\tC_{k,\orr}} \quad (i\geq 1)\,. \label{eq:Phii}
\end{align}
Substituting these expressions of $\Phi_{i}^{0,+}$ and $\Phi_{i}^{-1,+}$ in \eqref{eq:cont2} and \eqref{eq:cont3}, yields
\begin{align}
\frac{\Phi_{1}^{-1,+}}{\Phi_{0}^{0,+} \tC_{1,\rr}}&=\zeta^{2i}
    \frac{\zeta\tB_{i,\orr}-\tB_{i,\rr}-\xi}{\tX_{i+1,\rr}-\zeta^2\tX_{i,\orr}} 
    \prod_{k=0}^{i-1}  \frac{\tX_{k+1,\orr}}
    {\tX_{k+1,\rr}} \,,\quad i\geq 0\label{eq:indei1a}\\
    &=\zeta^{2i-2}\frac{\tX_{i,\rr}-\zeta^{2}\tX_{i,\orr} }
    {(\zeta\tB_{i-1,\orr}-\tB_{i,\rr}-\xi)\tX_{i,\rr}}
    \prod_{k=0}^{i-2}
    \frac{\tX_{k+1,\orr}}
    {\tX_{k+1,\rr}} \,,\quad i\geq 1\,. \label{eq:indei2a}
\end{align}
This leads to the constraints $(\mathfrak{C}_{A_2})$ of Definition \ref{def:CA2}.
Substituting \eqref{eq:indei1a} (with $i=0$) in \eqref{eq:Phii}, one finds
\begin{equation}
\Phi_{i}^{-1,+}  =\Phi_{0}^{0,+}
 \zeta^{-i+1}\left(1-\frac{\zeta\tA_{0,\orr}+\xi}{\tA_{0,\rr}}\right)
 \prod_{k=1}^{i-1} \frac{\tC_{k+1,\rr}}{\tC_{k,\orr}}\, .
\end{equation}
The factor $\Phi_{0}^{0,+}$ is a global normalization which can be fixed to $1$, leading to formulas \eqref{eq:lam+}--\eqref{eq:Phii1}.
This proves the direct implication in the proposition.

By a recursion on $i$, the requirement that the expressions \eqref{eq:indei1} and \eqref{eq:indei2} do not depend on $i$ and are equal, is shown to impose sufficient conditions for $R_i(x,\rr)$ to satisfy the contiguity relation \eqref{eq:cons1a} with the coefficients given by \eqref{eq:Phii0} and \eqref{eq:Phii1}.\endproof

Similar results can be obtained for relation \eqref{eq:cons3}.
\begin{prop}\label{pro:cr2}
The polynomials $R_{i}(x;\rr)$ chosen such that $\lambda_{x,\rr}=\zeta\lambda_{\ox,\orr}-\xi$ satisfy the contiguity relation \eqref{eq:cons3} of type $A_2$
\begin{equation}\label{eq:cons3a}
\lambda^-_{x,\rr} R_{i}(\ox;\orr)=\Phi_{i}^{1,-}R_{i+1}(x;\rr)+ \Phi_{i}^{0,-}R_{i}(x;\rr) \,, \qquad i\geq 0\,, 
\end{equation}
(with $\Phi_{0}^{0,-}\neq 0$) if and only if the coefficients are given  (up to a global normalization) by
\begin{align}
& \lambda^-_{x,\rr}=(\tA_{0,\rr}-\zeta\tA_{0,\orr}-\xi) \left(\frac{\lambda_{x,\rr}}{A_{0,\rr}}+1\right)+\tC_{1,\rr}\,,\\
& \Phi_{i}^{1,-}=(\tA_{0,\rr}-\zeta\tA_{0,\orr}-\xi)\zeta^{-i}\prod_{k=1}^{i} 
    \frac{\tA_{k,\rr}}
    {\tA_{k-1,\orr}}\,, & &i\geq 0 \,, \\
&    \Phi_{i}^{0,-}=\zeta^i\tC_{1,\rr}\prod_{k=1}^{i} 
    \frac{\tC_{k,\orr}}
    {\tC_{k,\rr}}\,, & &i\geq 0\,,
\end{align}
and the constraints $(\mathfrak{C}_{A_2})$ of Definition \ref{def:CA2} are satisfied.
\end{prop}
\proof Similar to the proof of Proposition \ref{pro:cr1}.
\endproof

Let us emphasize that relations \eqref{eq:cons1a} and \eqref{eq:cons3a} both require the same constraints $(\mathfrak{C}_{A_2})$ in order to be satisfied. In other words, the same constraints $(\mathfrak{C}_{A_2})$ allow the existence of both contiguity relations.

%%%%%%%%%%%%%%%%%%%%%%%%%%%%%%%%%
%%%%%%%%%%%%%%%%%%%%%%%%%%%%%%%%%
%%%%%%%%%%%%%%%%%%%%%%%%%%%%%%%%%
%%%%%%%%%%%%%%%%%%%%%%%%%%%%%%%%%
\subsection{Contiguity relations of type $B_2$ \label{sec:B2crr}}

Contiguity relations \eqref{eq:cons1} and \eqref{eq:cons3} of type $B_2$ are studied in this subsection. Similarly to the case $A_2$, both relations will require some constraints to be satisfied.

\begin{defi}\label{def:CB2}
The constraints $(\mathfrak{C}_{B_2})$ are defined by the requirement that the following expressions 
\begin{align}
   & 
 \zeta^{2i}
    \frac{\frac{\zeta^2\tX_{i+1,\orr}-\tX_{i,\rr}}{\xi+\tB_{i,\rr}-\zeta\tB_{i,\orr}}- \frac{\zeta^2\tX_{i+2,\orr}-\tX_{i+1,\rr}}{\xi+\tB_{i+1,\rr}-\zeta\tB_{i+1,\orr}}+\zeta\tB_{i+1,\orr}-\tB_{i,\rr}-\xi}{\frac{\tX_{i+1,\rr}}{\zeta^2\tX_{i+1,\orr}}\frac{\zeta^2\tX_{i+1,\orr}-\tX_{i+2,\rr}}{\xi+\tB_{i+1,\rr}-\zeta\tB_{i+1,\orr}}-\frac{\zeta^2\tX_{i,\orr}-\tX_{i+1,\rr}}{\xi+\tB_{i,\rr}-\zeta\tB_{i,\orr}}}
    \prod_{k=1}^{i}  \frac{\tX_{k,\orr}}
    {\tX_{k,\rr}} \,, %\qquad i\geq 0,
    \label{eq:B2indei1}\\
    & 
    \zeta^{2i}\frac{\frac{\zeta^2\tX_{i+1,\orr}-\tX_{i,\rr}}{\xi+\tB_{i,\rr}-\zeta\tB_{i,\orr}} - \frac{\zeta^2\tX_{i+1,\orr}}{\tX_{i+1,\rr}}\, \frac{\zeta^2\tX_{i+2,\orr}-\tX_{i+1,\rr}}{\xi+\tB_{i+1,\rr}-\zeta\tB_{i+1,\orr}}}{\frac{\zeta^2\tX_{i+1,\orr}-\tX_{i+2,\rr}}{\xi+\tB_{i+1,\rr}-\zeta\tB_{i+1,\orr}}-\frac{\zeta^2\tX_{i,\orr}-\tX_{i+1,\rr}}{\xi+\tB_{i,\rr}-\zeta\tB_{i,\orr}}+\zeta\tB_{i,\orr}-\tB_{i+1,\rr}-\xi}
    \prod_{k=1}^{i}  \frac{\tX_{k,\orr}}{\tX_{k,\rr}}\,, %\qquad i\geq 0,
    \label{eq:B2indei2}
\end{align} 
are equal and independent of $i$ for $i \geq 0$, where the notations \eqref{eq:XY} have been used.
\end{defi}

Let us remark that the 
vanishing of the denominators in equations \eqref{eq:B2indei1} and \eqref{eq:B2indei2} corresponds to constraints $(\mathfrak{C}_{A_2})$. 

\begin{prop}\label{pro:B2cr1}
The polynomials $R_{i}(x;\rr)$ chosen such that $\lambda_{x,\rr}=\zeta\lambda_{\ox,\orr}-\xi$ satisfy the contiguity relation of type $B_2$
\begin{equation}\label{eq:B2cons1a}
\lambda^+_{x,\rr} R_{i}(x;\rr)=\Phi_{i}^{1,+}R_{i+1}(\ox;\orr) +\Phi_{i}^{0,+}R_{i}(\ox;\orr)+\Phi_{i}^{-1,+}R_{i-1}(\ox;\orr)\,, \qquad i\geq 0\,, 
\end{equation}
(with the convention $\Phi_{0}^{-1,+} = 0$ and with $\Phi_0^{1,+}\neq 0$) 
if and only if the coefficients are given  (up to a global normalization) by
\begin{align}
&\lambda^+_{x,\rr}=
 1+\frac{\zeta\tC_{1,\orr}}{\xi+\zeta\tA_{0,\orr}-\tA_{0,\rr}}+\frac{\lambda_{x,\rho}+\xi}{\zeta\tA_{0,\orr}}-\frac{\frac{\zeta^2\tX_{1,\orr}}{\xi+\tB_{0,\rr}-\zeta\tB_{0,\orr}}- \frac{\zeta^2\tX_{2,\orr}-\tX_{1,\rr}}{\xi+\tB_{1,\rr}-\zeta\tB_{1,\orr}}+\zeta\tB_{1,\orr}-\tB_{0,\rr}-\xi}{\frac{(\xi+\zeta\tA_{0,\orr}-\tA_{0,\rr})(\zeta^2\tX_{1,\orr}-\tX_{2,\rr})}{\zeta\tC_{1,\orr}(\xi+\tB_{1,\rr}-\zeta\tB_{1,\orr})}+\zeta\tA_{0,\orr}}\,,\label{eq:lamB+}\\
&  \Phi_{i}^{1,+}  = \zeta^i
 \prod_{k=1}^{i} \frac{\tA_{k,\orr}}{\tA_{k-1,\rr}}\,, \qquad i\geq 0\,, \label{eq:B2Phii0} \\
 &\Phi_{i}^{-1,+}  =\zeta^{-i}
 \frac{\frac{\zeta^2\tX_{1,\orr}}{\xi+\tB_{0,\rr}-\zeta\tB_{0,\orr}}- \frac{\zeta^2\tX_{2,\orr}-\tX_{1,\rr}}{\xi+\tB_{1,\rr}-\zeta\tB_{1,\orr}}+\zeta\tB_{1,\orr}-\tB_{0,\rr}-\xi}{\tA_{0,\rr}\tA_{0,\orr}\left(\frac{\zeta^2\tX_{1,\orr}-\tX_{2,\rr}}{\zeta^2\tX_{1,\orr}(\xi+\tB_{1,\rr}-\zeta\tB_{1,\orr})}+\frac{1}{\xi+\tB_{0,\rr}-\zeta\tB_{0,\orr}}\right)} \prod_{k=1}^{i-1} \frac{\tC_{k+1,\rr}}{\tC_{k,\orr}}\, , \quad i\geq 1 \,,\label{eq:B2Phii2} \\
 &\Phi_{i}^{0,+}  =
 \frac{\zeta^2\tX_{i+1,\orr}-\tX_{i,\rr}}{\zeta\tA_{i,\orr}(\xi+\tB_{i,\rr}-\zeta\tB_{i,\orr})}\Phi^{1,+}_{i}
 +\frac{\zeta^2\tX_{i,\orr}-\tX_{i+1,\rr}}{\tC_{i+1,\rr}(\xi+\tB_{i,\rr}-\zeta\tB_{i,\orr})}\Phi^{-1,+}_{i+1} \,, \qquad i\geq 0\,, \label{eq:B2Phii1}
\end{align}
and the constraints $(\mathfrak{C}_{B_2})$ of Definition \ref{def:CB2} are satisfied.
\end{prop}

\proof
As the proof is similar to that of Proposition \ref{pro:cr1}, let us focus on the main differences. In this case, one obtains the following constraints:
\begin{subequations}
\begin{align}
& \tA_{i,\rr}\Phi_{i+1}^{1,+}  =
\zeta\tA_{i+1,\orr} \Phi^{1,+}_{i} \,, & & i\geq 0\,, \label{eq:B2cont1}\\
& \tA_{i,\rr}\Phi_{i+1}^{0,+} =
\zeta\tA_{i,\orr}\Phi^{0,+}_{i}+(\zeta\tB_{i+1,\orr}-\tB_{i,\rr}-\xi)\Phi^{1,+}_{i} 
  \,, & & i\geq 0\,, \label{eq:B2cont2}\\
&\tA_{i,\rr}\Phi_{i+1}^{-1,+}=
\zeta\tA_{i-1,\orr}\Phi_{i}^{-1,+}
+(\zeta\tB_{i,\orr}-\tB_{i,\rr}-\xi)\Phi^{0,+}_{i}
+\zeta\tC_{i+1,\orr}\Phi^{1,+}_{i}
-\tC_{i,\rr}\Phi^{1,+}_{i-1}\,, & & i\geq 0\,,\label{eq:B2cont3}\\
&0=(\zeta\tB_{i-1,\orr}-\tB_{i,\rr}-\xi)\Phi^{-1,+}_{i}+\zeta\tC_{i,\orr}\Phi^{0,+}_{i}
-\tC_{i,\rr}\Phi^{0,+}_{i-1}\,, & & i\geq 1\,, \label{eq:B2cont4} \\
&0=\zeta\tC_{i-1,\orr}\Phi^{-1,+}_{i}
-\tC_{i,\rr}\Phi^{-1,+}_{i-1}\label{eq:B2cont5}\,, & & i\geq 2\,.
\end{align}
\end{subequations}
Equations \eqref{eq:B2cont1} and \eqref{eq:B2cont5} lead to
\begin{equation}
\Phi_{i}^{1,+}  =\zeta^i
 \Phi_{0}^{1,+}\prod_{k=1}^{i} \frac{\tA_{k,\orr}}{\tA_{k-1,\rr}} \quad (i\geq 0)\,, \qquad \Phi_{i}^{-1,+}  =\zeta^{-i+1}
 \Phi_{1}^{-1,+}
 \prod_{k=1}^{i-1} \frac{\tC_{k+1,\rr}}{\tC_{k,\orr}} \quad (i\geq 1)\, .
\end{equation}
 Substituting these expressions of $\Phi_{i}^{1,+}$ and $\Phi_{i}^{-1,+}$ into \eqref{eq:B2cont3}, one arrives at \eqref{eq:B2Phii1}. Using all these results in equations \eqref{eq:B2cont2} and \eqref{eq:B2cont4}, one finds the following constraints:  
\begin{align}
\frac{\Phi_{1}^{-1,+}\tA_{0,\orr}}{\Phi_{0}^{1,+} \tC_{1,\rr}}&=\zeta^{2i-1}
    \frac{\frac{\zeta^2\tX_{i+1,\orr}-\tX_{i,\rr}}{\xi+\tB_{i,\rr}-\zeta\tB_{i,\orr}}- \frac{\zeta^2\tX_{i+2,\orr}-\tX_{i+1,\rr}}{\xi+\tB_{i+1,\rr}-\zeta\tB_{i+1,\orr}}+\zeta\tB_{i+1,\orr}-\tB_{i,\rr}-\xi}{\frac{\tX_{i+1,\rr}}{\zeta^2\tX_{i+1,\orr}}\frac{\zeta^2\tX_{i+1,\orr}-\tX_{i+2,\rr}}{\xi+\tB_{i+1,\rr}-\zeta\tB_{i+1,\orr}}-\frac{\zeta^2\tX_{i,\orr}-\tX_{i+1,\rr}}{\xi+\tB_{i,\rr}-\zeta\tB_{i,\orr}}}
    \prod_{k=1}^{i}  \frac{\tX_{k,\orr}}
    {\tX_{k,\rr}} \,,\quad i\geq 0\label{eq:B2indei1a}\\
    &=\zeta^{2i-3}\frac{\frac{\zeta^2\tX_{i,\orr}-\tX_{i-1,\rr}}{\xi+\tB_{i-1,\rr}-\zeta\tB_{i-1,\orr}} - \frac{\zeta^2\tX_{i,\orr}}{\tX_{i,\rr}}\frac{(\zeta^2\tX_{i+1,\orr}-\tX_{i,\rr})}{(\xi+\tB_{i,\rr}-\zeta\tB_{i,\orr})}}{\frac{\zeta^2\tX_{i,\orr}-\tX_{i+1,\rr}}{\xi+\tB_{i,\rr}-\zeta\tB_{i,\orr}}-\frac{\zeta^2\tX_{i-1,\orr}-\tX_{i,\rr}}{\xi+\tB_{i-1,\rr}-\zeta\tB_{i-1,\orr}}+\zeta\tB_{i-1,\orr}-\tB_{i,\rr}-\xi}
    \prod_{k=1}^{i-1}  \frac{\tX_{k,\orr}}
    {\tX_{k,\rr}} \,,\quad i\geq 1\,. \label{eq:B2indei2a}
\end{align}
The case $i=0$ of \eqref{eq:B2indei1a} leads to
\begin{align}
\Phi_{1}^{-1,+}=\frac{\frac{\zeta^2\tX_{1,\orr}}{\xi+\tB_{0,\rr}-\zeta\tB_{0,\orr}}- \frac{\zeta^2\tX_{2,\orr}-\tX_{1,\rr}}{\xi+\tB_{1,\rr}-\zeta\tB_{1,\orr}}+\zeta\tB_{1,\orr}-\tB_{0,\rr}-\xi}{\zeta\tA_{0,\rr}\tA_{0,\orr}\left(\frac{\zeta^2\tX_{1,\orr}-\tX_{2,\rr}}{\zeta^2\tX_{1,\orr}(\xi+\tB_{1,\rr}-\zeta\tB_{1,\orr})}+\frac{1}{\xi+\tB_{0,\rr}-\zeta\tB_{0,\orr}}\right)} \Phi_{0}^{1,+}\, .\label{eq:B2Phii22}
\end{align} 
Equations \eqref{eq:B2indei1a} and \eqref{eq:B2indei2a} give the constraints $(\mathfrak{C}_{B_2})$. One can fix $\Phi_{0}^{1,+}=1$, since it appears as a global normalization.
\endproof

\begin{prop}\label{pro:B2cr2}
The polynomials $R_{i}(x;\rr)$ chosen such that $\lambda_{x,\rr}=\zeta\lambda_{\ox,\orr}-\xi$ satisfy the contiguity relation of type $B_2$
\begin{equation}\label{eq:B2cons6a}
\lambda^-_{x,\rr} R_{i}(\ox;\orr)=\Phi_{i}^{-1,-}R_{i-1}(x;\rr) +\Phi_{i}^{0,-}R_{i}(x;\rr)+\Phi_{i}^{1,-}R_{i+1}(x;\rr)\,, \qquad i\geq 0\,, 
\end{equation}
(with the convention $\Phi_{0}^{-1,-} = 0$ and with $\Phi_0^{1,-}\neq 0$) if and only if the coefficients are given (up to a global normalization) by
\begin{align}
&\lambda^-_{x,\rr}=
 1+\frac{\tC_{1,\rr}}{\tA_{0,\rr}-\zeta\tA_{0,\orr}-\xi}+\frac{\lambda_{x,\rho}}{\tA_{0,\rr}}-\frac{\frac{\tX_{1,\rr}}{\zeta\tB_{0,\orr}-\tB_{0,\rr}-\xi}- \frac{\tX_{2,\rr}-\zeta^2\tX_{1,\orr}}{\zeta\tB_{1,\orr}-\tB_{1,\rr}-\xi}+\tB_{1,\rr}-\zeta\tB_{0,\orr}+\xi}{\frac{(\tA_{0,\rr}-\zeta\tA_{0,\orr}-\xi)(\tX_{1,\rr}-\zeta^2\tX_{2,\orr})}{\tC_{1,\rr}(\zeta\tB_{1,\orr}-\tB_{1,\rr}-\xi)}+\tA_{0,\rr}},\\
&  \Phi_{i}^{1,-}  = 
\zeta^{-i} \prod_{k=1}^{i} \frac{\tA_{k,\rr}}{\tA_{k-1,\orr}}\,, \qquad i\geq 0, \\
 &\Phi_{i}^{-1,-}  =\zeta^i
 \frac{\frac{\tX_{1,\rr}}{\zeta\tB_{0,\orr}-\tB_{0,\rr}-\xi}- \frac{\tX_{2,\rr}-\zeta^2\tX_{1,\orr}}{\zeta\tB_{1,\orr}-\tB_{1,\rr}-\xi}+\tB_{1,\rr}-\zeta\tB_{0,\orr}+\xi}{\tA_{0,\orr}\tA_{0,\rr}\left(\frac{\tX_{1,\rr}-\zeta^2\tX_{2,\orr}}{\tX_{1,\rr}(\zeta\tB_{1,\orr}-\tB_{1,\rr}-\xi)}+\frac{1}{\zeta\tB_{0,\orr}-\tB_{0,\rr}-\xi}\right)}
 \prod_{k=1}^{i-1} \frac{\tC_{k+1,\orr}}{\tC_{k,\rr}}\, , \quad i\geq 1, \\
 &\Phi_{i}^{0,-}  =
 \frac{\tX_{i+1,\rr}-\zeta^2\tX_{i,\orr}}{\tA_{i,\rr}(\zeta\tB_{i,\orr}-\tB_{i,\rr}-\xi)}\Phi^{1,-}_{i}
 +\frac{\tX_{i,\rr}-\zeta^2\tX_{i+1,\orr}}{\zeta\tC_{i+1,\orr}(\zeta\tB_{i,\orr}-\tB_{i,\rr}-\xi)}\Phi^{-1,-}_{i+1} , \qquad i\geq 0,
\end{align}
and the constraints $(\mathfrak{C}_{B_2})$ of Definition \ref{def:CB2} are satisfied.
\end{prop}
\begin{proof}
Similar to the proof of Proposition \ref{pro:B2cr1}.
\end{proof}

\begin{remax}
    Equation \eqref{eq:B2cons6a} can be obtained from equation \eqref{eq:B2cons1a} by exchanging $(x,\rr)$ with $(\ox,\orr)$, and performing the transformations $\zeta \to \zeta^{-1}$ and $\xi \to -\xi\zeta^{-1}$. The constraints $(\mathfrak{C}_{B_2})$ remain invariant under these transformations.
\end{remax}

\begin{remax}
\label{rmk:B2-from-A2}
From two relations of type ${A_2}$ given as follows
\begin{align}
\lambda^+_{x,\rr}R_{i}(x;\rr)=\Phi_{i}^{0,+}R_{i}(\ox;\orr) +\Phi_{i}^{-1,+}R_{i-1}(\ox;\orr)\,,\\
\widetilde \lambda^-_{\widetilde x,\widetilde \rr}R_{i}(\ox;\orr)=\widetilde\Phi_{i}^{1,-}R_{i+1}(\widetilde x;\widetilde\rr) +\widetilde\Phi_{i}^{0,-}R_{i}(\widetilde x;\widetilde \rr)\,,
\end{align}
one can obtain a relation of type $B_2$
\begin{align}
\widetilde \lambda^-_{\widetilde x,\widetilde \rr}\lambda^+_{x,\rr}R_{i}(x;\rr)=\Phi_{i}^{0,+}\widetilde\Phi_{i}^{1,-}R_{i+1}(\widetilde x;\widetilde\rr)
+(\Phi_{i}^{0,+}\widetilde\Phi_{i}^{0,-}+ \Phi_{i}^{-1,+}\widetilde\Phi_{i-1}^{1,-})R_{i}(\widetilde x;\widetilde\rr)  +\Phi_{i}^{-1,+}\widetilde \Phi_{i-1}^{0,-}R_{i-1}(\widetilde x;\widetilde\rr)\,.
\end{align}
This procedure is used in Appendix \ref{app:B2} to compute $B_2$-contiguity relations for the $q$-Racah polynomials.
Let us emphasize that if one chooses $\widetilde x=x$ and $\widetilde \rr=\rr$, then the previous relation has to be the recurrence relation for $R_{i}(x;\rr)$.
\end{remax}

%%%%%%%%%%%%%%%%%%%%%%%%%%%%%%%%%
%%%%%%%%%%%%%%%%%%%%%%%%%%%%%%%%%
%%%%%%%%%%%%%%%%%%%%%%%%%%%%%%%%%
%%%%%%%%%%%%%%%%%%%%%%%%%%%%%%%%%
\subsection{Contiguity relations of type $B'_2$ \label{sec:B2pcrr}}

Contiguity relations \eqref{eq:cons1} and \eqref{eq:cons3} of type $B'_2$ are studied in this subsection.

\begin{defi}\label{def:CB2p}
The constraints $(\mathfrak{C}_{B'_2})$ are defined by the requirement that the following expressions 
\begin{align}
   & 
   \zeta^{2i}\frac{\frac{\tX_{i,\rr}-\zeta^2\tX_{i,\orr}}{\zeta\tB_{i-1,\orr}-\tB_{i,\rr}-\xi}- \frac{\tX_{i+1,\rr}-\zeta^2\tX_{i+1,\orr}}{\zeta\tB_{i,\orr}-\tB_{i+1,\rr}-\xi}+\zeta\tB_{i,\orr}-\tB_{i,\rr}-\xi}{\tX_{i+1,\rr}\frac{\tX_{i+2,\rr}-\zeta^2\tX_{i,\orr}}{\zeta\tB_{i,\orr}-\tB_{i+1,\rr}-\xi}- \zeta^2\tX_{i,\orr}\frac{\tX_{i+1,\rr}-\zeta^2\tX_{i-1,\orr}}{\zeta\tB_{i-1,\orr}-\tB_{i,\rr}-\xi}}
    \prod_{k=1}^{i}  \frac{\tX_{k,\orr}}
    {\tX_{k,\rr}} \,, %\qquad i\geq 0,
    \label{eq:B2pindei1}\\
   &\zeta^{2i}\frac{\frac{\tX_{i+1,\rr}-\zeta^2\tX_{i+1,\orr}}{\zeta\tB_{i,\orr}-\tB_{i+1,\rr}-\xi} - \frac{\zeta^2\tX_{i+1,\orr}}{\tX_{i+2,\rr}}\frac{\tX_{i+2,\rr}-\zeta^2\tX_{i+2,\orr}}{\zeta\tB_{i+1,\orr}-\tB_{i+2,\rr}-\xi}}{\left(\frac{\tX_{i+3,\rr}-\zeta^2\tX_{i+1,\orr}}{\zeta\tB_{i+1,\orr}-\tB_{i+2,\rr}-\xi}- \frac{\tX_{i+2,\rr}-\zeta^2\tX_{i,\orr}}{\zeta\tB_{i,\orr}-\tB_{i+1,\rr}-\xi}+\zeta\tB_{i,\orr}-\tB_{i+2,\rr}-\xi\right)\tX_{i+1,\rr}}
    \prod_{k=1}^{i}  \frac{\tX_{k,\orr}}
    {\tX_{k,\rr}} \,, \label{eq:B2pindei2}
\end{align} 
are equal and independent of $i$ for $i \geq 0$, where the notations \eqref{eq:XY} have been used.
\end{defi}

\begin{prop}\label{pro:B2pcr1}
The polynomials $R_{i}(x;\rr)$ chosen such that $\lambda_{x,\rr}=\zeta\lambda_{\ox,\orr}-\xi$ satisfy the contiguity relation of type $B'_2$
\begin{equation}\label{eq:B2pcons1a}
\lambda^+_{x,\rr} R_{i}(x;\rr)=\Phi_{i}^{0,+}R_{i}(\ox;\orr) +\Phi_{i}^{-1,+}R_{i-1}(\ox;\orr)+\Phi_{i}^{-2,+}R_{i-2}(\ox;\orr)\,, \qquad i\geq 0\,, 
\end{equation}
(with the conventions $\Phi_{0}^{-2,+}=\Phi_{1}^{-2,+}=\Phi_{0}^{-1,+} = 0$, and with $\Phi_{0}^{0,+} \neq 0$ ) 
if and only if the coefficients are given by (up to a global normalization) 
\begin{align}
&\lambda^+_{x,\rr}=1\,,\\
&  \Phi_{i}^{0,+}  =\zeta^{i}
 \prod_{k=0}^{i-1} \frac{\tA_{k,\orr}}{\tA_{k,\rr}}\,, \qquad i\geq 0\,, \label{eq:B2pPhii0} \\
 &\Phi_{i}^{-2,+}  =\zeta^{-i+2}\frac{1}{\tA_{0,\rr}\tA_{1,\rr}}\left\{
 \left(\tB_{0,\rr}+\xi-\zeta\tB_{0,\orr}\right)\left(\tB_{1,\rr}+\xi-\zeta\tB_{0,\orr}\right)+\zeta^2\tX_{1,\orr}-\tX_{1,\rr}\right\}
 \prod_{k=1}^{i-2} \frac{\tC_{k+2,\rr}}{\tC_{k,\orr}}\, , \quad i\geq 2\,, \label{eq:B2pPhii2} \\
 &\Phi_{i}^{-1,+}  = \frac{(\tX_{i,\rr}-\zeta^2\tX_{i,\orr})}{\zeta\tA_{i-1,\orr}(\zeta\tB_{i-1,\orr}-\tB_{i,\rr}-\xi)}\Phi^{0,+}_{i} +\frac{(\tX_{i+1,\rr}-\zeta^2\tX_{i-1,\orr})}{\tC_{i+1,\rr}(\zeta\tB_{i-1,\orr}-\tB_{i,\rr}-\xi)}\Phi^{-2,+}_{i+1} \,, \qquad i\geq 1\,, \label{eq:B2pPhii1}
\end{align}
and the constraints $(\mathfrak{C}_{B'_2})$ of Definition \ref{def:CB2p} are satisfied.
\end{prop}

%\bleu{Contrainte supp pour x=0??? $\Phi^{0,1}_{0,j}=\Phi^{0,1}_{i,j}+\Phi^{-1,1}_{i,j}+\Phi^{-2,1}_{i,j}$. Pour $i=0,1$ ça ne donne rien de nouveau, mais pour $i\geq 2$ c'est plus difficile à vérifier (il faudrait utiliser les autres contraintes?). On peut se servir de cette relation pour exprimer $\Phi^{-1,1}_{i,j}$ plus simplement.}

\proof
Again, let us focus on the main differences. In this case, 
%using the recurrence relation \eqref{eq:recuR} satisfied by $R_i(x)$ and comparing the result with the contiguity recurrence relation \eqref{eq:B2cons1a} with $i$ replaced by $i+1$, 
one obtains the following constraints:
\begin{subequations}
\begin{align}
& \tA_{i,\rr}\Phi_{i+1}^{0,+}  =\zeta
\tA_{i,\orr} \Phi^{0,+}_{i} \,, & & i\geq 0 \,, \label{eq:B2pcont1}\\
 & \tA_{i,\rr}\Phi_{i+1}^{-1,+} =
\zeta\tA_{i-1,\orr}\Phi^{-1,+}_{i}+(\zeta\tB_{i,\orr}-\tB_{i,\rr}-\xi)\Phi^{0,+}_{i} 
  \,, & & i\geq 0\,, \label{eq:B2pcont2}\\
&\tA_{i,\rr}\Phi_{i+1}^{-2,+}=
\zeta\tA_{i-2,\orr}\Phi_{i}^{-2,+}
+(\zeta\tB_{i-1,\orr}-\tB_{i,\rr}-\xi)\Phi^{-1,+}_{i}
+\zeta\tC_{i,\orr}\Phi^{0,+}_{i}
-\tC_{i,\rr}\Phi^{0,+}_{i-1}\,, & & i\geq 1\,,\label{eq:B2pcont3}\\
&0=(\zeta\tB_{i-2,\orr}-\tB_{i,\rr}-\xi)\Phi^{-2,+}_{i}+\zeta\tC_{i-1,\orr}\Phi^{-1,+}_{i}
-\tC_{i,\rr}\Phi^{-1,+}_{i-1}\,, & & i\geq 2\,, \label{eq:B2pcont4} \\
&0=\zeta\tC_{i-2,\orr}\Phi^{-2,+}_{i}
-\tC_{i,\rr}\Phi^{-2,+}_{i-1}\label{eq:B2pcont5}\,, & & i\geq 3\,.
\end{align}
\end{subequations}
Equations \eqref{eq:B2pcont1} and \eqref{eq:B2pcont5} lead to
\begin{equation}
\Phi_{i}^{0,+}  =\zeta^i
 \Phi_{0}^{0,+}\prod_{k=0}^{i-1} \frac{\tA_{k,\orr}}{\tA_{k,\rr}} \quad (i\geq 0)\,, \qquad \Phi_{i}^{-2,+}  =\zeta^{-i+2}
 \Phi_{2}^{-2,+}
 \prod_{k=1}^{i-2} \frac{\tC_{k+2,\rr}}{\tC_{k,\orr}} \quad (i\geq 2)\,.
\end{equation}
Substituting these expressions of $\Phi_{i}^{0,+}$ and $\Phi_{i}^{-2,+}$ into \eqref{eq:B2pcont3}, one arrives at \eqref{eq:B2pPhii1}. Using all these results in equations \eqref{eq:B2pcont2} and \eqref{eq:B2pcont4}, one finds the following constraints:  
\begin{align}
\frac{\Phi_{2}^{-2,+}}{\Phi_{0}^{0,+} \tC_{1,\rr}\tC_{2,\rr}}&=
    \zeta^{2i}\frac{\frac{\tX_{i,\rr}-\zeta^2\tX_{i,\orr}}{\zeta\tB_{i-1,\orr}-\tB_{i,\rr}-\xi}- \frac{\tX_{i+1,\rr}-\zeta^2\tX_{i+1,\orr}}{\zeta\tB_{i,\orr}-\tB_{i+1,\rr}-\xi}+\zeta\tB_{i,\orr}-\tB_{i,\rr}-\xi}{\tX_{i+1,\rr}\frac{\tX_{i+2,\rr}-\zeta^2\tX_{i,\orr}}{\zeta\tB_{i,\orr}-\tB_{i+1,\rr}-\xi}- \zeta^2\tX_{i,\orr}\frac{\tX_{i+1,\rr}-\zeta^2\tX_{i-1,\orr}}{\zeta\tB_{i-1,\orr}-\tB_{i,\rr}-\xi}}
    \prod_{k=0}^{i-1}  \frac{\tX_{k+1,\orr}}
    {\tX_{k+1,\rr}} \,,\quad i\geq 0\label{eq:B2pindei1a}\\
    &=\zeta^{2i-2}\frac{\frac{1}{\zeta^2\tX_{i-1,\orr}}\frac{\tX_{i-1,\rr}-\zeta^2\tX_{i-1,\orr}}{\zeta\tB_{i-2,\orr}-\tB_{i-1,\rr}-\xi} - \frac{1}{\tX_{i,\rr}}\frac{\tX_{i,\rr}-\zeta^2\tX_{i,\orr}}{\zeta\tB_{i-1,\orr}-\tB_{i,\rr}-\xi}}{\frac{\tX_{i+1,\rr}-\zeta^2\tX_{i-1,\orr}}{\zeta\tB_{i-1,\orr}-\tB_{i,\rr}-\xi}- \frac{\tX_{i,\rr}-\zeta^2\tX_{i-2,\orr}}{\zeta\tB_{i-2,\orr}-\tB_{i-1,\rr}-\xi}+\zeta\tB_{i-2,\orr}-\tB_{i,\rr}-\xi}
    \prod_{k=0}^{i-2}  \frac{\tX_{k+1,\orr}}
    {\tX_{k+1,\rr}} \,,\quad i\geq 2\,. \label{eq:B2pindei2a}
\end{align}
The case $i=0$ of \eqref{eq:B2pindei1a} leads to
\begin{align}
&\Phi_{2}^{-2,+}  =\frac{\Phi_{0}^{0,+}}{\tA_{0,\rr}\tA_{1,\rr}}\left\{
 \left(\tB_{0,\rr}+\xi-\zeta\tB_{0,\orr}\right)\left(\tB_{1,\rr}+\xi-\zeta\tB_{0,\orr}\right)+\zeta^2\tX_{1,\orr}-\tX_{1,\rr}\right\}\, .\label{eq:B2pPhii22}
\end{align}
% Note that with these results one can compute  
% \begin{equation}
% \Phi_{1,j}^{-1,1}  =
%  \Phi^{0,1}_{0,j}\left(1-\frac{\tA_{0,\rr_{j+1}}}{\tA_{0,\rr_{j}}}\right).
% \end{equation}
% (An alternative for finding the previous relation is to use directly equation \eqref{eq:B2cont2} at $i=0$). 
Equations \eqref{eq:B2pindei1a} and \eqref{eq:B2pindei2a} yield the constraints $(\mathfrak{C}_{B'_2})$.
\endproof

\begin{prop}\label{pro:B2pcr2}
The polynomials $R_{i}(x;\rr)$ chosen such that $\lambda_{x,\rr}=\zeta\lambda_{\ox,\orr}-\xi$ satisfy the contiguity relation of type $B'_2$
\begin{equation}\label{eq:B2pcons6a}
\lambda^-_{x,\rr} R_{i}(\ox;\orr)=\Phi_{i}^{0,-}R_{i}(x;\rr) +\Phi_{i}^{1,-}R_{i+1}(x;\rr)+\Phi_{i}^{2,-}R_{i+2}(x;\rr)\,, \qquad i\geq 0\,, 
\end{equation}
(with $\Phi_{0}^{0,-} \neq 0$) if and only if the coefficients are given (up to a global normalization) by
\begin{align}
&  \lambda^-_{x,\rr}=\left(1+\frac{(\zeta\tB_{0,\orr}-\tB_{0,\rr}-\xi)({\lambda}_{x,\rho}-\tB_{0,\rr}-\xi)}{\tX_{1,\rr}}\right)  \\
& \ +\left(1-\frac{({\lambda}_{x,\rho}-\tB_{0,\rr}-\xi)({\lambda}_{x,\rho}-\tB_{1,\rr}-\xi)}{\tX_{1,\rr}}\right)\frac{\left\{\tX_{1,\rr}-\zeta^2\tX_{1,\orr}
 +\left(\zeta\tB_{0,\orr}-\tB_{0,\rr}-\xi\right)\left(\xi+\tB_{1,\rr}-\zeta\tB_{0,\orr}\right)\right\}}{\tX_{2,\rr}}\,,  \nonumber\\
&  \Phi_{i}^{0,-}  =\zeta^i
 \prod_{k=1}^{i} \frac{\tC_{k,\orr}}{\tC_{k,\rr}}\,, \qquad i\geq 0\,,  \\
 &\Phi_{i}^{2,-}  =\zeta^i
 \frac{\left\{
 \left(\tB_{0,\rr}-\zeta\tB_{0,\orr}+\xi\right)\left(\tB_{1,\rr}-\zeta\tB_{0,\orr}+\xi\right)+\zeta^2\tX_{1,\orr}-\tX_{1,\rr}\right\}}{\tC_{1,\rr}\tC_{2,\rr}}
 \prod_{n=0}^{i-1} \frac{\tA_{n+2,\rr}}{\tA_{n,\orr}}\, , \qquad i\geq 0\,, \\
 &\Phi_{i}^{1,-}  =
 \frac{\tX_{i+1,\rr}-\zeta^2\tX_{i+1,\orr}}{\tC_{i+1,\rr}(\zeta\tB_{i,\orr}-\tB_{i+1,\rr}-\xi)}\Phi^{0,-}_{i}
 +\frac{\tX_{i+2,\rr}-\zeta^2\tX_{i,\orr}}{\tA_{i+1,\rr}(\zeta\tB_{i,\orr}-\tB_{i+1,\rr}-\xi)}\Phi^{2,-}_{i} \,, \qquad i\geq 0\,, 
\end{align}
and the constraints $(\mathfrak{C}_{B'_2})$ of Definition \ref{def:CB2p} are satisfied.
\end{prop}
\begin{proof}
Similar to the proof of Proposition \ref{pro:B2pcr1}.
\end{proof}

\begin{remax}
\label{rmk:B2'-from-A2}
From two relations of type ${A_2}$ given as follows
\begin{align}
\lambda^+_{x,\rr}R_{i}(x;\rr)=\Phi_{i}^{0,+}R_{i}(\ox;\orr) +\Phi_{i}^{-1,+}R_{i-1}(\ox;\orr)\,,\\
\widetilde \lambda^+_{\ox,\orr}R_{i}(\ox;\orr)=\widetilde\Phi_{i}^{0,+}R_{i}(\widetilde x;\widetilde\rr) +\widetilde\Phi_{i}^{-1,+}R_{i-1}(\widetilde x;\widetilde \rr)\,,
\end{align}
one can obtain a relation of type $B'_2$
\begin{align}
\widetilde \lambda^+_{\ox,\orr}\lambda^+_{x,\rr}R_{i}(x;\rr)=\Phi_{i}^{0,+}\widetilde\Phi_{i}^{0,+}R_{i}(\widetilde x;\widetilde\rr)
+(\Phi_{i}^{0,+}\widetilde\Phi_{i}^{-1,+}+ \Phi_{i}^{-1,+}\widetilde\Phi_{i-1}^{0,+})R_{i-1}(\widetilde x;\widetilde\rr)  +\Phi_{i}^{-1,+}\widetilde \Phi_{i-1}^{-1,+}R_{i-2}(\widetilde x;\widetilde\rr)\,.
\end{align}
This procedure is used in Appendix \ref{app:B2p} to compute $B_2$-contiguity relations for the $q$-Racah polynomials.
\end{remax}

%%%%%%%%%%%%%%%%%%%%%%%%%%%%%%%%%
%%%%%%%%%%%%%%%%%%%%%%%%%%%%%%%%%
%%%%%%%%%%%%%%%%%%%%%%%%%%%%%%%%%
%%%%%%%%%%%%%%%%%%%%%%%%%%%%%%%%%

\section{Spectral transforms and orthogonal polynomials \label{sec:CG}}

The goal of this section is to demonstrate that the contiguity relations of type $A_2$ can be understood as spectral transforms of orthogonal polynomials, specifically of Christoffel and Geronimus types. These two transforms have the property that
the transformed polynomials have the same degree as the initial ones. The first type of spectral transform of orthogonal polynomials was introduced by Christoffel in the previous century, and the orthogonal polynomials resulting from this transformation are referred to as kernel polynomials \cite{Chiara}. The second type of spectral transform for orthogonal polynomials was thoroughly explored by Geronimus in 1940 \cite{Geronimus,Geronimus2}, and  was subsequently rediscovered in various contexts over the years. The application of a Geronimus transform followed by a Christoffel transform  is the identity transform, and the application of a  Christoffel transform followed by a Geronimus transform is the so-called Uvarov transform, see for instance \cite{zhedanov-spectral-transf}. 

We begin by introducing the contiguity relations for monic polynomials. The monic polynomials with respect to $\lambda_{x,\rr}$ are given by

\begin{equation}
    P_i(x;\rr)=\Gamma_{i,\rr}R_i(x;\rr)\,, \qquad Q_i(x;\rr)=\zeta^i\Gamma_{i,\orr}R_i(\ox;\orr)\,,
\end{equation} where $\Gamma_{i,\rr}=\prod_{k=0}^{i-1}A_{k,\rr}.$ Observe that to prove that $Q_i(x;\rr)$ is a monic polynomial in $\lambda_{x,\rr}$ we have used relation \eqref{eq:cond-lambda}. The monic three-term recurrence relations are given by 
\begin{align}
\lambda_{x,\rr}P_i(x;\rr)&=P_{i+1}(x;\rr)+Y_{i,\rr}P_i(x;\rr)+X_{i,\rr}P_{i-1}(x;\rr)\,,\label{eq:monic-3tr-P}\\
\lambda_{x,\rr}Q_i(x;\rr)&=Q_{i+1}(x;\rr)+(\zeta Y_{i,\orr}-\xi)Q_i(x;\rr)+\zeta^2 X_{i,\orr}Q_{i-1}(x;\rr)\, ,\label{eq:monic-3tr-Q}
\end{align}
with boundary conditions $P_{-1}=Q_{-1}=0,$ and $P_0=Q_0=1.$ The monic version of the $A_2$-contiguity relations \eqref{eq:cons1a} and \eqref{eq:cons3a} are 
\begin{align}
\omega_{x}\,Q_i(x;\rr)&=P_{i+1}(x;\rr)-a_iP_{i}(x;\rr)\,,\label{eq:contiguity-monic-2}\\
\,P_i(x;\rr)&=Q_i(x;\rr)-c_iQ_{i-1}(x;\rr)\,,\label{eq:contiguity-monic-1}
\end{align}
and the coefficients are given by
\begin{equation}
\label{eq:coefficients-contiguity-monic}
c_i=\zeta^{-2i+2}(\zeta A_{0,\orr}-A_{0,\rr}+\xi)\prod_{k=1}^{i-1}\frac{X_{k+1,\rr}}{X_{k,\orr}}\,,\qquad a_i=\zeta^{2i}\frac{X_{1,\rr}}{\zeta\tA_{0,\orr}-\tA_{0,\rr}+\xi}\prod_{k=1}^i\frac{X_{k,\orr}}{X_{k,\rr}}\,,
\end{equation}
\begin{equation}
    \omega_{x}=\lambda_{x,\rr}+A_{0,\rr}+\frac{X_{1,\rr}}{\tA_{0,\rr}-\zeta\tA_{0,\orr}-\xi}=\lambda_{x,\rr}-\lambda_{\nu,\rr}\,,
\end{equation}
where $\nu$ is such that 
\begin{equation}
  \lambda_{\nu,\rr}=-A_{0,\rr}-\frac{X_{1,\rr}}{\tA_{0,\rr}-\zeta\tA_{0,\orr}-\xi}\,.  
\end{equation}

\begin{remax}
Equation \eqref{eq:contiguity-monic-2} is a Christoffel transform of the polynomials $P_i(x;\rr)$ at the spectral parameter $\nu$ provided $a_i=\frac{P_{i+1}(\nu;\rr)}{P_{i}(\nu;\rr)}$. Similarly, equation \eqref{eq:contiguity-monic-1} is a Geronimus transform of the polynomials $Q_i(x;\rr)$ provided $c_i=\frac{\phi_i}{\phi_{i-1}}$, where $\phi_i$ is a solution of the recurrence relation \eqref{eq:monic-3tr-Q} with $x=\nu$, \textit{i.e.}, 
\begin{equation}
    \phi_i=F_i(\nu;\rr)+\chi Q_i(\nu;\rr)\,, \qquad F_i(x;\rr)=\int\frac{Q_i(y;\rr)d\sigma_Q(y)}{\lambda_x-\lambda_y}\,,
\end{equation}
where $d\sigma_Q(x)$ denotes the orthogonality measure for the polynomials $(Q_i)_i$ and $\chi$ is a parameter. Additionally, if we denote $d\sigma_P(x)$ the orthogonality measure for the polynomials $(P_i)_i$, one has 
\begin{equation}
\label{eq:measures-spectral-transformations}
    (\mu_1-\nu)d\sigma_Q(x)=(\lambda_{x,\rr}-\lambda_{\nu,\rr})d\sigma_P(x)\,,\qquad (\chi+F_0(\nu;\rr))d\sigma_P(x)=\frac{d\sigma_Q(x)}{\lambda_{\nu,\rr}-\lambda_{x,\rr}}+\chi \delta_{x-\nu}\,,
\end{equation}
where $\mu_1$ is the first moment of $d\sigma_P(x)$ and $\delta_{x-\nu}$ corresponds to adding a mass point at $\nu$.
\end{remax}

\begin{prop}
All contiguity relations of type $A_2$ are either Christoffel or Geronimus transforms. 
\end{prop}
\begin{proof}
From equation \eqref{eq:contiguity-monic-2} at $x=\nu$ one gets $a_i=\frac{P_{i+1}(\nu;\rr)}{P_i(\nu;\rr)}$, \textit{i.e.} the contiguity relation \eqref{eq:contiguity-monic-2} is the Christoffel transform of the polynomials $P_i(x;\rr)$ for the spectral parameter $\nu.$ Replacing \eqref{eq:contiguity-monic-1} into \eqref{eq:contiguity-monic-2}, one obtains
\begin{equation}
    \lambda_{x,\rr}Q_i(x;\rr)=Q_{i+1}(x;\rr)-(c_{i+1}+a_i-\lambda_{\nu,\rr})Q_i(x;\rr)+c_ia_iQ_{i-1}(x;\rr)\,.
\end{equation}
Comparing the coefficients in the equation above with those in \eqref{eq:monic-3tr-Q}, one gets
\begin{equation}
    -c_{i+1} -a_i+\lambda_{\nu,\rr}=\zeta Y_{i,\orr}-\xi\,,\qquad c_ia_i=\zeta^2X_{i,\orr}\,,
\end{equation}
which implies 
\begin{equation}
\label{eq:3tr-phi}
-c_{i+1} -\zeta^2\frac{X_{i,\orr}}{c_i}=\zeta Y_{i,\orr}-\xi-\lambda_{\nu,\rr}\,.
\end{equation}
Let $\phi_i$ be such that $c_i=\frac{\phi_i}{\phi_{i-1}}$. Equation \eqref{eq:3tr-phi} implies that $\phi_i$ satisfies 
\begin{equation}
    \lambda_{\nu,\rr}\phi_i=\phi_{i+1}+(\zeta Y_{i,\orr}-\xi)\phi_i+\zeta^2 X_{i,\orr}\phi_{i-1}\,,
\end{equation}
\textit{i.e.} $\phi_i$ is a solution of the recurrence relation defining $Q_i$. This allows one to conclude that \eqref{eq:contiguity-monic-1} is a Geronimus transform of the polynomials $Q_i$ for the spectral parameter $\nu.$ 
\end{proof}

\begin{example}
The monic Krawtchouk polynomials with parameters $\rr=\alpha,N$ are defined by
\begin{equation}
    xP_i(x;\rr)=P_{i+1}(x;\rr)+[\alpha(N-i)+i(1-\alpha)]P_i(x;\rr)+\alpha(1-\alpha)i(N+1-i)P_{i-1}(x;\rr)\,,
\end{equation}
with initial conditions $P_{-1}=0,\, P_0=1,$ and they are orthogonal with respect to the binomial distribution $$d\sigma(x;\alpha,N)=\sum_{x=0}^N \binom{N}{x}\alpha^x(1-\alpha)^{N-x}\delta_x\,.$$
The Geronimus and Christoffel transforms for the spectral parameter $\nu=N$ are respectively given by:
\begin{align}
P_i(x;\rr)&=P_i(x;\orr)-\alpha iP_{i-1}(x;\orr)\,,\\
(x-N)P_i(x;\orr)&=P_{i+1}(x;\rr)-(1-\alpha)(N-i)P_i(x;\rr)\,,
\end{align}
where $\orr=\alpha,N-1$, and one has
\begin{align}
(x-N)d\sigma(x;\rr)&=N(\alpha-1) d\sigma(x;\orr)\,,\\
\frac{\alpha^N}{N(1-\alpha)}\delta_{x-N}+\frac{d\sigma(x;\orr)}{N-x}&=-\frac{1}{N(\alpha-1)} d\sigma(x;\rr)\,.
\end{align}
\end{example}

%%%%%%%%%%%%%%%%%%%%%%%%%%%%%%%%%
%%%%%%%%%%%%%%%%%%%%%%%%%%%%%%%%%
%%%%%%%%%%%%%%%%%%%%%%%%%%%%%%%%%
%%%%%%%%%%%%%%%%%%%%%%%%%%%%%%%%%

\section{Solutions of the constraints  $(\mathfrak{C}_{A_2})$,  $(\mathfrak{C}_{B_2})$ and  $(\mathfrak{C}_{B'_2})$\label{sec:cons}}

In this section, we consider the twelve families of discrete finite orthogonal polynomials which also satisfy a three-term ($q$-)difference equation.
We denote them with a superscript $R^{(P)}_{i}(x;\rr)$ where the letters $P$ will be used to identify the different kinds of polynomials: $q$-Racah ($P=qR$), $q$-Hahn ($P=qH$), dual $q$-Hahn ($P=dqH$), quantum $q$-Krawtchouk ($P=qqK$), $q$-Krawtchouk ($P=qK$), 
dual $q$-Krawtchouk ($P=dqK$), affine $q$-Krawtchouk ($P=aqK$), Racah ($P=R$), Hahn ($P=H$), dual Hahn ($P=dH$), Krawtchouk ($P=K$) and Bannai--Ito ($P=BI$). All the quantities associated to these polynomials will also be identified by the superscript $(P)$.

We provide for each family 
the relations between the parameters $\rr$ and $\orr$ such that the relation 
\begin{equation}
\label{eq:lxr} 
\lambda_{x,\rr}=\zeta\lambda_{\ox,\orr}-\xi
\end{equation}
holds for $\ox=x+\eta$ where $\eta\in\{0,+1,-1\}$, 
and such that the constraints $(\mathfrak{C}_{A_2})$, $(\mathfrak{C}_{B_2})$ or $(\mathfrak{C}_{B'_2})$ are satisfied. To obtain this classification, we solve using formal mathematical software a few of these constraints $(\mathfrak{C})$ to find necessary conditions and, in general, $\oN$ is chosen among $\{N,N+1,N-1,N-2,N+2\}$. By direct computations, we prove that these conditions are sufficient in order for the constraints $(\mathfrak{C})$ to be satisfied. 

\subsection{\texorpdfstring{$q$}--Racah} \label{sec:A2qRacah}

The $q$-Racah polynomials are given by, for $i=0,1,\dots,N$,
\begin{align}\label{eq:qRacah}
R^{(qR)}_i(x;\rr)={}_4\phi_3 \left({{q^{-i},\; \alpha\beta q^{i+1}, \;q^{-x},\;\gamma  q^{x-N}}\atop
{\alpha q,\; \beta \gamma q,\;q^{-N} }}\;\Bigg\vert \; q;q\right)\,,
\end{align}
where $\rr=\alpha,\beta,\gamma,N,q$ and $N$ is a non-negative integer. In this paper, the usual definition of the generalized $q$-hypergeometric function is used, for $p=1,2,3,4$ and $i=0,1,2,\dots$:
\begin{align}
{}_{p+1}\phi_p \left({{q^{-i},\; a_1,\; \dots,\; a_p }\atop
{ b_1,\; \dots ,\; b_p}}\;\Bigg\vert \; q;z\right)=
\sum_{k=0}^i \frac{(q^{-i},a_1,\dots,a_p;q)_k}{(q,b_1,\dots,b_p;q)_k}  z^k\,,
\end{align}
where $a_1, \dots,a_p,b_1,\dots,b_p$ are parameters and
\begin{align}
(b_1,\dots,b_p;q)_k=(b_1;q)_k\dots (b_p;q)_k\,,\qquad (b_i;q)_k=\prod_{\ell=0}^{k-1}(1-b_i q^\ell)\;.
\end{align}
We recall that if there exists an integer $k$ such that $b_1\dots b_p=q^{k-i} a_1\dots a_p$ and $z=q$, the $q$-hypergeometric function is called $k$-balanced. We see that the $q$-Racah polynomials are $1$-balanced.

The $q$-Racah polynomials satisfy the orthogonality relation 
\begin{align}
    \sum_{x=0}^NR_i^{(qR)}(x;\rr)R_j^{(qR)}(x;\rr)w^{(qR)}(x;\rr)=\delta_{i,j}h_i\,,
\end{align}
where 
\begin{align}w^{(qR)}(x;\rr)=\frac{(\alpha q,\beta\gamma q,q^{-N},q^{-N}\gamma;q)_x (1-\gamma q^{2x-N})}{(q,\alpha^{-1}\gamma q^{-N},\beta^{-1} q^{-N},\gamma q;q)_x(\alpha\beta q)^x(1-\gamma q^{-N})}\,.
\end{align}
Knowing that $\lambda^{(qR)}_{x,\rr}=-(1-q^{-x})(1-\gamma q^{x-N})$, relation \eqref{eq:lxr} leads to
\begin{align}
&\zeta=q^\eta\,,\quad q^\ox=  q^{x+\eta}\,,\quad
    \ogamma=\gamma q^{\oN-N-2\eta}\,, \quad \xi=(1-q^\eta)(1 -\gamma q^{-\eta-N})\,,\label{eq:cons33}
\end{align}
where $\eta\in\{0,+1,-1\}$.

The coefficients of the recurrence relation read
\begin{subequations}
\begin{align}
 &   A^{(qR)}_{i,\rr}=\frac{(1-q^{i-N})(1-\alpha q^{i+1})(1-\alpha\beta q^{i+1})(1-\beta\gamma q^{i+1})}{(1-\alpha\beta q^{2i+1})(1-\alpha\beta q^{2i+2})}\,,\\
  &  C^{(qR)}_{i,\rr}=\frac{(1-q^{i})(1-\beta q^{i})(1-\alpha\beta q^{i+N+1})(\gamma-\alpha q^i)}{q^{N}(1-\alpha\beta q^{2i})(1-\alpha\beta q^{2i+1})}\,.
\end{align}
\end{subequations}
From these expressions, the constraints $(\mathfrak{C}_{A_2})$ for the $q$-Racah polynomials are solved and we provide the whole list of the possibilities up to transformations on the parameters leaving the polynomial invariant.
In this list, and in the following ones associated to the other polynomials, we do not give the trivial solution: 
$\eta=0\,,\quad \oalpha=\alpha\,,\quad \obeta=\beta\,,\quad \ogamma=\gamma\,,\quad \oN=N\,$ with  $\Phi_i^{0,+}=1,\ \Phi_i^{-1,+}=0, \    \lambda^{-}_{x,\rr}=1,\  \Phi_i^{0,-}=1,\ \Phi_i^{1,-}=0$.
We provide in each case the values of the coefficients, which may be rescaled, of the associated contiguity relation (with $\oN \text{ chosen in } \{N,N+1,N-1,N-2,N+2\}$):
\begin{itemize}
%     \item $\zeta=1\,,\quad \oalpha=\alpha\,,\quad \obeta=\beta\,,\quad \ogamma=\gamma\,,\quad \oN=N\,$      
 %    \begin{subequations}
  %        \begin{align}
   %    &  \Phi_i^{0,+}=1\,,\quad \Phi_i^{-1,+}=0\,,\\
          %\lambda^{-}_{\ox,\orr}=1\,,\quad &\Phi_i^{0,-}=1\,,\quad \Phi_i^{1,-}=0\,,
  %   \end{align}
   %  \end{subequations}
     \item[(qRI):] $\eta=0\,,\quad \oalpha=\alpha\,,\quad \obeta=q\beta\,,\quad \ogamma=\gamma/q\,,\quad \oN=N-1\,,\quad q^\ox=q^x$;
     \begin{subequations}
         \begin{align}
         &\lambda^{+}_{x,\rr}=1 \,,\quad \Phi_i^{0,+}=\frac{(1- q^{i-N})(1-\alpha\beta q^{i+1})}{(1-q^{-N})(1-\alpha\beta q^{2i+1})}\,,\quad \Phi_i^{-1,+}=\frac{(1-q^i)(1-\alpha \beta q^{N+i+1})}{(1-q^N)(1-\alpha\beta q^{2i+1})}\,,\nonumber\\
         & \lambda^{-}_{x,\rr}=\frac{(1-\gamma q^x)(1-q^{N-x})}{1-q^N}\,,\nonumber\\
         &\Phi_i^{0,-}=\frac{(1-\beta q^{i+1})(\alpha q^{i+1}-\gamma)}{1-\alpha \beta q^{2i+2}}\,,\quad \Phi_i^{1,-}=\frac{(1-\alpha q^{i+1})(1-\beta\gamma q^{i+1})}{1-\alpha \beta q^{2i+2}}\,,\nonumber
     \end{align}  
    \end{subequations}
%    \textcolor{red}{Christoffel parameter: $\nu=N,\quad  R_{i+1}(N;\rr)=\frac{\left(\alpha  \,q^{i+1}-\gamma \right) \left(\beta  \,q^{i+1}-1\right)}{\left(\beta  \gamma  \,q^{i+1}-1\right) \left(\alpha  \,q^{i+1}-1\right)} R_i(N;\rr)$ }
    \item[(qRII):] $\eta=0\,,\quad \oalpha=\alpha\,,\quad \obeta=q\beta\,,\quad \ogamma=\gamma\,,\quad \oN=N\,,\quad q^\ox=q^x;$
     \begin{subequations}
         \begin{align}
         &\lambda^{+}_{x,\rr}=1 \,,\quad \Phi_i^{0,+}=\frac{(1-\alpha\beta q^{i+1})(1-\beta\gamma q^{i+1})}{(1-\beta\gamma q)(1-\alpha\beta q^{2i+1})}\,,\quad \Phi_i^{-1,+}=-\frac{\beta q(1-q^i)(\gamma-\alpha q^i)}{(1-\beta\gamma q)(1-\alpha\beta q^{2i+1})}\,,\nonumber\\
          &\lambda^{-}_{x,\rr}=\frac{(1-\beta\gamma q^{x+1})(1-\beta q^{N+1-x})}{\beta q(1-\beta\gamma q)}\,,\nonumber\\
          &\Phi_i^{0,-}=\frac{(1-\beta q^{i+1})(1-\alpha\beta q^{N+i+2})}{\beta q(1-\alpha \beta q^{2i+2})}\,,\quad \Phi_i^{1,-}=\frac{(q^i- q^{N})(1-\alpha q^{i+1})}{ 1-\alpha \beta q^{2i+2}}\,,\nonumber
     \end{align}  
    \end{subequations}
%    \textcolor{red}{The Christoffel parameter must satisfy $\frac{1}{\beta}q^{-N-1}+\beta\gamma q=\gamma q^{\nu-N}+q^{-\nu}$ and $\frac{R_{i+1}(\nu;\rr)}{R_i(\nu;\rr)}=-\frac{\Phi_i^{0,-}}{\Phi_i^{1,-}}$. This conditions seem to hold both for $\nu=-\log_q(\beta\gamma)-1$ and $\nu=-\log_q(\frac{1}{\beta})+N+1$.}
    \item[(qRIII):] $\eta=-1\,,\quad \oalpha=q\alpha\,,\quad \obeta=\beta\,,\quad \ogamma=q \gamma\,,\quad \oN=N-1\,,\quad q^\ox=q^{x-1};$
     \begin{subequations}
         \begin{align}
         &\lambda^{+}_{x,\rr}=1 \,,\quad \Phi_i^{0,+}=\frac{(1-q^{N-i})(1-\alpha q^{i+1})(1-\alpha\beta q^{i+1})(1-\beta\gamma q^{i+1})}{(1-q^N)(1-\alpha q)(1-\beta\gamma q)(1-\alpha\beta q^{2i+1})}\,,\nonumber\\
         &\Phi_i^{-1,+}=\frac{(1-q^i)(\gamma q^{1-i}-\alpha q)(1-\beta q^i)(1-\alpha\beta q^{N+i+1})}{(1-q^N)(1-\alpha q)(1-\beta\gamma q)(1-\alpha\beta q^{2i+1})}\,,\nonumber\\
          &\lambda^{-}_{x,\rr}=\frac{(1- q^{-x})(1-\gamma q^{x-N})}{(1-\beta\gamma q)(1-\alpha q)(1-q^N)}\,,\quad
          \Phi_i^{0,-}=-\frac{q^{i-N}}{1-\alpha \beta q^{2i+2}}\,,\quad \Phi_i^{1,-}=\frac{q^{i-N}}{ 1-\alpha \beta q^{2i+2}}\,,\nonumber
     \end{align}  
    \end{subequations}
%    \textcolor{red}{Christoffel parameter: $\nu=0$,\quad $R_i(0,\rr)=1$}
      \item[(qRIV):] $\eta=0\,,\quad \oalpha=q\alpha\,,\quad \obeta=\beta\,,\quad \ogamma=\gamma\,,\quad \oN=N\,,\quad q^\ox=q^x;$
    \begin{subequations}
         \begin{align}
         &\lambda^{+}_{x,\rr}=1 \,,\quad \Phi_i^{0,+}=\frac{(1-\alpha q^{i+1})(1-\alpha\beta q^{i+1})}{(1-\alpha q)(1-\alpha\beta q^{2i+1})}\,,\quad \Phi_i^{-1,+}=-\frac{q\alpha(1-q^i)(1-\beta q^i)}{(1-\alpha q)(1-\alpha\beta q^{2i+1})}\,,\nonumber\\
        & \lambda^{-}_{x,\rr}=\frac{(1-\alpha q^{x+1})(\gamma-\alpha q^{N+1-x})}{\alpha q(1-\alpha q)}\,,\nonumber\\
           &\Phi_i^{0,-}=\frac{(1-\alpha\beta q^{N+i+2})(\gamma-\alpha q^{i+1})}{\alpha q(1-\alpha\beta q^{2i+2})}\,,\quad \Phi_i^{1,-}=-\frac{(q^N-q^{i})(1-\beta\gamma q^{i+1})}{1-\alpha\beta q^{2i+2}}\,.\nonumber
     \end{align}  
    \end{subequations}
%    \textcolor{red}{The Christoffel parameter must satisfy $\frac{\gamma}{\alpha}q^{-N-1}+\alpha q=\gamma q^{\nu-N}+q^{-\nu}$ and $\frac{R_{i+1}(\nu;\rr)}{R_i(\nu;\rr)}=-\frac{\Phi_i^{0,-}}{\Phi_i^{1,-}}$. This conditions seem to hold both for $\nu=-\log_q(\alpha)-1$ and $\nu=-\log_q(\frac{\gamma}{\alpha})+N+1$. }
\end{itemize}

We solved the constraints $\mathfrak{C}_{B_2}$ for the cases $\oN=N, N-1,N+1,N-2,N+2$. All contiguity relations of type $B_2$ for the $q$-Racah polynomials are derived by applying twice the contiguity relations of type $A_2$ as discussed in Remark \ref{rmk:B2-from-A2}. These relations are listed in Appendix \ref{app:B2}. Similarly, we proved that all the contiguity relations of type ${B'_2}$ for the $q$-Racah polynomials are obtained by the procedure described in Remark \ref{rmk:B2'-from-A2}. Appendix \ref{app:B2p} contains the complete list of these relations.

We conclude by giving for each of the cases above the corresponding parameters $\nu, \chi$ in the Christoffel and Geronimus transforms. 
\begin{itemize}
 \item[(qRI)] $\nu=N\,, \quad \chi=\frac{w^{(qR)}(N;\rr)}{(1-\gamma)(1-q^{-N})};$ 
%    $$ \sum_{x=0}^{N-1}\frac{w^{(qR)}(x,\orr)}{\lambda_{N,\rr}-\lambda_{x,\rr}}\delta_x +\frac{w^{(qR)}(N;\rr)}{(\gamma-1)(1-q^{-N})}\delta_{x-N}=\frac{1}{(\gamma-1)(1-q^{-N})}\sum_{x=0}^N w^{(qR)}(x,\rr)\delta_x\,.$$
 \item[(qRII)] $q^\nu=\beta q^{N+1}\,,\quad \chi=0;$ 
%  $$ \sum_{x=0}^{N}\frac{w^{(qR)}(x,\orr)}{\lambda_{\nu,\rr}-\lambda_{x,\rr}}\delta_x =\frac{1}{(\beta\gamma q-1)(1-\beta^{-1}q^{-N-1})}\sum_{x=0}^N w^{(qR)}(x,\rr)\delta_x\,.$$
 \item[(qRIII)] $\nu=0\,, \quad \chi=\frac{(1-\beta^{-1}q^{-N})(1-\gamma q)(1-\gamma q^{-N})(1-\gamma\alpha^{-1}q^{-N})\alpha\beta q^2 }{(1-\alpha q)(1-\beta\gamma q)(1-q^{-N})(1-\gamma q^{-N+1})(1-\gamma q^{-N+2})};$ 
%  $$ \sum_{x=1}^{N}\frac{w^{(qR)}(x-1,\orr)}{\lambda_{0,\rr}-\lambda_{x,\rr}}\delta_x +\chi\delta_{0}=\chi\sum_{x=0}^N w^{(qR)}(x,\rr)\delta_x\,.$$
 \item[(qRIV)] $q^\nu=\alpha^{-1}q^{-1}\,,\quad \chi=0;$
%  $$ \sum_{x=0}^{N}\frac{w^{(qR)}(x,\orr)}{\lambda_{\nu,\rr}-\lambda_{x,\rr}}\delta_x =\frac{1}{(\alpha q-1)(1-\alpha^{-1}\gamma q^{-N-1})}\sum_{x=0}^N w^{(qR)}(x,\rr)\delta_x\,.$$
\end{itemize}
For completeness, we verify that the relations  \eqref{eq:measures-spectral-transformations} connecting the measures indeed hold. For example:
\begin{itemize}
\item[(qRI)]
$ \sum_{x=0}^{N-1}\frac{w^{(qR)}(x,\orr)}{\lambda_{N,\rr}-\lambda_{x,\rr}}\delta_x +\frac{w^{(qR)}(N;\rr)}{(\gamma-1)(1-q^{-N})}\delta_{x-N}=\frac{1}{(\gamma-1)(1-q^{-N})}\sum_{x=0}^N w^{(qR)}(x,\rr)\delta_x\,.$
 \item[(qRII)] 
$ \sum_{x=0}^{N}\frac{w^{(qR)}(x,\orr)}{\lambda_{\nu,\rr}-\lambda_{x,\rr}}\delta_x =\frac{1}{(\beta\gamma q-1)(1-\beta^{-1}q^{-N-1})}\sum_{x=0}^N w^{(qR)}(x,\rr)\delta_x\,.$
\end{itemize}
The Christoffel and Geronimus parameters for the next families can be obtained from those of the $q$-Racah polynomials through limits.

\subsection{\texorpdfstring{$q$}--Hahn \label{sec:qHahn}}

The $q$-Hahn polynomials are given by
\begin{align}
  R^{(qH)}_i(x;\rr)={}_3\phi_2 \left({{q^{-i},\; \alpha\beta q^{i+1}, \;q^{-x}}\atop
{\alpha q,\; q^{-N} }}\;\Bigg\vert \; q;q\right)\,,
\end{align}
where $\rr=\alpha,\beta,N,q$. Knowing that $\lambda^{(qH)}_{x,\rr}=-(1-q^{-x})$, relation \eqref{eq:lxr} leads to
\begin{align}
\zeta=q^\eta\,,\quad q^\ox= q^{x+\eta}\,,\quad \xi=1-q^\eta\,,
\end{align}
where $\eta\in \{0,+1,-1\}$.
The coefficients of the recurrence relation read
\begin{subequations}
    \begin{align}
 &   A^{(qH)}_{i,\rr}=\frac{(1-q^{i-N})(1-\alpha q^{i+1})(1-\alpha \beta q^{i+1})}{(1-\alpha \beta q^{2i+1})(1-\alpha \beta q^{2i+2})}\,,\\
  &  C^{(qH)}_{i,\rr}=-\frac{\alpha q^{i-N}(1-q^{i})(1-\beta q^{i})(1-\alpha \beta q^{i+N+1})}{(1-\alpha \beta q^{2i})(1-\alpha \beta q^{2i+1})}\,.
\end{align}
\end{subequations}
The whole list of the contiguity relations of type $A_2$ (except the trivial one) for the $q$-Hahn polynomials is (with $\oN \text{ chosen in } \{N,N+1,N-1,N-2,N+2\}$):
\begin{itemize}
\item[(qHI)]  $\eta=0\,,\quad \oalpha=q\alpha\,,\quad \obeta=\beta\,,\quad \oN=N\,,\quad q^\ox=q^x;$
    \begin{subequations}
         \begin{align}
         &\lambda^{+}_{x,\rr}=1 \,,\quad \Phi_i^{0,+}=\frac{(1-\alpha q^{i+1})(1-\alpha\beta q^{i+1})}{(1-\alpha q)(1-\alpha\beta q^{2i+1})}\,,\quad \Phi_i^{-1,+}=-\frac{\alpha q(1-q^i)(1-\beta q^i)}{(1-\alpha q)(1-\alpha\beta q^{2i+1})}\,,\nonumber \\
        & \lambda^{-}_{x,\rr}=\frac{q^N(\alpha q-q^{-x})}{ 1-\alpha q}\,,\quad \Phi_i^{0,-}=-\frac{ q^{i}(1-\alpha\beta q^{N+i+2})}{1-\alpha\beta q^{2i+2}}\,,\quad \Phi_i^{1,-}=\frac{q^i-q^{N}}{1-\alpha\beta q^{2i+2}}\,,\nonumber 
     \end{align}  
    \end{subequations}
%    \textcolor{red}{Christoffel parameter:$\nu=-\log_q(\alpha)-1$}
    \item[(qHII)] $\eta= -1\,,\quad \oalpha=q\alpha\,,\quad \obeta=\beta\,,\quad \oN=N-1\,,\quad q^\ox=q^{x-1}$;
 \begin{subequations}
         \begin{align}
         &\lambda^{+}_{x,\rr}=1 \,,\quad \Phi_i^{0,+}=\frac{(1- q^{N-i})(1-\alpha\beta q^{i+1})(1-\alpha q^{i+1})}{(1-q^{N})(1-\alpha\beta q^{2i+1})(1-\alpha q)}\,,\nonumber \\
         &\Phi_i^{-1,+}=-\frac{\alpha q(1-q^i)(1-\alpha \beta q^{N+i+1})(1-\beta q^i)}{(1-q^N)(1-\alpha\beta q^{2i+1})(1-\alpha q)}\,,\nonumber \\
         & \lambda^{-}_{x,\rr}=\frac{1-q^{-x}}{(1-q^N)(1-\alpha q )}\,,\quad \Phi_i^{0,-}=-\frac{ q^{i-N}}{1-\alpha \beta q^{2i+2}}\,,\quad \Phi_i^{1,-}=\frac{q^{i-N}}{1-\alpha \beta q^{2i+2}}\,,\nonumber 
     \end{align}  
    \end{subequations}
 %   \textcolor{red}{Christoffel parameter: $\nu=0$}
    \item[(qHIII):] $\eta=0\,,\quad \oalpha=\alpha\,,\quad \obeta=q\beta\,,\quad \oN=N\,,\quad q^\ox=q^x;$
     \begin{subequations}
         \begin{align}
         &\lambda^{+}_{x,\rr}=1 \,,\quad \Phi_i^{0,+}=\frac{1-\alpha\beta q^{i+1}}{1-\alpha\beta q^{2i+1}}\,,\quad \Phi_i^{-1,+}=\frac{\alpha\beta q^{i+1}(1-q^i)}{1-\alpha\beta q^{2i+1}}\,,\nonumber \\
          &\lambda^{-}_{x,\rr}=\frac{1-\beta q^{N+1-x}}{\beta q}\,,\quad \Phi_i^{0,-}=\frac{(1-\beta q^{i+1})(1-\alpha\beta q^{N+i+2})}{\beta q(1-\alpha \beta q^{2i+2})}\,,\quad \Phi_i^{1,-}=\frac{(q^i- q^{N})(1-\alpha q^{i+1})}{ 1-\alpha \beta q^{2i+2}}\,,\nonumber 
     \end{align}  
    \end{subequations}
 %    \textcolor{red}{Christoffel parameter:$\nu=-\log_q(\frac{1}{\beta})+N+1$}
    \item[(qHIV):] $\eta=0\,,\quad \oalpha=\alpha\,,\quad \obeta=q\beta\,,\quad \oN=N-1\,,\quad q^\ox=q^x$;
     \begin{subequations}
         \begin{align}
         &\lambda^{+}_{x,\rr}=1 \,,\quad \Phi_i^{0,+}=\frac{(1- q^{i-N})(1-\alpha\beta q^{i+1})}{(1-q^{-N})(1-\alpha\beta q^{2i+1})}\,,\quad \Phi_i^{-1,+}=\frac{(1-q^i)(1-\alpha \beta q^{N+i+1})}{(1-q^N)(1-\alpha\beta q^{2i+1})}\,,\nonumber \\
         & \lambda^{-}_{x,\rr}=\frac{1-q^{N-x}}{1-q^N}\,,\quad \Phi_i^{0,-}=\frac{\alpha q^{i+1}(1-\beta q^{i+1})}{1-\alpha \beta q^{2i+2}}\,,\quad \Phi_i^{1,-}=\frac{1-\alpha q^{i+1}}{1-\alpha \beta q^{2i+2}}\,.\nonumber 
     \end{align}  
    \end{subequations}
 %    \textcolor{red}{Christoffel parameter:$\nu=N$}
\end{itemize}

As for the $q$-Racah case, we solved the constraints $\mathfrak{C}_{B_2}$ for the cases $\oN=N, N-1,N+1,N-2,N+2$. All contiguity relations of type $B_2$ for the $q$-Hahn polynomials are derived by applying twice the contiguity relations of type $A_2$ as discussed in Remark \ref{rmk:B2-from-A2}. Similarly, all the contiguity relations of type ${B'_2}$ for the $q$-Hahn polynomials are obtained by using twice contiguity relations of type ${A_2}$ as in Remark \ref{rmk:B2'-from-A2}.

The results in this subsection can be recovered from the $q$-Racah case by taking $\gamma = 0$, with the following correspondence: (qRI) $\to$ (qHIV), (qRII) $\to$ (qHIII), (qRIII) $\to$ (qHII) and (qRIV) $\to$ (qHI).

\subsection{Dual \texorpdfstring{$q$}--Hahn} 

The dual $q$-Hahn polynomials are given by
\begin{align}
R^{(dqH)}_i(x;\rr)={}_3\phi_2 \left({{q^{-i},\;q^{-x},\;\alpha\beta q^{x+1}}\atop
{\alpha q,\;q^{-N} }}\;\Bigg\vert \; q;q\right)\,,
\end{align}
where $\rr=\alpha,\beta,N,q.$ Knowing that $\lambda^{(dqH)}_{x,\rr}=-(1-q^{-x})(1-\alpha\beta q^{x+1})$, relation \eqref{eq:lxr} leads to
\begin{align}
\zeta=q^\eta\,,\quad q^\ox=  q^{x+\eta}\,,\quad
    \obeta=\frac{\alpha \beta}{\oalpha} q^{-2\eta} \,, \quad \xi=(1-q^\eta)(1-\alpha\beta q^{1-\eta})\,,
\end{align}
where $\eta\in \{0,+1,-1\}$.
The coefficients of the recurrence relation read
\begin{align}
    A^{(dqH)}_{i,\rr}=(1-q^{i-N})(1-\alpha q^{i+1})\,,\quad
    C^{(dqH)}_{i,\rr}=\alpha q(1-q^{i})(\beta-q^{i-N-1})\,.
\end{align}
The whole list of the contiguity relations of type $A_2$ (except the trivial one) for the dual $q$-Hahn polynomials is
\begin{itemize}
\item[(dqHI)] $\eta = -1\, \quad\oalpha=q\alpha\,,\quad\obeta=q \beta\,,\quad \oN=N-1\,,\quad q^\ox=q^{x-1}$;
 \begin{subequations}
         \begin{align}
         &\lambda^{+}_{x,\rr}=1 \,,\quad \Phi_i^{0,+}=\frac{(1-q^{N-i})(1-\alpha q^{i+1})}{(1-q^{N})(1-\alpha q)}\,,\quad \Phi_i^{-1,+}=-\frac{\alpha q(1-q^i)(1-\beta q^{N-i+1})}{(1-q^N)(1-\alpha q)}\,,\nonumber \\
         & \lambda^{-}_{x,\rr}=\frac{(1-q^{-x})(1-\alpha \beta q^{x+1})}{(1-\alpha q)(1-q^N)}\,,\quad \Phi_i^{0,-}= -q^{i-N}\,,\quad \Phi_i^{1,-}=q^{i-N}\,,\nonumber 
     \end{align}  
    \end{subequations}
%    \textcolor{red}{Christoffel parameter: $\nu=0$}
\item [(dqHII)]$\eta=0\,,\quad\oalpha=\alpha\,,\quad\obeta=\beta\,,\quad \oN=N-1\,,\quad q^\ox=q^x$;
\begin{subequations}
         \begin{align}
         &\lambda^{+}_{x,\rr}=1 \,,\quad \Phi_i^{0,+}=\frac{q^i-q^N}{1-q^N}\,,\quad \Phi_i^{-1,+}=\frac{1-q^i}{1- q^N}\,,\nonumber \\
         & \lambda^{-}_{x,\rr}=\frac{(1-q^{N-x})(1-\alpha\beta q^{x+N+1})}{1- q^N}\,,\quad \Phi_i^{0,-}=\alpha q(q^i-\beta q^N) \,,\quad \Phi_i^{1,-}=1-\alpha q^{i+1}\,,\nonumber 
     \end{align}  
    \end{subequations}
%\textcolor{red}{Christoffel parameter: $\nu=N$}
\item[(dqHIII)] $\eta=0\,,\quad\oalpha=q\alpha\,,\quad\obeta=\beta/q\,,\quad \oN=N\,,\quad q^\ox=q^x$;
 \begin{subequations}
         \begin{align}
         &\lambda^{+}_{x,\rr}=1 \,,\quad \Phi_i^{0,+}=\frac{1-\alpha q^{i+1}}{1-\alpha q}\,,\quad \Phi_i^{-1,+}=-\frac{\alpha q(1-q^i)}{1-\alpha q}\,,\nonumber \\
         & \lambda^{-}_{x,\rr}=\frac{(q^{-x}-\beta)(1-\alpha q^{x+1})}{1-\alpha q}\,,\quad \Phi_i^{0,-}= q^{i-N}-\beta\,,\quad \Phi_i^{1,-}=1- q^{i-N}\,.\nonumber 
     \end{align}  
    \end{subequations}
%\textcolor{red}{Christoffel parameter $\nu=\log_q(\frac{1}{\beta})$}
\end{itemize}

As for the $q$-Racah case, we solved the constraints $\mathfrak{C}_{B_2}$ for the cases $\oN=N, N-1,N+1,N-2,N+2$. All contiguity relations of type $B_2$ for the dual $q$-Hahn polynomials are derived by applying twice the contiguity relations of type $A_2$ as discussed in Remark \ref{rmk:B2-from-A2}. Similarly, all the contiguity relations of type ${B'_2}$ for the dual $q$-Hahn polynomials are obtained by using twice contiguity relations of type ${A_2}$ as in Remark \ref{rmk:B2'-from-A2}.

The results in this subsection can be recovered (up to global normalizations) from the $q$-Racah case by taking $\beta = 0$ and then replacing $\gamma\to \alpha\beta q^{N+1}$, with the following correspondence: (qRI) $\to$ (dqHII), (qRII) $\to$ trivial, (qRIII) $\to$ (dqHI) and (qRIV) $\to$ (dqHIII).

\subsection{Quantum \texorpdfstring{$q$}--Krawtchouk} 

The quantum $q$-Krawtchouk polynomials are given by
\begin{align}
R^{(qqK)}_i(x;\rr)={}_2\phi_1 \left({{q^{-i},\;q^{-x}}\atop
{q^{-N} }}\;\Bigg\vert \; q;\alpha q^{i+1}\right)\,,
\end{align}
where $\rr=\alpha,N,q.$ Knowing that $\lambda^{(qqK)}_{x,\rr}=-(1-q^{-x})$, relation \eqref{eq:lxr} leads to
\begin{align}
\zeta=q^\eta\,,\quad q^\ox= q^{x+\eta}\,,\quad \xi=1-q^\eta\,,
\end{align}
where $\eta\in \{0,+1,-1\}$.
The coefficients of the recurrence relation read
\begin{align}
    A^{(qqK)}_{i,\rr}=\frac{1-q^{i-N}}{\alpha q^{2i+1}}\,,\quad
    C^{(qqK)}_{i,\rr}=\frac{(1-q^{i})(1-\alpha q^i)}{\alpha q^{2i}}\,.
\end{align}
The whole list of the contiguity relations of type $A_2$ (except the trivial one) for the quantum $q$-Krawtchouk polynomials is
\begin{itemize}
\item[(qqKI)] $\eta=0\,,\quad\oalpha=q\alpha\,,\quad \oN=N\,,\quad q^\ox=q^x$;
\begin{subequations}
         \begin{align}
         &\lambda^{+}_{x,\rr}=1 \,,\quad \Phi_i^{0,+}=q^{-i}\,,\quad \Phi_i^{-1,+}=1-q^{-i}\,,\nonumber \\
         & \lambda^{-}_{x,\rr}=q^{-N}-\alpha q^{1-x}\,,\quad \Phi_i^{0,-}=q^{-i}-\alpha q \,,\quad \Phi_i^{1,-}=q^{-N}- q^{-i}\,,\nonumber 
     \end{align}  
    \end{subequations}
%    \textcolor{red}{Christoffel parameter: $\nu=-\log_q(\frac{1}{\alpha}q^{-N-1})$}
\item[(qqKII)]  $\eta=0\,,\quad\oalpha=q\alpha\,,\quad \oN=N-1\,,\quad q^\ox=q^x$;
\begin{subequations}
         \begin{align}
         &\lambda^{+}_{x,\rr}=1 \,,\quad \Phi_i^{0,+}=\frac{1-q^{N-i}}{1-q^N}\,,\quad \Phi_i^{-1,+}=\frac{1-q^{-i}}{1-q^{-N}}\,,\nonumber \\
         & \lambda^{-}_{x,\rr}=\frac{\alpha q(1-q^{N-x})}{1-q^N}\,,\quad \Phi_i^{0,-}=\alpha q-q^{-i} \,,\quad \Phi_i^{1,-}= q^{-i}\,,\nonumber 
     \end{align}  
    \end{subequations}
%\textcolor{red}{Christoffel parameter: $\nu=N$}
\item[(qqKIII)]  $\eta=-1\,,\quad\oalpha=\alpha\,,\quad \oN=N-1\,,\quad q^\ox=q^{x-1}$;
\begin{subequations}
         \begin{align}
         &\lambda^{+}_{x,\rr}=1-q^N \,,\quad \Phi_i^{0,+}=1-q^{N-i}\,,\quad \Phi_i^{-1,+}=-q^N(1-q^{-i})(1-\alpha q^{i})\,,\nonumber \\
         & \lambda^{-}_{x,\rr}=\alpha q^{N+1}\frac{1-q^{-x}}{1-q^N}\,,\quad \Phi_i^{0,-}=-q^{-i} \,,\quad \Phi_i^{1,-}= q^{-i}\,.\nonumber     \end{align}  
    \end{subequations}
%    \textcolor{red}{Christoffel parameter: $\nu=0$}
\end{itemize}

As for the $q$-Racah case, we solved the constraints $\mathfrak{C}_{B_2}$ for the cases $\oN=N, N-1,N+1,N-2,N+2$. All contiguity relations of type $B_2$ for the quantum $q$-Krawtchouk polynomials are derived by applying twice the contiguity relations of type $A_2$ as discussed in Remark \ref{rmk:B2-from-A2}. Similarly, all the contiguity relations of type ${B'_2}$ for the quantum $q$-Krawtchouk polynomials are obtained by using twice contiguity relations of type ${A_2}$ as in Remark \ref{rmk:B2'-from-A2}.

The results in this subsection can be recovered (up to global normalizations) from the $q$-Racah case by taking $\gamma=0$, taking the limit $\alpha \to \infty$ and then renaming $\beta\to \alpha$, with the following correspondence: (qRI) $\to$ (qqKII), (qRII) $\to$ (qqKI), (qRIII) $\to$ (qqKIII) and (qRIV) $\to$ trivial.

\subsection{\texorpdfstring{$q$}--Krawtchouk} 

The  $q$-Krawtchouk polynomials are given by
\begin{align}
R^{(qK)}_i(x;\rr)={}_3\phi_2 \left({{q^{-i},\;\alpha q^i},\;q^{-x}\atop
{q^{-N},\; 0}}\;\Bigg\vert \; q;q\right)\,,
\end{align}
where $\rr=\alpha,N,q.$ Knowing that $\lambda^{(qK)}_{x,\rr}=-(1-q^{-x})$, relation \eqref{eq:lxr} leads to
\begin{align}
\zeta=q^\eta\,,\quad q^\ox= q^{x+\eta}\,,\quad \xi=1-q^\eta\,,
\end{align}
where $\eta \in \{0,+1,-1\}$.
The coefficients of the recurrence relation read
\begin{align}
    A^{(qK)}_{i,\rr}=\frac{(1-q^{i-N})(1-\alpha q^i)}{(1-\alpha q^{2i})(1-\alpha q^{2i+1})}\,,\quad
    C^{(qK)}_{i,\rr}=\alpha q^{2i-N-1}\frac{(1-\alpha q^{i+N})(1-q^{i})}{(1-\alpha q^{2i-1})(1-\alpha q^{2i})}\,.
\end{align}
The whole list of the contiguity relations of type $A_2$ (except the trivial one) for the $q$-Krawtchouk polynomials is:
\begin{itemize}
\item[(qKI)] $\eta=0\,,\quad\oalpha=q\alpha\,,\quad \oN=N\,,\quad q^\ox=q^x$;
\begin{subequations}
         \begin{align}
         &\lambda^{+}_{x,\rr}=1 \,,\quad \Phi_i^{0,+}=\frac{1-\alpha q^{i}}{1-\alpha q^{2i}}\,,\quad \Phi_i^{-1,+}=\frac{\alpha q^i(1-q^i)}{1-\alpha q^{2i}}\,,\nonumber\\
         & \lambda^{-}_{x,\rr}=q^{N-x} \,,\quad \Phi_i^{0,-}=\frac{q^i(1-\alpha q^{N+i+1})}{1-\alpha q^{2i+1}} \,,\quad \Phi_i^{1,-}= \frac{q^N-q^i}{1-\alpha q^{2i+1}}\,,\nonumber
     \end{align}  
    \end{subequations}
%    \textcolor{red}{Christoffel parameter??? we sould have $q^{-\nu}=0$}
\item[(qKII)] $\eta=0\,,\quad\oalpha=q\alpha\,,\quad \oN=N-1\,,\quad q^\ox=q^x$;
\begin{subequations}
         \begin{align}
         &\lambda^{+}_{x,\rr}=1 \,,\quad \Phi_i^{0,+}=\frac{(1-q^{i-N})(1-\alpha q^i)}{(1-\alpha q^{2i})(1-q^{-N})}\,,\quad \Phi_i^{-1,+}=\frac{
        (1-q^i)(1-\alpha q^{N+i})}{(1-\alpha q^{2i})(1-q^N)}\,,\nonumber\\
         & \lambda^{-}_{x,\rr}=\frac{1-q^{N-x}}{1-q^N} \,,\quad \Phi_i^{0,-}=-\frac{\alpha q^{2i+1}}{1-\alpha q^{2i+1}} \,,\quad \Phi_i^{1,-}= \frac{1}{1-\alpha q^{2i+1}}\,,\nonumber
     \end{align}  
    \end{subequations}
%\textcolor{red}{Christoffel parameter: $\nu=N$}
\item[(qKIII)] $\eta=-1\,,\quad \oalpha=q\alpha\,,\quad \oN=N-1\,,\quad q^\ox=q^{x-1}; $
\begin{subequations}
         \begin{align}
         &\lambda^{+}_{x,\rr}=1 \,,\quad \Phi_i^{0,+}=\frac{(1-q^{N-i})(1-\alpha q^i)}{(1-\alpha q^{2i})(1-q^{N})}\,,\quad \Phi_i^{-1,+}=\alpha q^{i}\frac{(1-\alpha q^{i+N})(1-q^i)}{(1-\alpha q^{2i})(1-q^{N})}\,,\nonumber\\
         & \lambda^{-}_{x,\rr}=\frac{ 1-q^{-x}}{1-q^{N}}\,,\quad \Phi_i^{0,-}=-\frac{q^{i-N}}{1-\alpha q^{2i+1}}\,,\quad \Phi_i^{1,-}=\frac{q^{i-N}}{1-\alpha q^{2i+1}}\,.\nonumber
     \end{align}  
    \end{subequations}
%\textcolor{red}{Christoffel parameter: $\nu=0$}
\end{itemize}

As for the $q$-Racah case, we solved the constraints $\mathfrak{C}_{B_2}$ for the cases $\oN=N, N-1,N+1,N-2,N+2$. All contiguity relations of type $B_2$ for the $q$-Krawtchouk polynomials are derived by applying twice the contiguity relations of type $A_2$ as discussed in Remark \ref{rmk:B2-from-A2}. Similarly, all the contiguity relations of type ${B'_2}$ for the $q$-Krawtchouk polynomials are obtained by using twice contiguity relations of type ${A_2}$ as in Remark \ref{rmk:B2'-from-A2}.

The results in this subsection can be recovered (up to global normalizations) from the $q$-Racah case by taking $\gamma=0$, replacing $\beta \to q^{-1}\alpha^{-1}\beta$, taking the limit $\alpha \to 0$, and then renaming $\beta \to \alpha$, with the following correspondence: (qRI) $\to$ (qKII), (qRII) $\to$ (qKI), (qRIII) $\to$ (qKIII) and (qRIV) $\to$ (qKI).

\subsection{Affine \texorpdfstring{$q$}--Krawtchouk} 

The  affine $q$-Krawtchouk polynomials are given by
\begin{align}
R^{(aqK)}_i(x;\rr)={}_3\phi_2 \left({{q^{-i},\;0,\;q^{-x}}\atop
{\alpha q,\;q^{-N}}}\;\Bigg\vert \; q;q\right)\,,
\end{align}
where $\rr=\alpha,N,q.$ Knowing that $\lambda^{(aqK)}_{x,\rr}=-(1-q^{-x})$, \eqref{eq:cond-lambda} leads to
\begin{align}
\zeta=q^\eta\,,\quad q^\ox=q^{x+\eta}\,,\quad \xi=1-q^{\eta}\,,
\end{align}
where $\eta\in \{0,+1,-1\}$.
The coefficients of the recurrence relation read
\begin{align}
    A^{(aqK)}_{i,\rr}=(1-q^{i-N})(1-\alpha q^{i+1})\,,\quad
    C^{(aqK)}_{i,\rr}=-\alpha q^{i-N}(1-q^{i})\,.
\end{align}
The whole list of the contiguity relations of type $A_2$ (except the trivial one) for the affine $q$-Krawtchouk polynomials is:
\begin{itemize}
\item[(aqKI)] $\eta=0\,,\quad \oalpha=q\alpha\,,\quad \oN=N\,,\quad q^\ox=q^x$;
\begin{subequations}
         \begin{align}
         &\lambda^{+}_{x,\rr}=1 \,,\quad \Phi_i^{0,+}=\frac{1-\alpha q^{i+1}}{1-\alpha q}\,,\quad \Phi_i^{-1,+}=-\frac{\alpha q 
        (1-q^i)}{1-\alpha q}\,,\nonumber\\
         & \lambda^{-}_{x,\rr}=\frac{q^N(\alpha q-q^{-x})}{1-\alpha q} \,,\quad \Phi_i^{0,-}=-q^i \,,\quad \Phi_i^{1,-}=q^i-q^{N}\,,\nonumber
     \end{align}  
    \end{subequations}
%\textcolor{red}{Christoffel parameter $\nu=-\log_q(\alpha)-1$}
\item[(aqKII)] $\eta=0\,,\quad\oalpha=\alpha\,,\quad \oN=N-1\,,\quad q^\ox=q^x$;
\begin{subequations}
         \begin{align}
         &\lambda^{+}_{x,\rr}=1 \,,\quad \Phi_i^{0,+}=\frac{1- q^{i-N}}{1- q^{-N}}\,,\quad \Phi_i^{-1,+}=\frac{1-q^i}{1- q^N}\,,\nonumber\\
         & \lambda^{-}_{x,\rr}=\frac{1-q^{N-x}}{1-q^N} \,,\quad \Phi_i^{0,-}=\alpha q^{i+1} \,,\quad \Phi_i^{1,-}=1-\alpha q^{i+1}\,,\nonumber
     \end{align}  
    \end{subequations}
%\textcolor{red}{Christoffel parameter: $\nu=N$}
\item[(aqKIII)] $\eta=-1\,,\quad \oalpha=q\alpha\,,\quad \oN=N-1\,,\quad q^\ox=q^{x-1}$;
\begin{subequations}
         \begin{align}
         &\lambda^{+}_{x,\rr}=1 \,,\quad \Phi_i^{0,+}=\frac{(1-q^{N-i})(1-\alpha q^{i+1})}{(1-q^{N})(1-\alpha q)}\,,\quad \Phi_i^{-1,+}=-\frac{\alpha q (1-q^i)}{(1-q^N)(1-\alpha q)}\,,\nonumber\\
         & \lambda^{-}_{x,\rr}=\frac{1-q^{-x}}{(1-q^N)(1-\alpha q)},\quad
         \Phi_i^{0,-}=-q^{i-N}\,,\quad \Phi_i^{1,-}=q^{i-N} \,.\nonumber
     \end{align}  
    \end{subequations}
%\textcolor{red}{Christoffel parameter: $\nu=0$}
\end{itemize}

As for the $q$-Racah case, we solved the constraints $\mathfrak{C}_{B_2}$ for the cases $\oN=N, N-1,N+1,N-2,N+2$. All contiguity relations of type $B_2$ for the affine $q$-Krawtchouk polynomials are derived by applying twice the contiguity relations of type $A_2$ as discussed in Remark \ref{rmk:B2-from-A2}. Similarly, all the contiguity relations of type ${B'_2}$ for the affine $q$-Krawtchouk polynomials are obtained by using twice contiguity relations of type ${A_2}$ as in Remark \ref{rmk:B2'-from-A2}.

The results in this subsection can be recovered (up to global normalizations) from the $q$-Racah case by taking $\beta=\gamma=0$, with the following correspondence: (qRI) $\to$ (aqKII), (qRII) $\to$ trivial, (qRIII) $\to$ (aqKIII) and (qRIV) $\to$ (aqKI). 

%%%%%%%%%%%%%%%%%%%%%%%%%%%%%%%%%%%%%%%%%%%%%%%%%%%%%%%%
%%%%%%%%%%%%%%%%%%%%%%%%%%%%%%%%%%%%%%%%%%%%%%%%%%%%%%%%
%%%%%%%%%%%%%%%%%%%%%%%%%%%%%%%%%%%%%%%%%%%%%%%%%%%%%%%%
%%%%%%%%%%%%%%%%%%%%%%%%%%%%%%%%%%%%%%%%%%%%%%%%%%%%%%%%
\subsection{Dual \texorpdfstring{$q$}--Krawtchouk} 

The  dual $q$-Krawtchouk polynomials are given by
\begin{align}
R^{(dqK)}_i(x;\rr)={}_3\phi_2 \left({{q^{-i},\;q^{-x},\;\alpha q^{x}}\atop
{q^{-N},\; 0}}\;\Bigg\vert \; q;q\right)\,,
\end{align}
where $\rr=\alpha,N,q$. Knowing that $\lambda^{(dqK)}_{x,\rr}=-(1-q^{-x})(1-\alpha q^{x})$, relation \eqref{eq:cond-lambda} leads to \begin{align}
\zeta=q^\eta\,,\quad q^\ox=  q^{x+\eta}\,,\quad
    \oalpha=\alpha q^{-2\eta} \,, \quad \xi=(1-q^\eta)(1-\alpha q^{-\eta}) \,,
\end{align}
where $\eta\in \{0,+1,-1\}$.
The coefficients of the recurrence relation read
\begin{align}
  A^{(dqK)}_{i,\rr}=1-q^{i-N}\,,\quad
  C^{(dqK)}_{i,\rr}=\alpha(1-q^{i})\,.
\end{align}
The whole list of the contiguity relations of type $A_2$ (except the trivial one) for the dual $q$-Krawtchouk polynomials is:
\begin{itemize}
\item[(dqKI)]$\eta=0\,,\quad  \oalpha=\alpha\,,\quad \oN=N-1\,,\quad q^\ox=q^x$;
\begin{subequations}
         \begin{align}
         &\lambda^{+}_{x,\rr}=1 \,,\quad \Phi_i^{0,+}=\frac{1- q^{i-N}}{1- q^{-N}}\,,\quad \Phi_i^{-1,+}=\frac{1-q^i}{1- q^N}\,,\nonumber\\
         & \lambda^{-}_{x,\rr}=\frac{(1-q^{N-x})(1-\alpha q^{x+N})}{1-q^N} \,,\quad \Phi_i^{0,-}=-\alpha q^{N} \,,\quad \Phi_i^{1,-}=1\;,\nonumber
     \end{align}  
    \end{subequations}
%    \textcolor{red}{Christoffel paramenter: $\nu=N$ Monic version: 
 %   \begin{subequations}
 %         \begin{align*}
 %      & \lambda^{+}_{x,\rr}=1,\quad \Psi_i^{0,+}=1\,,\quad \Psi_i^{-1,+}=-q^{-N}(1-q^i)\,,\\
 %       &  \lambda^{-}_{x,\rr}=\frac{-(1-q^{N-x})(1-\alpha q^{x+N})}{q^N}=\lambda_{x,\rr}^{(dqK)}-\lambda_{N,\rr}^{(dqK)},\quad \Psi_i^{0,-}=\alpha  q^N(1-q^{i-N})
%\,,\quad \Psi_i^{1,-}=1\,,
%     \end{align*}
%     \end{subequations}
%   We verify: 
%   \begin{align*}
%       \Psi_i^{-1,+}\Psi_{i-1}^{0,-}&=A_{i-1,\rr}C_{i,\rr}\\
%       -\Psi_i^{-1,+}-\Psi_{i}^{0,-}+\lambda_{N,\rr}&=-A_{i,\rr}-C_{i,\rr}
%   \end{align*}
%   }
\item[(dqKII)] $\eta=-1\,,\quad \oalpha=q^2\alpha\,,\quad \oN=N-1\,,\quad q^\ox=q^{x-1};$
\begin{subequations}
         \begin{align}
         &\lambda^{+}_{x,\rr}=1 \,,\quad \Phi_i^{0,+}=\frac{1- q^{N-i}}{1- q^{N}}\,,\quad \Phi_i^{-1,+}=-\alpha q^{N+1} \frac{1-q^{-i}}{1-q^N}\,,\nonumber\\
         & \lambda^{-}_{x,\rr}=\frac{(1-q^{-x})(1-\alpha q^x)}{1-q^{-N}} \,,\quad \Phi_i^{0,-}=q^{i} \,,\quad \Phi_i^{1,-}=-q^{i}.\nonumber
     \end{align}  
    \end{subequations}
%\textcolor{red}{Christoffel parameter: $\nu=0$}
\end{itemize}

As for the $q$-Racah case, we proved that the contiguity relations of type $B_2$ for the dual $q$-Krawtchouk  polynomials are obtained using twice contiguity relations of type $A_2$ as in Remark \ref{rmk:B2-from-A2}. Similarly,  all the contiguity relations of type ${B'_2}$ for the dual $q$-Krawtchouk polynomials are obtained by using twice contiguity relations of type ${A_2}$ as in Remark \ref{rmk:B2'-from-A2}.

The results in this subsection can be recovered (up to global normalizations) from the $q$-Racah case by taking $\alpha=\beta=0$ and then replacing $\gamma\to\alpha q^N$, with the following correspondence: (qRI) $\to$ (dqKI), (qRII) $\to$ trivial, (qRIII) $\to$ (dqKII) and (qRIV) $\to$ trivial.

%%%%%%%%%%%%%%%%%%%%%%%%%%%%%%%%%%%%%%%%%%%%%%%%%%%%%%%%
%%%%%%%%%%%%%%%%%%%%%%%%%%%%%%%%%%%%%%%%%%%%%%%%%%%%%%%%
%%%%%%%%%%%%%%%%%%%%%%%%%%%%%%%%%%%%%%%%%%%%%%%%%%%%%%%%
%%%%%%%%%%%%%%%%%%%%%%%%%%%%%%%%%%%%%%%%%%%%%%%%%%%%%%%%
\subsection{Racah} 

The Racah polynomials are given by
\begin{align}
    R^{(R)}_i(x;\rr)={}_4F_3 \left({{-i,\;i+\alpha+\beta+1, \;-x,\;x+\gamma-N}\atop
{\alpha+1,\; \beta+\gamma+1,\;-N}}\;\Bigg\vert \; 1\right)\,,
\end{align}
where $\rr=\alpha,\beta,\gamma,N.$ In this paper, the usual definition of the generalized hypergeometric function is used, for $p=1,2,3,4$ and $i=0,1,2,\dots$,
\begin{align}
{}_{p+1}F_p \left({{-i,\; a_1,\; \dots,\; a_p }\atop
{ b_1,\; \dots ,\; b_p}}\;\Bigg\vert z\right)=
\sum_{k=0}^i \frac{(-i,a_1,\dots,a_p)_k\;z^k}{(b_1,\dots,b_p)_k\; k!}  \,,
\end{align}
where $a_1, \dots,a_p,b_1,\dots,b_p$ are parameters and
\begin{align}
(b_1,\dots,b_p)_k=(b_1)_k\dots (b_p)_k\,,\qquad (b_i)_k=\prod_{\ell=0}^{k-1}(b_i +\ell)\;.
\end{align}
The Racah polynomials satisfy the orthogonality relation 
\begin{align}
\sum_{x=0}^NR_i^{(R)}(x;\rr)R_j^{(R)}(x;\rr)w^{(R)}(x;\rr)=\delta_{i,j}h_i\,,\end{align}
where  
\begin{align}w^{(R)}(x;\rr)=\frac{(\alpha+1)_x(\beta+\gamma+1)_x(-N)_x(\gamma-N)_x\left(\frac{-N+\gamma+2}{2}\right)_x}{(-\alpha-N+\gamma)_x(-\beta-N)_x\left(\frac{\gamma-N}{2}\right)_x(\gamma+1)_xx!}\,.
\end{align}
Knowing that  $\lambda^{(R)}_{x,\rr}=x(x+\gamma-N)$, relation \eqref{eq:lxr} leads to
\begin{align}
 \zeta=1\,,\quad \quad \xi=\eta\,(\gamma-N-\eta)\,,\quad    2\eta=\gamma-\ogamma+\oN-N\,,
\end{align}
where $\eta\in\{0,+1,-1\}$. The coefficients of the recurrence relation read
\begin{subequations}
  \begin{align}
    &A^{(R)}_{i,\rr}=\frac{(i-N)(i+\alpha+1)(i+\alpha+\beta+1)(i+\beta+\gamma+1)}{(2i+\alpha+\beta+1)(2i+\alpha+\beta+2)}\,,\\
    &C^{(R)}_{i,\rr}=
    \frac{i(i+\alpha-\gamma)(i+\alpha+\beta+N+1)(i+\beta)}{(2i+\alpha+\beta)(2i+\alpha+\beta+1)}\,.
\end{align}  
\end{subequations}
From these expressions, the constraints $(\mathfrak{C}_{A_2})$ for the Racah polynomials are solved and we provide the whole list of the possibilities (except the trivial one) up to transformations on the parameters leaving the polynomial invariant. We also provide the values of the coefficients, which may be rescaled, of the associated contiguity relation:
\begin{itemize}
 %    \item[(RI)] $\eta=0, \quad \zeta=1, \quad\oalpha=\alpha\,,\quad \obeta=\beta\,,\quad \ogamma=\gamma\,,\quad \oN=N\,;$
 %      \begin{subequations}
 %         \begin{align*}
 %      & \lambda^{+}_{x,\rr}=?,\quad \Phi_i^{0,+}=?\,,\quad \Phi_i^{-1,+}=?\,,\\
 %       &  \lambda^{-}_{x,\rr}=?\,,\quad \Phi_i^{0,-}=?\,,\quad \Phi_i^{1,-}=?\,,
 %    \end{align*}
 %    \end{subequations}
   \item[(RI)] $\eta=0\,, \quad\oalpha=\alpha\,,\quad \obeta=\beta+1\,,\quad \ogamma=\gamma-1\,,\quad \oN=N-1\,, \quad \ox= x$;
     \begin{subequations}
          \begin{align*}
       & \lambda^{+}_{x,\rr}=1\,,\quad \Phi_i^{0,+}=\frac{\left(i +1+\alpha +\beta \right) \left(N-i \right)}{\left(2 i +1+\alpha +\beta \right) N}\,,\quad \Phi_i^{-1,+}=\frac{\left(i +1+\alpha +\beta +N \right) i}{\left(2 i +1+\alpha +\beta \right) N}\,,\\
        &  \lambda^{-}_{x,\rr}=\frac{\left(x +\gamma \right) \left(x -N \right)}{ N}
\,,\quad\Phi_i^{0,-}=\frac{ \left(i +1+\beta \right) \left(i +1+\alpha -\gamma \right)}{ 2 i +2+\alpha +\beta }\,,\\&\Phi_i^{1,-}=-\frac{\left(i +\beta +\gamma +1\right) \left(i +\alpha +1\right) }{ 2 i +2+\alpha +\beta}\,.
     \end{align*}
     \end{subequations}

      \item[(RII)] $\eta=0\,, \quad\oalpha=\alpha\,,\quad \obeta=\beta+1\,,\quad \ogamma=\gamma\,,\quad \oN=N\,, \quad \ox= x;$
     \begin{subequations}
          \begin{align*}
       & \lambda^{+}_{x,\rr}=1\,,\quad \Phi_i^{0,+}=\frac{\left(i +1+\alpha +\beta \right) \left(i+1 +\beta +\gamma\right)}{\left(2 i +1+\alpha +\beta \right) \left(1+\beta +\gamma \right)}\,,\quad \Phi_i^{-1,+}=-\frac{i \left(i +\alpha -\gamma \right)}{\left(2 i +1+\alpha +\beta \right) \left(1+\beta +\gamma \right)}\,,\\
        &  \lambda^{-}_{x,\rr}=\frac{\left(x +\beta +\gamma +1\right) \left(N +\beta -x +1\right)}{1+\beta +\gamma},\quad\Phi_i^{0,-}=\frac{ \left(i+2 +\alpha+\beta +N \right) \left(i +1+\beta \right)}{2 i+2 +\alpha +\beta}\,,\\& \Phi_i^{1,-}=\frac{ \left(N-i \right) \left(i +1+\alpha \right)}{2 i+2 +\alpha +\beta }\,,
     \end{align*}
     \end{subequations}
 
         \item[(RIII)] $\eta=-1\,, \quad\oalpha=\alpha+1\,,\quad \obeta=\beta\,,\quad \ogamma=\gamma+1\,,\quad \oN=N-1\,, \quad \ox= x-1;$
     \begin{subequations}
          \begin{align*}
       & \lambda^{+}_{x,\rr}=1\,,\quad \Phi_i^{0,+}=\frac{\left(i +1+\alpha +\beta \right) \left(i +1+\alpha \right) \left(N-i \right) \left(i+1 +\beta +\gamma \right)}{\left(2 i +1+\alpha +\beta \right) \left(1+\alpha \right) \left(1+\beta +\gamma \right) N}\,,\\
       &\Phi_i^{-1,+}=\frac{i \left(i +\beta \right) \left(i +\alpha -\gamma \right) \left(i +1+\alpha +\beta +N \right)}{\left(2 i +1+\alpha +\beta \right) \left(1+\alpha \right) \left(1+\beta +\gamma \right) N}\,,\\
        &  \lambda^{-}_{x,\rr}= \frac{x   \left(x +\gamma -N\right)}{ \left(1+\alpha \right) \left(1+\beta +\gamma \right) N}\,,\quad \Phi_i^{0,-}=\frac{1 }{2 i +2+\alpha +\beta}\,,\quad\Phi_i^{1,-}=-\frac{1 }{2 i +2+\alpha +\beta}\,,
     \end{align*}
     \end{subequations}
   
    \item[(RIV)] $\eta=0\,,  \quad\oalpha=\alpha+1\,,\quad \obeta=\beta\,,\quad \ogamma=\gamma\,,\quad \oN=N\,, \quad \ox= x;$
     \begin{subequations}
          \begin{align*}
       & \lambda^{+}_{x,\rr}=1\,,\quad \Phi_i^{0,+}=\frac{(i+1+\alpha+\beta)(i+1+\alpha)}{(1+\alpha)(2i+1+\alpha+\beta)}\,,\quad \Phi_i^{-1,+}=-\frac{i(i+\beta)}{(1+\alpha)(2i+1+\alpha+\beta)}\,,\\
        &  \lambda^{-}_{x,\rr}=\frac{\left(x+1+\alpha \right)  \left(N +\alpha -\gamma -x +1\right)}{1+\alpha}\,,\quad\Phi_i^{0,-}=\frac{ \left(i+2 +\alpha +\beta +N \right) \left(i +1+\alpha-\gamma \right)}{2 i+2 +\alpha +\beta }\,,\\
        & \Phi_i^{1,-}=\frac{ \left(i+1 +\beta +\gamma \right) \left(N-i \right)}{2 i +2+\alpha +\beta }\,.
     \end{align*}
     \end{subequations}

\end{itemize}
We recover the classification presented in the appendix of \cite{Oste2015Doubling}. We have used the same labeling of the solutions as in  \cite{Oste2015Doubling}. 
%As for the $q$-Racah case, we do not need to consider the four last possibilities for $\eta=0$ since they lead exactly to the same result than the one obtained with one of the first four possibilities.

The constraints $(\mathfrak{C}_{B_2})$ for the Racah polynomials have been solved for the cases where $\oN=N, N+1,N-1,N-2,N+2$. In these cases, all the contiguity relations of type $B_2$ for the Racah polynomials are derived by applying the contiguity relations of type $A_2$ twice, as discussed in Remark \ref{rmk:B2-from-A2}. Similarly, we proved that all the contiguity relations of type ${B'_2}$ for the Racah polynomials are obtained by using twice contiguity relations of type ${A_2}$ as in Remark \ref{rmk:B2'-from-A2}.

The results in this subsection can be recovered (up to global normalizations) from the $q$-Racah case by reparametrizing $\alpha \to q^\alpha$, $\beta \to q^\beta$, $\gamma \to q^\gamma$ and taking the limit $q\to 1$ (it is necessary to renormalize some coefficients in order for the limit to be well-defined). The correspondence between the solutions preserves the numbering.

We conclude by giving in each case the corresponding parameters $\nu, \chi$ in the Christoffel and Geronimus transforms. 
\begin{itemize}
 \item[(RI)] $\nu=N\,,\quad \chi=\frac{w^{(R)}(N;\rr)}{N\gamma};$ 
%   $$ \sum_{x=0}^{N-1}\frac{w^{(R)}(x,\orr)}{\lambda_{N,\rr}-\lambda_{x,\rr}}\delta_x + \frac{w^{(R)}(N,\rr)}{N\gamma}\delta_{x-N}=\frac{1}{N\gamma}\sum_{x=0}^N w^{(R)}(x,\rr)\delta_x\,.$$

 \item[(RII)]  $\nu=\beta+N+1\,,\quad \chi=0;$ 
%  $$ \sum_{x=0}^{N}\frac{w^{(R)}(x,\orr)}{\lambda_{\beta+N+1,\rr}-\lambda_{x,\rr}}\delta_x =\frac{1}{\left(\beta +N +1\right) \left(\beta +\gamma +1\right)}\sum_{x=0}^N w^{(R)}(x,\rr)\delta_x\,.$$
 \item[(RIII)] $\nu=0\,,\quad \chi=\frac{(\gamma+1)(\beta+N)(\alpha+N-\gamma)}{(\beta+\gamma+1)(\alpha+1)(N-2-\gamma)(N-\gamma-1)N};$ 
%   $$ \sum_{x=1}^{N}\frac{w^{(R)}(x-1,\orr)}{\lambda_{0,\rr}-\lambda_{x,\rr}}\delta_x +\chi\delta_{0}=\chi\sum_{x=0}^N w^{(R)}(x,\rr)\delta_x\,.$$
 \item[(RIV)] $\nu=-\alpha-1\,, \quad \chi=0;$ 
%  $$ \sum_{x=0}^{N}\frac{w^{(R)}(x,\orr)}{\lambda_{\beta+N+1,\rr}-\lambda_{x,\rr}}\delta_x =\frac{1}{\left(\alpha +1\right) \left(\alpha +1-\gamma +N \right)} \sum_{x=0}^N w^{(R)}(x,\rr)\delta_x\,.$$
\end{itemize}
For completeness, we verify that the relations \eqref{eq:measures-spectral-transformations} involving the measures hold. For example:
\begin{itemize}
\item[(RI)]
$\sum_{x=0}^{N-1}\frac{w^{(R)}(x,\orr)}{\lambda_{N,\rr}-\lambda_{x,\rr}}\delta_x + \frac{w^{(R)}(N,\rr)}{N\gamma}\delta_{x-N}=\frac{1}{N\gamma}\sum_{x=0}^N w^{(R)}(x,\rr)\delta_x\,.$
 \item[(RII)] 
$ \sum_{x=0}^{N}\frac{w^{(R)}(x,\orr)}{\lambda_{\beta+N+1,\rr}-\lambda_{x,\rr}}\delta_x =\frac{1}{\left(\beta +N +1\right) \left(\beta +\gamma +1\right)}\sum_{x=0}^N w^{(R)}(x,\rr)\delta_x\,.$
\end{itemize}

\subsection{Hahn}

The Hahn polynomials are given by
\begin{align}
    R^{(H)}_i(x;\rr)={}_3F_2 \left({{-i,\;i+\alpha+\beta+1, \;-x}\atop
{\alpha+1,\;-N}}\;\Bigg\vert \; 1\right)\,,
\end{align}
where $\rr=\alpha,\beta,N$. Knowing that $\lambda^{(H)}_{x,\rr}=-x$, relation \eqref{eq:lxr} leads to 
$$\zeta=1\,, \quad \xi=-\eta \in\{0,+1,-1\}\,.$$
The coefficients of the recurrence relation read
\begin{align}
    A^{(H)}_{i,\rr}=\frac{(i+\alpha+\beta+1)(i+\alpha+1)(N-i)}{(2i+\alpha+\beta+1)(2i+\alpha+\beta+2)}\,,\quad
    C^{(H)}_{i,\rr}=\frac{i(i+\alpha+\beta+N+1)(i+\beta)}{(2i+\alpha+\beta)(2i+\alpha+\beta+1)}\,.
\end{align}
The whole list of the contiguity relations of type $A_2$ (except the trivial one) for the Hahn polynomials is
\begin{itemize}
     %\item[(HI)] $\xi=0,\quad\oalpha=\alpha\,,\quad \obeta=\beta\,,\quad \oN=N\,;$
      %\begin{subequations}
       %   \begin{align*}
       %& \lambda^{+}_{x,\rr}=?,\quad \Phi_i^{0,+}=?\,,\quad \Phi_i^{-1,+}=?\,,\\
       % &  \lambda^{-}_{x,\orr}=?\,,\quad \Phi_i^{0,-}=?\,,\quad \Phi_i^{1,-}=?\,,
     %\end{align*}
     %\end{subequations}
    \item[(HI)] $\eta=0 \,, \quad \oalpha=\alpha+1\,,\quad \obeta=\beta\,,\quad \oN=N\,, \quad \ox= x;$
     \begin{subequations}
          \begin{align*}
       & \lambda^{+}_{x,\rr}=1\,,\quad \Phi_i^{0,+}=\frac{\left(i +1+\alpha \right) \left(i +1+\alpha +\beta \right)}{(1+\alpha)(2 i +1+\alpha +\beta)}\,,\quad \Phi_i^{-1,+}=-\frac{i\left(i +\beta \right)}{(1+\alpha)(2 i +1+\alpha +\beta) }\,,\\
        &  \lambda^{-}_{x,\rr}=\frac{ x+1 +\alpha}{1+\alpha }\,,\quad  \Phi_i^{0,-}=\frac{ i +2+\alpha +\beta +N }{2 i +2+\alpha +\beta}\,,\quad 
        \Phi_i^{1,-}=\frac{i-N}{ 2 i +2+\alpha +\beta }\,,
     \end{align*}
     \end{subequations}
%     \textcolor{red}{Christoffel parameter: $\nu=-\alpha-1.$}
      \item[(HII)] $\eta=-1\,, \quad\oalpha=\alpha+1\,,\quad \obeta=\beta\,,\quad \oN=N-1\,, \quad \ox= x-1;$
    \begin{subequations}
          \begin{align*}
       & \lambda^{+}_{x,\rr}=1\,,\quad \Phi_i^{0,+}=\frac{\left(i +1+\alpha \right) \left(i +1+\alpha +\beta \right) \left(N -i \right)}{(2 i +1+\alpha +\beta)\left(1+\alpha \right) N }\,,\\
       &\Phi_i^{-1,+}=-\frac{i \left(i +1+\alpha +\beta +N \right) \left(i +\beta \right)}{(2 i +1+\alpha +\beta)\left(1+\alpha \right) N}\,,\\
        &  \lambda^{-}_{x,\rr}=-x \,,\quad \Phi_i^{0,-}=-\frac{\left(1+\alpha \right) N}{2 i+2 +\alpha +\beta }\,,\quad \Phi_i^{1,-}=\frac{\left(1+\alpha \right) N}{2 i +2+\alpha +\beta }\,,
     \end{align*}
     \end{subequations}
%     \textcolor{red}{Christoffel parameter: $\nu=0.$}
    \item[(HIII)] $\eta=0\,, \quad\oalpha=\alpha\,,\quad \obeta=\beta+1\,,\quad \oN=N\,, \quad \ox= x;$
     \begin{subequations}
          \begin{align*}
       & \lambda^{+}_{x,\rr}=1\,,\quad \Phi_i^{0,+}=\frac{i +1+\alpha +\beta}{2 i +1+\alpha +\beta}\,,\quad \Phi_i^{-1,+}=\frac{i}{2 i +1+\alpha +\beta}\,,\\
        &  \lambda^{-}_{x,\rr}= N +\beta -x +1\,,\ \Phi_i^{0,-}=\frac{\left(i +2+\alpha +\beta +N \right) \left(i +1+\beta \right)}{ 2 i+2 +\alpha +\beta}\,,\ &\Phi_i^{1,-}=\frac{ \left(N -i \right) \left(i +1+\alpha \right)}{2 i+2 +\alpha +\beta}\,,
     \end{align*}
     \end{subequations}
%     \textcolor{red}{Christoffel parameter: $\nu=N+\beta+1.$}
    \item[(HIV)] $\eta=0\,,\quad\oalpha=\alpha\,,\quad \obeta=\beta+1\,,\quad \oN=N-1\,, \quad \ox= x;$
     \begin{subequations}
          \begin{align*}
       & \lambda^{+}_{x,\rr}=1\,,\quad \Phi_i^{0,+}=\frac{\left(i +1+\alpha +\beta \right) \left(N -i \right)}{(2 i +1+\alpha +\beta)N}\,,\quad \Phi_i^{-1,+}=\frac{i\left(i +1+\alpha +\beta +N \right)}{(2 i +1+\alpha +\beta)N}\,,\\
        &  \lambda^{-}_{x,\rr}=N -x \,,\quad \Phi_i^{0,-}=\frac{N \left(i +1+\beta \right)}{2 i+2+\alpha +\beta }\,,\quad \Phi_i^{1,-}=\frac{ N\left(i +1+\alpha \right)}{2 i +2+\alpha +\beta}\,.
     \end{align*}
     \end{subequations}
%     \textcolor{red}{Christoffel parameter: $\nu=N.$}
\end{itemize}
We recover the classification obtained in \cite{Oste2015Doubling}. We have used the same labeling of the solutions as in \cite{Oste2015Doubling}. 

The constraints $(\mathfrak{C}_{B_2})$ for the Hahn polynomials have been solved for the cases where $\oN=N,N+1,N-1,N+2,N-2$. In these cases, all the contiguity relations of type $B_2$ for the Hahn polynomials are derived by applying the contiguity relations of type $A_2$ twice, as discussed in Remark \ref{rmk:B2-from-A2}.
Similarly, we proved that all the contiguity relations of type ${B'_2}$ for the Hahn polynomials are obtained by using twice contiguity relations of type ${A_2}$ as in Remark \ref{rmk:B2'-from-A2}.

The results in this subsection can be recovered (up to global normalizations) from the $q$-Hahn case by reparametrizing $\alpha \to q^\alpha$, $\beta \to q^\beta$, and taking the limit $q\to 1$ (it is necessary to renormalize some coefficients in order for the limit to be well-defined). The correspondence between the solutions preserves the numbering.

\subsection{Dual Hahn}

The dual Hahn polynomials are given by
\begin{align}
    R^{(dH)}_i(x;\rr)={}_3F_2 \left({{-i, \;-x,\;x+\alpha+\beta+1}\atop
{\alpha+1,\;-N}}\;\Bigg\vert \; 1\right)\,,
\end{align}
where the parameter $\rr=\alpha,\beta,N.$ Knowing that $\lambda^{(dH)}_{x,\rr}=x(x+\alpha+\beta+1)$, relation \eqref{eq:lxr} leads to 
$$\zeta=1\,, \quad \xi=\eta(\alpha+\beta-\eta+1),\quad  \beta+\alpha-\oalpha-\obeta=2\eta\,,$$
where $\eta\in\{0,+1,-1\}$. The coefficients of the recurrence relation read
\begin{align}
    A^{(dH)}_{i,\rr}=(i+\alpha+1)(i-N)\,,\quad
    C^{(dH)}_{i,\rr}=i(i-\beta-N-1)\,.
\end{align}
The whole list of the contiguity relations of type $A_2$ (except the trivial one) for the dual Hahn polynomials is
\begin{itemize}
%    \item[(dHI)]$\eta=0, \quad \zeta=1, \quad\overline{\alpha}=\alpha\,,\quad \overline{\beta}=\beta\,,\quad \overline{N}=N,\,;$
%    \begin{subequations}
%          \begin{align*}
%       & \lambda^{+}_{x,\rr}=1,\quad \Phi_i^{0,+}=1\,,\quad \Phi_i^{-1,+}=0\,,\\
%        &  \lambda^{-}_{x,\rr}=?\,,\quad \Phi_i^{0,-}=?\,,\quad \Phi_i^{1,-}=?\,,
%     \end{align*}
%     \end{subequations}
  \item[(dHI)] $\eta=-1\,,  \quad\oalpha=\alpha+1\,,\quad \obeta=\beta+1\,,\quad \oN=N-1\,, \quad \ox= x-1;$
   \begin{subequations}
          \begin{align*}
       & \lambda^{+}_{x,\rr}=1\,,\quad \Phi_i^{0,+}=\frac{(i+1+\alpha)(N-i)}{N(1+\alpha)}\,,\quad\Phi_i^{-1,+}=\frac{i\left(i -1-\beta -N \right)}{N(1+\alpha)} \,, \\
        &  \lambda^{-}_{x,\rr}=\frac{x \left(x+\alpha +\beta+1 \right) }{\left(1+\alpha\right) N}\,,\quad \Phi_i^{0,-}=1\,,\quad \Phi_i^{1,-}=-1\,.
     \end{align*}
     \end{subequations}
%     \textcolor{red}{Christoffel parameter: $\nu=0.$}
%     \textcolor{red}{ Monic version
%     \begin{subequations}
%          \begin{align*}
%       & \lambda^{+}_{x,\rr}=1,\quad \Psi_i^{0,+}=1\,,\quad \Psi_i^{-1,+}=i \left(-i +1+\beta +N \right)=-C_{i,\rr}^{(dH)}\,,\qquad \text{Geronimus}\\
%        &  \lambda^{-}_{x,\rr}=x \left(x+1 +\alpha +\beta \right) ,\quad \Psi_i^{0,-}=\left(-i +N \right) \left(i +1+\alpha \right)=-A_{i,\rr}^{(dH)}
%\,,\quad \Psi_i^{1,-}=1\,,\qquad \text{Christoffel}
%     \end{align*}
%     \end{subequations}}
    \item[(dHII)] $\eta=0\,,\quad\oalpha=\alpha\,,\quad \obeta=\beta\,,\quad \oN=N-1\,, \quad \ox= x;$
     \begin{subequations}
          \begin{align*}
       & \lambda^{+}_{x,\rr}=1\,,\quad \Phi_i^{0,+}=\frac{N-i}{N}\,,\quad \Phi_i^{-1,+}=\frac{i}{N}\,,\\
        &  \lambda^{-}_{x,\rr}=\frac{\left(N-x  \right) \left(x+1+\alpha+\beta +N\right)}{N}\,,\quad \Phi_i^{0,-}=\beta+N-i\,,\quad \Phi_i^{1,-}=i+\alpha+1\,,
     \end{align*}
     \end{subequations}
%     \textcolor{red}{Christoffel parameter: $\lambda_\nu=N \left(\beta +N +\alpha +1\right)$. Monic version
%     \begin{subequations}
%          \begin{align*}
%       & \lambda^{+}_{x,\rr}=1,\quad \Psi_i^{0,+}=1\,,\quad \Psi_i^{-1,+}=-i(i+\alpha)\,,\qquad \text{Geronimus}\\
%        &  \lambda^{-}_{x,\rr}=\left(x-N  \right) \left(x+1+\alpha+\beta +N\right),\quad \Psi_i^{0,-}=-(N-i)(\beta+N-i)\,,\quad \Psi_i^{1,-}=1\,,\qquad \text{Christoffel}
%     \end{align*}
%     \end{subequations}
%        We verify: 
%     \begin{align*}
%\Psi_{i-1}^{0,-}\Psi_i^{-1,+}&=(-i+\beta+N+1)(i+\alpha)i(N+1-i)=A_{i-1,\rr}^{(dH)}C_i^{(dH)},\\
%\Psi_{i}^{0,-}+\Psi_i^{-1,+}+N \left(\beta +N +\alpha +1\right)
%&=N \alpha +2 i N -i \alpha +\beta  i -2 i^{2}+N
%=-(A_{i,\rr}^{(K)}+C_{i,\rr}^{(K)})
%     \end{align*}}
        \item[(dHIII)] $\eta=0\,,\quad\oalpha=\alpha+1\,,\quad \obeta=\beta-1\,,\quad \oN=N\,, \quad \ox= x;$
    \begin{subequations}
          \begin{align*}
       & \lambda^{+}_{x,\rr}=1\,,\quad \Phi_i^{0,+}=\frac{i+1+\alpha}{1+\alpha}\,,\quad \Phi_i^{-1,+}=-\frac{i}{1+\alpha}\,,\\
        &  \lambda^{-}_{x,\rr}=-\frac{\left(x+1 +\alpha \right) \left(x+\beta  \right)}{1+\alpha}\,,\quad \Phi_i^{0,-}=i - \beta - N\,,\quad \Phi_i^{1,-}=N-i\,.
     \end{align*}
     \end{subequations}
%     \textcolor{red}{Christoffel parameter: $\lambda_\nu=-\beta(1+\alpha)$. Monic version
%     \begin{subequations}
%          \begin{align*}
%       & \lambda^{+}_{x,\rr}=1,\quad \Psi_i^{0,+}=1\,,\quad \Psi_i^{-1,+}=i(N+1-i)\,,\qquad \text{Geronimus}\\
%        &  \lambda^{-}_{x,\rr}=\left(x+1 +\alpha \right) \left(x+\beta  \right),\quad \Psi_i^{0,-}=(-i+\beta+N)(i+1+\alpha)\,,\quad \Psi_i^{1,-}=1\,,\qquad \text{Christoffel}
%     \end{align*}
%    \end{subequations}
%        We verify: 
%     \begin{align*}
%\Psi_{i-1}^{0,-}\Psi_i^{-1,+}&=(-i+\beta+N+1)(i+\alpha)i(N+1-i)=A_{i-1,\rr}^{(dH)}C_i^{(dH)},\\
%\Psi_{i}^{0,-}+\Psi_i^{-1,+}-\beta(\alpha+1)&=N \alpha +2 i N -i \alpha +\beta  i -2 i^{2}+N
%=-(A_{i,\rr}^{(K)}+C_{i,\rr}^{(K)})
%     \end{align*}}
\end{itemize}
We recover the classification obtained in \cite{Oste2015Doubling}. We have used the same labeling of the solutions as in \cite{Oste2015Doubling}. 

%For example
%\begin{itemize}
%\item $\eta=0, \quad\oalpha=\alpha+2\,,\quad \obeta=\beta-2\,,\quad \oN=N\,;$
%\item $\eta=0, \quad\oalpha=\alpha+1\,,\quad \obeta=\beta-1\,,\quad \oN=N-1\,$;
 %\item $\eta=0, \quad\oalpha=\alpha\,,\quad \obeta=\beta\,,\quad \oN=N-2\,,$
 %\item $\eta=-1, \quad\oalpha=\alpha+2\,,\quad \obeta=\beta\,,\quad \oN=N-1\,,$
 %\item $\eta=-1, \quad\oalpha=\alpha+1\,,\quad \obeta=\beta+1\,,\quad \oN=N-2\,,$
 %\item $\eta=-2, \quad\oalpha=\alpha+2\,,\quad \obeta=\beta\,,\quad \oN=N-2\,$.
%\end{itemize}
%$$\eta=0, \quad\oalpha=\alpha+1\,,\quad \obeta=\beta-1\,,\quad \oN=N-1\,$$
%is obtained by using the second and third relation of type $A_2$, and
%$$\eta=-2, \quad\oalpha=\alpha+2\,,\quad \obeta=\beta\,,\quad \oN=N-2\,$$
%is obtained by using two times the last relation of type $A_2.$
The contiguity relations of type $B_2$ for the dual Hahn polynomials are obtained using twice contiguity relations of type $A_2$ as in Remark \ref{rmk:B2-from-A2}. Similarly, we proved that all the contiguity relations of type ${B'_2}$ for the dual Hahn polynomials are obtained by using twice contiguity relations of type ${A_2}$ as in Remark \ref{rmk:B2'-from-A2}. 

The results in this subsection can be recovered (up to global normalizations) from the dual $q$-Hahn case by reparametrizing $\alpha \to q^\alpha$, $\beta \to q^\beta$, and taking the limit $q\to 1$ (it is necessary to renormalize some coefficients in order for the limit to be well-defined). The correspondence between the solutions preserves the numbering.

\subsection{Krawtchouk \label{sec:kraw}}
The Krawtchouk polynomials are given by
\begin{align}
R^{(K)}_i(x;\rr)={}_2F_1 \left({{-i, \;-x}\atop
{-N}}\;\Bigg\vert \; \frac{1}{\alpha}\right)\,,
\end{align}
where the parameter $\rr=\alpha, N.$ Knowing that $\lambda^{(K)}_{x,\rr}=-x$, relation \eqref{eq:lxr} leads to 
$$\zeta=1\,, \quad \xi=-\eta \in\{0,+1,-1\}\,.$$
The coefficients of the recurrence relation read
\begin{align}
    A^{(K)}_{i,\rr}=\alpha(N-i)\,,\quad
    C^{(K)}_{i,\rr}=i(1-\alpha)\,.
\end{align}
The whole list of the contiguity relations of type $A_2$ (except the trivial one) for the Krawtchouk polynomials is
\begin{itemize}
  %  \item[(KI)] $\xi=0, \quad\oalpha=\alpha\,,\quad \oN=N\,,$
    % \begin{subequations}
   %       \begin{align*}
    %   & \lambda^{+}_{x,\rr}=1,\quad \Phi_i^{0,+}=1\,,\quad \Phi_i^{-1,+}=0\,,\\
     %   &  \lambda^{-}_{x,\rr}=?\,,\quad \Phi_i^{0,-}=?\,,\quad \Phi_i^{1,-}=?\,,
     %\end{align*}
     %\end{subequations}
    \item[(KI)] $\eta=0\,,\quad\oalpha=\alpha\,,\quad \oN=N-1\,, \quad \ox=x;$
     \begin{subequations}
          \begin{align*}
       & \lambda^{+}_{x,\rr}=1\,,\quad \Phi_i^{0,+}=\frac{N-i}{N}\,,\quad \Phi_i^{-1,+}=\frac{i}{N}\,,\\
        &  \lambda^{-}_{x,\rr}=\frac{N-x}{ N}\,,\quad \Phi_i^{0,-}=1-\alpha\,,\quad \Phi_i^{1,-}=\alpha\,,
     \end{align*}
     \end{subequations}
%     \textcolor{red}{Christoffel parameter: $\nu=N.$ Monic coefficients: \begin{subequations}
%          \begin{align*}
%       & \lambda^{+}_{x,\rr}=1,\quad \Psi_i^{0,+}=1\,,\quad \Psi_i^{-1,+}=\alpha i\,,\qquad\text{Geronimus}\\
%        &  \lambda^{-}_{x,\rr}=N-x\,,\quad \Psi_i^{0,-}=(1-\alpha)(N-i)\,,\quad \Psi_i^{1,-}=1\,,\qquad\text{Christoffel}
%     \end{align*}
%     We verify: 
%     \begin{align*}
%\Psi_{i-1}^{0,-}\Psi_i^{-1,+}&=\alpha(1-\alpha)i(N+1-i)=A_{i-1,\rr}^{(K)}C_{i,\rr}^{(K)},\\
%\Psi_{i}^{0,-}+\Psi_i^{-1,+}-N&=\alpha i-i-\alpha N+\alpha i=-(A_{i,\rr}^{(K)}+C_{i,\rr}^{(K)})
%     \end{align*}
%     \end{subequations}
%     }
    \item[(KII)] $\eta=-1\,,\quad\oalpha=\alpha\,,\quad \oN=N-1\,, \quad \ox= x-1;$
    \begin{subequations}
          \begin{align*}
       & \lambda^{+}_{x,\rr}=1\,,\quad \Phi_i^{0,+}=\frac{N-i}{ N}\,,\quad \Phi_i^{-1,+}=\frac{(\alpha-1)i}{\alpha N}\,,\\
        &  \lambda^{-}_{x,\rr}=x\,,\quad \Phi_i^{0,-}=\alpha N\,,\quad \Phi_i^{1,-}=-\alpha N\,.
     \end{align*}
     \end{subequations}
%     \textcolor{red}{Christoffel parameter: $\nu=0.$ Monic coefficients: \begin{subequations}
%          \begin{align*}
%       & \lambda^{+}_{x,\rr}=1,\quad \Psi_i^{0,+}=1\,,\quad \Psi_i^{-1,+}=(\alpha-1) i=-C_{i,\rr}^{(K)}\,,\qquad\text{Geronimus}\\
%        &  \lambda^{-}_{x,\rr}= \frac{\alpha Nx}{\alpha-1}\,,\quad \Psi_i^{0,-}=-\alpha(N-i)=-A_{i,\rr}^{(K)}\,,\quad \Psi_i^{1,-}=1\,,\qquad\text{Christoffel}
%     \end{align*}
%     \end{subequations}}
\end{itemize}

%For example
%\begin{itemize}
%    \item $\xi=0,\quad\oalpha=\alpha\,,\quad \oN=N-2\,,$
%    \item $\xi=2,\quad \oalpha=\alpha\,,\quad \oN=N-2\,$
%\end{itemize}
%are obtained by using two times the last two relations of type $A_2$ respectively. 
We proved that the contiguity relations of type $B_2$ for the Krawtchouk polynomials are obtained using twice contiguity relations of type $A_2$ as in Remark \ref{rmk:B2-from-A2}. Similarly, we proved that all the contiguity relations of type ${B'_2}$ for the Krawtchouk polynomials are obtained by using twice contiguity relations of type ${A_2}$ as in Remark \ref{rmk:B2'-from-A2}. 

The results in this subsection can be recovered (up to global normalizations) from the four $q$-deformed Krawtchouk families by taking the limit $q\to 1$ (it is necessary to renormalize some coefficients in order for the limit to be well-defined) and performing the following reparametrizations: $\alpha \to \alpha^{-1}$ in the quantum $q$-Krawtchouk case, $\alpha \to 1-\alpha^{-1}$ in the $q$-Krawtchouk and dual $q$-Krawtchouk cases, and $\alpha \to 1-\alpha$ in the affine $q$-Krawtchouk case. The correspondence between the solutions is: (qqKI) $\to$ trivial, (qqKII) $\to$ (KI), (qqKIII) $\to$ (KII), (qKI) $\to$ trivial, (qKII) $\to$ (KI), (qKIII) $\to$ (KII), (aqKI) $\to$ trivial, (aqKII) $\to$ (KI), (aqKIII) $\to$ (KII), (dqKI) $\to$ (KI) and (dqKII) $\to$ (KII). 

%\begin{itemize}
%\item $\xi=0, \quad \oalpha=\alpha,\quad \oN=N+1;$
%\item $\xi=1, \quad \oalpha=\alpha,\quad \oN=N;$
%\item $\xi=-1, \quad \oalpha=\alpha,\quad \oN=N+1;$
%\item $\xi=-1, \quad \oalpha=\alpha,\quad \oN=N;$
%\end{itemize}

%%%%%%%%%%%%%%%%%%%%%%%%%%%%%%%%%%%%%%%%%%%%%%%%%%%%%%%%
%%%%%%%%%%%%%%%%%%%%%%%%%%%%%%%%%%%%%%%%%%%%%%%%%%%%%%%%
%%%%%%%%%%%%%%%%%%%%%%%%%%%%%%%%%%%%%%%%%%%%%%%%%%%%%%%%
%%%%%%%%%%%%%%%%%%%%%%%%%%%%%%%%%%%%%%%%%%%%%%%%%%%%%%%%
\section{Relations for the Bannai--Ito polynomials \label{sec:BI}} 

The case of the Bannai--Ito polynomials is more involved. Indeed, in the following we show that they do not satisfy non-trivial $A_2$-contiguity relations. However, taking suitable $q \to -1$ limits of the contiguity relations for the $q$-Racah polynomials, relations involving Bannai--Ito polynomials and complementary Bannai--Ito polynomials are obtained. Additionally, we showed that all $B_2$ and $B_2'$ contiguity relations are obtained from the latter. We start by recalling the definition of the Bannai--Ito polynomials.

\subsection{Definition and properties of the Bannai--Ito polynomials}

Given any integer $n$, one can define uniquely  $n^e$ an integer and $n^p\in\{0,1\}$ by
\begin{align}
    n=2n^e+n^p\,.
\end{align}
The monic Bannai--Ito polynomials are defined by\footnote{To recover the notations used in previous papers as in \cite{Vinet2013}, the following transformations on the parameters must be performed
\begin{align}
&\alpha\to \rho_1+\rho_2\,,\quad \beta \to -r_1-r_2\,,\quad \gamma \to r_1+\rho_1+\frac12\,,\quad N \to 2r_1 -2\rho_1 -1\,, \qquad \text{for $N$ even}\,,\nonumber\\
&\alpha\to \rho_1-r_1-\frac12\,,\quad \beta \to \rho_2-r_2+\frac12\,,\quad \gamma \to \rho_1-\rho_2\,,\quad N \to -2\rho_1-2\rho_2-1\,, \qquad \text{for $N$ odd}\,.\nonumber
\end{align}
}
\begin{align}
   &  B_i(x;\rr)= \kappa_i\Bigg( {}_4F_3 \left({{-i^e,\;i^e+\alpha+\beta+1, \;-x^e,\;x^e+\gamma-N^e}\atop
{\alpha+1,\;\beta+\gamma,\;-N^e}}\;\Bigg\vert \; 1\right)\\
&\hspace{2cm}+\frac{(-1)^{i+x}\big(i+i^p(2\alpha+2\beta+1)\big)\big(x-x^p(2N^e-2\gamma+1)\big)}{2(\beta+\gamma)\big(2N^e-N^p(N+2\alpha+1)\big)}\nonumber\\
&\hspace{2.5cm}\times{}_4F_3 \left({{-i^e-i^p+1,\;i^e+i^p+\alpha+\beta+1, \;-x^e-x^p+1,\;x^e+x^p+\gamma-N^e}\atop
{\alpha+N^p+1,\;\beta+\gamma+1,\;-N^e-N^p+1}}\;\Bigg\vert \; 1\right)\Bigg)\,,\nonumber
\end{align}
where $\rr=\alpha,\beta,\gamma,N$ and 
\begin{align}
   \kappa_i= \frac{(-N^e)_{i^e+i^p(1-N^p)}(\beta+\gamma)_{i^e+i^p}(1+\alpha)_{i^e+N^p i^p}}{(\alpha+\beta+i^e+1)_{i^e+i^p} }\,.
\end{align}
Knowing that $\lambda^{(BI)}_{x,\rr}=\frac14\Big(1-2\gamma+2N^e-(-1)^x(1-2\gamma+2N^e-2x\Big)$, relation \eqref{eq:lxr} leads to
\begin{align}
    \ox=x+\eta\,,\quad \zeta=(-1)^\eta\,,\quad \ogamma=\gamma+\oN^e-N^e-\eta\,,\quad \xi=2(1-(-1)^\eta)(  \gamma-N^e-(\eta+1)/2)\,,
\end{align}
where $\eta \in\{0,+1,-1\}$.
The monic Bannai--Ito polynomials satisfy the recurrence relation
\begin{align}
    \lambda^{(BI)}_{x,\rr}B_i(x;\rr)=B_{i+1}(x;\rr)-(A^{(BI)}_{i,\rr}+C^{(BI)}_{i,\rr})B_i(x;\rr)+A^{(BI)}_{i-1,\rr}C^{(BI)}_{i,\rr}B_{i-1}(x;\rr)\,,
\end{align}
where the coefficients are
\begin{subequations}
    \begin{align}
   & A^{(BI)}_{i,\rr}=\frac{1}{4(i+1+\alpha+\beta)}\times\begin{cases}        (i-N)(i+2\beta+2\gamma)\,,&i^p=N^p=0\,,\\
        (i+1+2\alpha)(i+1+2\alpha+2\beta)\,,&i^p=1\,,N^p=0\,,\\
        (i+2+2\alpha)(i+2\beta+2\gamma)\,,&i^p=0\,,N^p=1\,,\\
         (i-N)(i+1+2\alpha+2\beta)\,,&i^p=N_p=1\,,\\
    \end{cases}\\
  &  C^{(BI)}_{i,\rr}=\frac{-1}{4(i+\alpha+\beta)}\times\begin{cases}
        i(i+2\beta)\,,&i^p=N^p=0\,,\\
        (i+1+2\alpha-2\gamma)(i+1+N+2\alpha+2\beta)\,,&i^p=1\,,N^p=0\,,\\
        i(i+1+N+2\alpha+2\beta)\,,&i^p=0\,,N^p=1\,,\\
         (i-1+2\beta)(i+1+2\alpha-2\gamma)\,,&i^p=N^p=1\,.\\
    \end{cases}
\end{align}
\end{subequations}
For $N$ odd, they also possess a symmetry property given by
\begin{align}
B_i(x;a,b,c,N)=B_i(x;b+c-1,a-c+1,c,N)\,.
\end{align}

From the previous expressions of $A^{(BI)}_{i,\rr}$ and $C^{(BI)}_{i,\rr}$, we show that the constraints $(\mathfrak{C}_{A_2})$ are satisfied only in the trivial case.
However, as explained in the introduction of this section, limits $q\to -1$ of the $A_2$-contiguity relations for the $q$-Racah polynomials provide  relations between Bannai--Ito and complementary Bannai--Ito polynomials. We recall the definition of the latter, then the limits $q\to -1$ we are interested in, and finally  we give a list of these relations.

\subsection{Complementary Bannai--Ito polynomials}

The complementary Bannai--Ito polynomials $I_i(x)=I_i(x;a,b,c,N)$ have been introduced and studied in \cite{VinetCBI,Vinet2013}. 
They are defined in terms of hypergeometric functions\footnote{The notations in \cite{Vinet2013} are recovered with the following changes
\begin{align}
&a \to \rho_1-r_1-\frac12\,,\quad b \to \rho_2-r_2+\frac12\,,\quad c \to \rho_1-\rho_2\,,\quad N \to -2\rho_1-2\rho_2-2\,, \qquad \text{for $N$ even}\,,\nonumber\\
&a \to \rho_1+\rho_2\,,\quad b \to -r_1-r_2\,,\quad c \to r_1+\rho_1+\frac12\,,\quad N \to 2r_1 -2\rho_1 -2\,, \qquad \text{for $N$ odd}\,.\nonumber
\end{align}} as, for $N$ even, 
\begin{align}
  I_{i}(x)=&\frac{(i^p-N^e)_{i^e}( a-c+i^p+1)_{i^e}(b+i^p)_{i^e}}{(i^e+i^p+1+a+b)_{i^e}} 
 (\mu(x;a,b,c,N)-\sigma(N) )^{i^p}\\
\times & {}_4F_3 \left({{-i^e, \;i^e+i^p+1+a+b,\; -x^e-x^p+i^p,\; x^e+x^p+i^p-N^e-c-1  }\atop
{-N^e+i^p,\; a-c+i^p+1,\; b+i^p}}\;\Bigg\vert \; 1\right)\,,\nonumber
\end{align}
and as, for $N$ odd,
\begin{align}
  I_{i}(x)=&\frac{(N^e-i^e+1)_{i^e}( a-c+i^p+1)_{i^e}(-b-i^e)_{i^e}}{(i^e+i^p+1+a+b)_{i^e}} 
 (\mu(x;a,b,c,N)-\sigma(N) )^{i^p}\\
\times & {}_4F_3 \left({{-i^e, \;i^e+i^p+1+a+b,\; -x^e,\; x^e-N^e-c  }\atop
{-N^e,\; a-c+i^p+1,\; b+1}}\;\Bigg\vert \; 1\right)\,,\nonumber
\end{align}
where 
\begin{equation}
    \begin{aligned}
\sigma(N)=-\frac12\times \begin{cases}
N^e+c+1& \text{for $N$ even,}\\
-N^e-2a+c-2& \text{for $N$ odd,}
\end{cases}
\end{aligned}
\end{equation}
and
\begin{align}
    \mu(x)=
        \frac{(-1)^{x+N}}{2}\left(x-c-N^e +N^p  -1\right)-\frac12 x^p\,.
\end{align}
They satisfy the following three-term recurrence relation:
\begin{align}
  I_{i+1}(x)+(-1)^i \sigma I_i(x) + \tau_i I_{i-1}(x) = \mu(x) I_i(x)\,, 
\end{align}
where 
\begin{align}
\tau_i&=A^{(BI)}_{i,a,b,c,N+1}C^{(BI)}_{i,a,b,c,N+1}\,.
\end{align}
In this case for $N$ even, they also possess a symmetry property given by
\begin{align}
I_i(x;a,b,c,N)=I_i(x;b+c-1,a-c+1,c,N)\,.
\end{align}

There exists a limit allowing to obtain the Bannai--Ito polynomials from the $q$-Racah ones \cite{BI84}. In the notations used here, this limit is performed as follows. Firstly change the parameters of the $q$-Racah polynomials in the following way:
\begin{align}
q\to -e^\epsilon\,,\ q^N \to (-e^\epsilon)^N\,,\ \alpha \to (-1)^N e^{\epsilon(2\alpha +N^p)}
\,,\ \beta \to (-1)^N e^{\epsilon(2\beta -N^p)}
\,,\ \gamma \to (-1)^{N+1} e^{\epsilon(2\gamma-1 +N^p)}\,.
\end{align}
Secondly, take the limit $\epsilon \to 0$.
With this procedure the monic $q$-Racah polynomials $P_i^{(qR)}(x;\rr)$ tend to Bannai--Ito polynomials:
\begin{align}    P_i^{(qR)}(x;\alpha,\beta,\gamma,N,q)
\to B_i(x;\alpha,\beta,\gamma,N)\,.
\end{align}
One can also perform this procedure on the coefficients of the recurrence relation of the $q$-Racah polynomials to show that one obtains at the leading order on $\epsilon$ the coefficients of the recurrence relation of the Bannai--Ito polynomials.

If the parameters in the $q$-Racah polynomials are transformed as in the contiguity relations, the $q$-Racah polynomials do not tend to Bannai--Ito polynomials but to complementary Bannai--Ito polynomials.
For example, the following limit holds, for $N$ even,
    \begin{align}
&P_i^{(qR)}(x;\alpha,\beta q,\gamma/q,N-1,q) \to (-1)^i I_i(x;\beta,\alpha,1-\gamma,N-1)\,.
\end{align}

\subsection{Relations between Bannai--Ito and complementary Bannai--Ito polynomials  }

The results in the previous section suggest looking for relations between Bannai--Ito and complementary Bannai--Ito polynomials of the following type
\begin{subequations}\label{eq:BIcBI}
\begin{align}
&B_i(x;\alpha,\beta,\gamma,N)=
 z_i^0  I_i(\ox;\oalpha,\obeta,\ogamma,\oN)
+ \frac{z_i^-}{i+\alpha+\beta} I_{i-1}(\ox;\oalpha,\obeta,\ogamma,\oN)\,,\\
 &   \frac12 \left( x(-1)^x+(-1)^x\left(\gamma-N^e-\frac{1}{2}\right)+\omega_x \right)I_i(\ox;\oalpha,\obeta,\ogamma,\oN)\nonumber\\ &\qquad=\frac{w_i^0}{i+1+\alpha+\beta} B_i(x;\alpha,\beta,\gamma,N)+ w_i^+ B_{i+1}(x;\alpha,\beta,\gamma,N)\,.
\end{align}
\end{subequations}
Numerous possibilities for $\ox;\oalpha,\obeta,\ogamma,\oN$ have been tested with a mathematical software, and the relations found have then been proved by a direct computation. By this method, we find that for the following possibilities relations of the form \eqref{eq:BIcBI} hold:
\begin{itemize}
    \item[(B1)] $N$ even, $\ox=x\,,\ \oalpha=\beta-1\,,\ \obeta=\alpha+1\,,\ \ogamma=-\gamma\,,\ \oN=N$ with
    \begin{align}
      & z_i^0=1\,,\quad 
       z_i^-=\begin{cases}i^e(\beta+i^e)\,,\ i\ \text{even}\\
         (i^e-N^e)(\beta+\gamma+i^e)\,, \ i\ \text{odd}\\
       \end{cases}\,,\quad w_i^+=1\,,\nonumber\\
       & \omega_x= \gamma-2\alpha-\frac{N+3}{2} \,,\ w_i^0=-\begin{cases}
          (\alpha-\gamma+i^e+1)(\alpha+\beta+N^e+i^e+1)\,, \,i\ \text{even}\\
         (\alpha+i^e+1)(\alpha+\beta+i^e+1)\,, \ i\ \text{odd}\\
       \end{cases} \,,\nonumber
    \end{align}
    \item[(B2)] $N$ even,  $\ox=x\,,\ \oalpha=\beta\,,\ \obeta=\alpha\,,\ \ogamma=-\gamma+1\,,\ \oN=N-1$  with
     \begin{align}
      & z_i^0=(-1)^{i_p}\,,\quad 
       z_i^-=\begin{cases}i^e(\alpha+i^e)\,,i\ \text{even}\\
          -(\alpha+\beta+1 +i^e+N^e)(\beta+\gamma+i^e)\,,  i\ \text{odd}\\
       \end{cases}\,,\quad w_i^+=(-1)^{i_p}\,,\nonumber\\
       &\omega_x=-\gamma+\frac{1-N}{2}
    \,,\quad w_i^0=\begin{cases}
           (i^e-N^e)(\alpha-\gamma+i^e+1)\,, \,i\ \text{even}\\
          -(\beta+i^e+1)(\alpha+\beta+i^e+1)\,, \ i\ \text{odd}\\
       \end{cases} \,,\nonumber
    \end{align}
     \item[(B3)] $N$ odd,  $\ox=x\,,\ \oalpha=\beta-1\,,\ \obeta=\alpha+1\,,\ \ogamma=-\gamma\,,\ \oN=N-1$  with
     \begin{align}
      & z_i^0=1\,,\quad 
       z_i^-=\begin{cases}i^e(\alpha+\beta+N^e+i^e+1)\,,i\ \text{even}\\
          (\alpha+1 +i^e)(\beta+\gamma+i^e)\,,  i\ \text{odd}\\
       \end{cases}\,,\quad w_i^+=1\,,\nonumber\\
       &\omega_x=\gamma+\frac{N}{2}\,,\quad w_i^0=\begin{cases}
           -(\beta+i^e)(\alpha-\gamma+i^e+1)\,, \,i\ \text{even}\\
          (N^e-i^e)(\alpha+\beta+i^e+1)\,, \ i\ \text{odd}\\
       \end{cases} \,,\nonumber
    \end{align}
       \item[(B4)] $N$ odd,  $\ox=x\,,\ \oalpha=\beta-1\,,\ \obeta=\alpha+1\,,\ \ogamma=-\gamma\,,\ \oN=N$  with
     \begin{align}
      & z_i^0=(-1)^{i^p}\,,\quad 
       z_i^-=\begin{cases}i^e(i^e-N^e-1)\,,i\ \text{even}\\
          -(\beta +i^e)(\beta+\gamma+i^e)\,,  i\ \text{odd}\\
       \end{cases}\,,\quad w_i^+=(-1)^{i^p}\,,\nonumber\\
       &\omega_x=2\alpha-\gamma+2+\frac{N}{2} \,,\quad w_i^0=\begin{cases}
           (\alpha+i^e+1)(\alpha-\gamma+i^e+1)\,, \,i\ \text{even}\\
          -(\alpha+\beta+i^e+1)(\alpha+\beta+N^e+i^e+2)\,, \ i\ \text{odd}\\
       \end{cases} \,,\nonumber
    \end{align}
      \item[(B5)] $N$ even, $\ox=x-1\,,\ \oalpha=\alpha-\gamma\,,\ \obeta=\beta+\gamma\,,\ \ogamma=-\gamma\,,\ \oN=N-1$ with
    \begin{align}
      & z_i^0=1\,,\quad 
       z_i^-=\begin{cases}i^e(\beta+i^e)\,,\ i\ \text{even}\\
          (\alpha-\gamma+i^e+1)(\alpha+\beta+N^e+i^e+1)\,, \ i\ \text{odd}\\
       \end{cases}\,,\quad w_i^+=1\,,\nonumber\\
       & \omega_x= \frac{N+1}{2}-\gamma \,,\quad 
       w_i^0=\begin{cases}
          (N^e-i^e)(\beta+\gamma+i^e)\,, \,i\ \text{even}\\
        - (\alpha+i^e+1)(\alpha+\beta+i^e+1)\,, \ i\ \text{odd}\\
       \end{cases} \,.\nonumber
    \end{align}
\end{itemize}
It is also natural to look for relations of the following type
\begin{subequations}
\begin{align}
&I_i(x;\alpha,\beta,\gamma,N)=
 z_i^0  B_i(\ox;\oalpha,\obeta,\ogamma,\oN)
+ \frac{z_i^-}{\alpha+\beta+i+1} B_{i-1}(\ox;\oalpha,\obeta,\ogamma,\oN)\,,\\
 &    \frac12 \left( x(-1)^x-(-1)^x\left(\gamma+N^e+\frac{1}{2}\right)+\omega_x \right) B_i(\ox;\oalpha,\obeta,\ogamma,\oN)\nonumber\\
 &\qquad = \frac{w_i^0}{\alpha+\beta+i+1} I_i(x;\alpha,\beta,\gamma,N)+ w_i^+ I_{i+1}(x;\alpha,\beta,\gamma,N)\,.
\end{align}
\end{subequations}
We find the following relations:
\begin{itemize}
    \item[(I1)] $N$ even\,, $\ox=x\,,\ \oalpha=\beta\,,\ \obeta=\alpha+1\,,\ \ogamma=-\gamma\,,\ \oN=N$ with
     \begin{align}
      & z_i^0=1\,,\quad 
       z_i^-=\begin{cases}-i^e(\alpha+i^e+1)\,,\ i\ \text{even}\\
          (\alpha-\gamma+i^e+1)(N^e-i^e)\,, \ i\ \text{odd}\\
       \end{cases}\,,\quad w_i^+=1\,,\nonumber\\
       & \omega_x=2\beta+ \gamma+\frac{N+1}{2}\,,\quad
        w_i^0=\begin{cases}
           (\beta+i^e)(\alpha+\beta+i^e+1)\,, \,i\ \text{even}\\
          (\alpha+\beta+N^e+i^e+2)(\beta+\gamma+i^e+1)\,, \ i\ \text{odd}\\
       \end{cases} \,,\nonumber
    \end{align}
      \item[(I2)] $N$ even$\,,\ \ox=x\,,\ \oalpha=\beta-1\,,\ \obeta=\alpha+2\,,\ \ogamma=-\gamma-1\,,\ \oN=N-1$ with
     \begin{align}
      & z_i^0=1\,,\quad 
       z_i^-=\begin{cases}-i^e(\alpha+\beta+i^e+N^e+1)\,,\ i\ \text{even}\\
          -(\beta+i^e)(\alpha-\gamma+i^e+1)\,, \ i\ \text{odd}\\
       \end{cases}\,,\quad w_i^+=1\,,\nonumber\\
       & \omega_x= \gamma-\frac{N-1}{2}\,,\quad
        w_i^0=\begin{cases}
           (\alpha+\beta+i^e+1)(i^e-N^e)\,, \,i\ \text{even}\\
          (\alpha+i^e+2)(\beta+\gamma+i^e+1)\,, \ i\ \text{odd}\\
       \end{cases} \,,\nonumber
    \end{align}
    \item[(I3)] $N$ odd$\,,\ \ox=x\,,\ \oalpha=\beta\,,\ \obeta=\alpha+1\,,\ \ogamma=-\gamma\,,\ \oN=N$ with
     \begin{align}
      & z_i^0=(-1)^{i^p}\,,\quad 
       z_i^-=\begin{cases}i^e(i^e-N^e-1)\,,\ i\ \text{even}\\
          -(\alpha+i^e+1)(\alpha-\gamma+i^e+1)\,, \ i\ \text{odd}\\
       \end{cases}\,,\quad w_i^+=-(-1)^{i^p}\,,\nonumber\\
       & \omega_x=-2\beta- \gamma-\frac{N+2}{2}\,,\quad
        w_i^0=\begin{cases}
           -(\alpha+\beta+N^e+i^e+2)(\alpha+\beta+i^e+1)\,, \,i\ \text{even}\\
          (\beta+i^e+1)(\beta+\gamma+i^e+1)\,, \ i\ \text{odd}\\
       \end{cases} \,,\nonumber
    \end{align}
    \item[(I4)]  $N$ odd$\,,\ \ox=x\,,\ \oalpha=\beta\,,\ \obeta=\alpha+1\,,\ \ogamma=-\gamma\,,\ \oN=N-1$ with
     \begin{align}
      & z_i^0=(-1)^{i^p}\,,\quad 
       z_i^-=\begin{cases}i^e(\beta+i^e)\,,\ i\ \text{even}\\
          -(\alpha+\beta+N^e+i^e+2)(\alpha-\gamma+i^e+1)\,, \ i\ \text{odd}
       \end{cases}\,,\quad w_i^+=-(-1)^{i^p}\,,\nonumber\\
       & \omega_x=- \gamma+\frac{N}{2}\,,\quad
        w_i^0=\begin{cases}
          - (\alpha+\beta+i^e+1)(\alpha+i^e+1)\,, \,i\ \text{even}\\
          (i^e-N^e)(\beta+\gamma+i^e+1)\,, \ i\ \text{odd}
       \end{cases} \,,\nonumber
    \end{align}
%    \item  $N$ even$\,,\ \ox=x\,,\ \oalpha=\alpha-\gamma+1\,,\ \obeta=\beta+\gamma\,,\ \ogamma=-\gamma\,,\ \oN=N$ with
 %    \begin{align}
  %    & z_i^0=1\,,\quad 
  %     z_i^-=\begin{cases}-i^e(\beta+\gamma+i^e)\,,\ i\ \text{even}\\
   %       (N^e-i^e)(\beta+i^e)\,, \ i\ \text{odd}
  %     \end{cases}\,,\quad w_i^+=1\,,\nonumber\\
  %     & \omega_x=2\alpha- \gamma+\frac{N+5}{2}\,,\quad
  %      w_i^0=\begin{cases}
  %         (\alpha-\gamma+i^e+1)(\alpha+\beta+i^e+1)\,, \,i\ \text{even}\\
   %       (\alpha+i^e+2)(\alpha+\beta+N^e+i^e+2)\,, \ i\ \text{odd}
    %   \end{cases}\nonumber
  %  \end{align}
 %    \item $N$ odd$\,,\ \ox=x\,,\ \oalpha=\alpha-\gamma\,,\ \obeta=\beta+\gamma+1\,,\ \ogamma=-\gamma\,,\ \oN=N$ with
  %   \begin{align}
   %   & z_i^0=(-1)^{i^p}\,,\quad 
  %     z_i^-=\begin{cases}i^e(i^e-N^e-1)\,,\ i\ \text{even}\\
  %        -(\alpha+i^e+1)(\alpha-\gamma+i^e+1)\,, \ i\ \text{odd}\\
  %     \end{cases}\,,\quad w_i^+=-(1)^{i^p}\,,\nonumber\\
  %     & \omega_x=-2\beta-\gamma-\frac{N+2}{2}\,,\quad
  %      w_i^0=\begin{cases}
  %        - (\alpha+\beta+N^e+i^e+2)(\alpha+\beta+i^e+1)\,, \,i\ \text{even}\\
  %        (\beta+i^e+1)(\beta+\gamma+i^e+1)\,, \ i\ \text{odd}\\
   %    \end{cases}\nonumber
  %  \end{align}
    \item[(I5)]  $N$ odd$\,,\ \ox=x-1\,,\ \oalpha=\alpha-\gamma+1\,,\ \obeta=\beta+\gamma\,,\ \ogamma=-\gamma+1\,,\ \oN=N-1$ with
     \begin{align}
      & z_i^0=1\,,\quad 
       z_i^-=\begin{cases}-i^e(\beta+\gamma+i^e)\,,\ i\ \text{even}\\
        - (\alpha+i^e+1)(\alpha+\beta+N^e+i^e+2)\,, \ i\ \text{odd}
       \end{cases}\,,\quad w_i^+=-1\,,\nonumber\\
       & \omega_x= \gamma+\frac{N}{2}\,,\quad
        w_i^0=\begin{cases}
           -(\alpha-\gamma+i^e+1)(\alpha+\beta+i^e+1)\,, \,i\ \text{even}\\
          (\beta+i^e+1)(N^e-i^e)\,, \ i\ \text{odd}
       \end{cases} \,.\nonumber
    \end{align}
\end{itemize}

The constraints $(\mathfrak{C}_{B_2})$ and $(\mathfrak{C}_{B'_2})$ for the Bannai--Ito polynomials are solved. We proved that all these relations are obtained by applying the previous relations twice.

\section{Generalizations of contiguity relations for $_4\phi_3$ functions \label{sec:gen} }

The $q$-Racah polynomials are defined using $1$-balanced ${}_4\phi_3$ functions (see the introduction of Subsection \ref{sec:A2qRacah}). 
In the following, we want to obtain contiguity relations for ${}_4\phi_3$ functions of the form:
\begin{align}
    R^{(G)}_i(x;\alpha,\beta,\gamma,\delta,N,z)={}_4\phi_3 \left({{q^{-i},\;\alpha\beta q^{i+1}, \;q^{-x},\;\gamma q^{x-N}}\atop
{\alpha q,\; \delta,\;q^{-N}}}\;\Bigg\vert \; q;z\right)\,.
\end{align}
We recover the definition of $q$-Racah polynomials by imposing $\delta=\beta\gamma q$ and $z=q$.

To obtain contiguity relations for these ${}_4\phi_3$ functions, numerous possibilities for $\ox,\oalpha,\obeta,\ogamma,\oN$ have been tested with a mathematical
software, then the relations found have been proved by direct computation.

We found that the contiguity relations \eqref{eq:cons1} of type (qRI), (qRII) and (qRIV) can be generalized to, respectively,
\begin{itemize}
    \item[(GI)]
    \begin{align}
    R^{(G)}_{i}(x;\alpha,\beta,\gamma,\delta,N,z) =&\frac{(1-\alpha\beta q^{i+1})(q^i-q^N)}{(1-q^N)(1-\alpha\beta q^{2i+1})}
    R^{(G)}_i(x;\alpha,\beta q,\gamma/q,\delta,N-1,z)\nonumber\\
    +&\frac{(1-\alpha\beta q^{N+i+1})(1-q^i)}{(1-q^N)(1-\alpha\beta q^{2i+1})}
    R^{(G)}_{i-1}(x;\alpha,\beta q,\gamma/q,\delta,N-1,z) \,;\nonumber
\end{align}
\item[(GII)]
\begin{align}
R^{(G)}_{i}(x;\alpha,\beta,\gamma,\delta,N,z)=& \frac{(1-\delta q^{i})(1-\alpha\beta q^{i+1})}{(1-\delta)(1-\alpha\beta q^{2i+1})}
    R^{(G)}_i(x;\alpha,\beta q,\gamma,\delta q,N,z) \nonumber\\
    -&\frac{(1- q^{i})(\delta-\alpha\beta q^{i+1})}{(1-\delta)(1-\alpha\beta q^{2i+1})}
    R^{(G)}_{i-1}(x;\alpha,\beta q,\gamma,\delta q,N,z) \,;\nonumber
\end{align}
\item[(GIV)]
\begin{align}
    R^{(G)}_{i}(x;\alpha,\beta,\gamma,\delta,N,z)=&\frac{(1-\alpha q^{i+1})(1-\alpha\beta q^{i+1})}{(1-\alpha q)(1-\alpha\beta q^{2i+1})} R^{(G)}_i(x;\alpha q,\beta,\gamma,\delta,N,z)\nonumber\\
    -& \frac{\alpha q(1-\beta q^{i})(1-q^{i})}{(1-\alpha q)(1-\alpha\beta q^{2i+1})}R^{(G)}_{i-1}(x;\alpha q,\beta,\gamma,\delta,N,z) \,.\nonumber
\end{align}
\end{itemize}
The contiguity relation \eqref{eq:cons3} of type (qRIII) can be generalized and reads as follows:
\begin{itemize}
    \item[(GIII)] 
    \begin{align}
&(1-q^{-x})(1-\gamma q^{x-N}) R^{(G)}_i(x-1;\alpha q,\beta,\gamma q,\delta q,N-1,z) \nonumber \\
&\quad =\frac{q^{i-N+1}(1-\alpha q)(1- \delta)(1-q^N)}{z(1-\alpha\beta q^{2i+2})}\left(R^{(G)}_{i+1}(x;\alpha,\beta,\gamma,\delta,N,z)-R^{(G)}_i(x;\alpha,\beta,\gamma,\delta,N,z)\right) \,.\nonumber
\end{align}
\end{itemize}

For the above relations  (GI), (GII), (GIII) and (GIV), if the  ${}_4\phi_3$ functions are $k$-balanced on the r.h.s., then they are also $k$-balanced on the l.h.s. of the equations. 
There are equations which do not preserve the balance and which are given by
\begin{itemize}
    \item[(GV)] \begin{align}
R^{(G)}_i(x;\alpha,\beta,\gamma,\delta,N,z)=&\frac{1-\alpha\beta q^{i+1}}{1-\alpha\beta q^{2i+1}} R^{(G)}_i(x;\alpha,\beta q,\gamma,\delta,N,z)\nonumber\\
+&\frac{\alpha\beta q^{i+1}(1-q^i)}{1-\alpha\beta q^{2i+1}} R^{(G)}_{i-1}(x;\alpha,\beta q,\gamma,\delta,N,z) \,,   \nonumber  
\end{align}
  \item[(GVI)] \begin{align}
R^{(G)}_i(x;\alpha,\beta,\gamma,\delta,N,z)=&\frac{q^i(1-\alpha\beta q^{i+1})}{1-\alpha\beta q^{2i+1}} R^{(G)}_i(x;\alpha,\beta q,\gamma,\delta,N,z/q)\nonumber\\
+&\frac{1-q^i}{1-\alpha\beta q^{2i+1}} R^{(G)}_{i-1}(x;\alpha,\beta q,\gamma,\delta,N,z/q) \,.    \nonumber  
\end{align}
\end{itemize}
There exist similar relations for other ($q$-)hypergeometric functions that can be obtained by taking suitable limits, but we do not provide a list here. 

\paragraph{Acknowledgements:}
The authors thank Satoshi Tsujimoto and Maxim Derevyagin for fruitful discussions.
N.~Cramp\'e is partially supported by the international research project AAPT of the CNRS.  L.~Vinet is funded in part by a Discovery Grant from the Natural Sciences and Engineering Research Council (NSERC) of Canada. M.~Zaimi 
%was funded in part by an Alexander Graham-Bell scholarship from NSERC and a doctoral scholarship from Fonds de Recherche du Québec - Nature et Technologies (FRQNT), and 
is supported by a postdoctoral fellowship from NSERC
and by Perimeter Institute for Theoretical Physics. Research at Perimeter Institute is supported in part by the Government of Canada through the Department of Innovation, Science and Economic Development and by the Province of Ontario through the Ministry of Colleges and Universities.
L. Morey wishes to thank Luc Vinet for her postdoctoral fellowship and the CRM for its hospitality.

\appendix

\section{$B_2$-contiguity relations for $q$-Racah polynomials \label{app:B2}} 

The first type of $B_2$-contiguity relations reads as follows
\begin{align}
&\lambda^{+}_{x; \rr}\; R_i(x; \rr)= \Phi^{+1,+}_{i}\ R_{i+1}(\ox; \orr)+\Phi^{0,+}_{i}\ R_{i}(\ox; \orr)+\Phi^{-1,+}_{i}\ R_{i-1}(\ox; \orr)\,.
\end{align}
The second type of the $B_2$-contiguity relations can be obtained from the previous one. 
As proved in this paper, all the $B_2$-contiguity relations for the $q$-Racah polynomials can be obtained from the $A_2$-contiguity relations following the procedure described in Remark \ref{rmk:B2-from-A2}.

In the following, the complete list of the $B_2$-contiguity relations for $q$-Racah polynomials is given (the names refer to the two $A_2$-contiguity relations used to compute the $B_2$-contiguity relations):
\begin{enumerate}
\item[(qRI/II)] $\ox=x, \quad \oalpha=\alpha,\quad \obeta=\beta,\quad \ogamma=\gamma/q\,,\quad \oN=N-1$
\begin{align*}
 &  \lambda_{x,\rr}^+=\frac{(1-\beta\gamma q^{x})(1-\beta q^{N-x})}{\beta q(1-\beta\gamma)}\,,
 \\
  &  \Phi_i^{+1,+}=\frac{(1-q^{i-N})(1-\alpha\beta q^{i+1})(q^i-q^{N-1})(1-\alpha q^{i+1})}
  {(1-q^{-N})(1-\alpha\beta q^{2i+1})(1-\alpha \beta q^{2i+2})}\,,
  \\
 % &  \Phi_i^{0+}=\frac{(1-q^{i-N})(1-\alpha\beta q^{N+i+1})}{q(1-q^{-N})(1-\alpha\beta q^{2i+1})}\left(\frac{(1-\alpha\beta q^{i+1})(1-\beta q^{i+1})}{\beta (1-\alpha\beta q^{2i+2})}+\frac{(1-q^i)(1-\alpha q^i)}{1-\alpha\beta q^{2i}} \right)\,,
  %\\
    &  \Phi_i^{-1,+}=\frac{(1-q^i)(1-\alpha\beta q^{N+i+1})(1-\beta q^i)(1-\alpha\beta q^{N+i})}{\beta q(1-q^N)(1-\alpha\beta q^{2i})(1-\alpha \beta q^{2i+1})}\,,
    \\
     &  \Phi_i^{0+}=\lambda_{0,\rr}^+-\Phi_i^{+1,+}-\Phi_i^{-1,+}\,,
  \\
\end{align*}
\item[(qRI/III)] $\ox=x +1, \quad \oalpha=\alpha/q,\quad \obeta=\beta q,\quad \ogamma=\gamma/q^2\,,\quad \oN=N$
\begin{align*}
 &  \lambda_{x,\rr}^+=\frac{(1- q^{-x-1})(1-\gamma q^{x-N-1})}{(1-\beta\gamma )(1-\alpha)}\,,
 \\
  &  \Phi_i^{+1,+}=-\frac{q^{i}(1-q^{i-N})(1-\alpha\beta q^{i+1})}
  {(1-\alpha\beta q^{2i+1})(1-\alpha \beta q^{2i+2})}\,,
  \\
%  &  \Phi_i^{0+}=\frac{q^i}{(1-q^N)(1-\alpha\beta q^{2i+1})}
 % \left(\frac{(1-q^{i-N})(1-\alpha\beta q^{i+1})}{1-\alpha\beta q^{2i+2}}+
 % \frac{(1-q^i)(1-\alpha\beta q^{N+i+1})}{q^{N+1}(1-\alpha\beta q^{2i})}
  %\right)\,,
  %\\
    &  \Phi_i^{-1,+}=-\frac{q^{i-N-1}(1-q^i)(1-\alpha\beta q^{N+i+1})}{(1-\alpha\beta q^{2i})(1-\alpha \beta q^{2i+1})}\,,
    \\
      &  \Phi_i^{0+}=-\Phi_i^{+1,+}-\Phi_i^{-1,+}\,,\\
\end{align*}
\item[(qRI/IV)] $\ox=x, \quad \oalpha=\alpha/q,\quad \obeta=\beta q,\quad \ogamma=\gamma/q\,,\quad \oN=N-1$
\begin{align*}
 &  \lambda_{x,\rr}^+=\frac{(1-\alpha q^{x})(\gamma-\alpha q^{N-x})}{\alpha q(1-\alpha )}\,,
 \\
  &  \Phi_i^{+1,+}=\frac{(q^i-q^{N-1})(q^{i}-q^N)(1-\alpha\beta q^{i+1})(1-\beta\gamma q^{i+1})}
  {(1-q^{N})(1-\alpha\beta q^{2i+1})(1-\alpha \beta q^{2i+2})}\,,
  \\
 %&  \Phi_i^{0+}=\frac{(1-q^{i-N})(1-\alpha\beta q^{N+i+1})}{q(1-q^{-N})(1-\alpha\beta q^{2i+1})}\left( \frac{(1-\alpha\beta q^{i+1})(\gamma-\alpha q^{i+1})}{\alpha(1-\alpha \beta q^{2i+2})} +\frac{(1-q^i)(1-\beta\gamma q^i)}{1-\alpha\beta q^{2i}}\right)
 % \\
    &  \Phi_i^{-1,+}=\frac{(1-q^i)(1-\alpha\beta q^{N+i})(1-\alpha\beta q^{N+i+1})(\gamma-\alpha q^{i})}{\alpha q(1-q^N)(1-\alpha\beta q^{2i})(1-\alpha \beta q^{2i+1})}\,,
    \\
     &  \Phi_i^{0+}=\lambda_{0,\rr}^+-\Phi_i^{+1,+}-\Phi_i^{-1,+}\,,\\
\end{align*}
\item[(qRII/I)] $\ox=x, \quad \oalpha=\alpha,\quad \obeta=\beta,\quad\ogamma=\gamma q\,,\quad \oN=N+1$
\begin{align*}
 &  \lambda_{x,\rr}^+=\frac{(1-\gamma q^{x+1})(1-q^{N-x+1})}{1-q^{N+1}}\,,
 \\
  &  \Phi_i^{+1,+}=\frac{(1-\alpha\beta q^{i+1})(1-\beta\gamma q^{i+1})(1-\beta\gamma q^{i+2})(1-\alpha q^{i+1})}
  {(1-\beta\gamma q)(1-\alpha\beta  q^{2i+1})(1-\alpha\beta  q^{2i+2})}\,,
  \\
 % &  \Phi_i^{0+}=\frac{(1-\beta\gamma q^{i+1})(\alpha q^i-\gamma)}{(1-\beta \gamma q)(1-\alpha\beta q^{2i+1})}\left( \frac{q(1-\alpha\beta q^{i+1})(1-\beta q^{i+1})}{1-\alpha\beta q^{2i+2}}+ \frac{\beta\gamma(1-q^i)(1-\alpha q^i)}{1-\alpha\beta q^{2i}}\right)
 % \\
    &  \Phi_i^{-1,+}=\frac{\beta q^2(1-q^i)(1-\beta q^i)(\gamma-\alpha q^{i-1})(\gamma-\alpha q^i)}{(1-\beta\gamma q)(1-\alpha\beta q^{2i})(1-\alpha\beta q^{2i+1})}\,,
    \\
    &  \Phi_i^{0+}=\lambda_{0,\rr}^+-\Phi_i^{+1,+}-\Phi_i^{-1,+}\,,\\
\end{align*}
\item[(qRII/III)] $\ox= x +1, \quad \oalpha=\alpha/q,\quad \obeta=\beta q,\quad \ogamma=\gamma /q\,,\quad \oN=N+1$
\begin{align*}
 &  \lambda_{x,\rr}^+=\frac{(1- q^{-x-1})(1-\gamma q^{x-N-1})}{(1-\alpha )(1-q^{N+1})}\,,
 \\
  &  \Phi_i^{+1,+}=\frac{q^{i-N-1}(1-\alpha\beta q^{i+1})(1-\beta\gamma q^{i+1})}{(1-\alpha\beta  q^{2i+1})(1-\alpha\beta  q^{2i+2})}\,,
  \\
  %&  \Phi_i^{0+}=-\frac{q^{i-N-2}}{1-\alpha\beta q^{2i+1}}\left( \frac{q(1-\alpha\beta q^{i+1})(1-\beta\gamma q^{i+1})}{1-\alpha\beta q^{2i+2}}  +\frac{\beta\gamma(1-q^i)(\gamma\alpha q^i)}{1-\alpha\beta q^{2i}}\right)\,,\\
    &  \Phi_i^{-1,+}=\frac{\beta q^{i-N-1}(1-q^i)(\gamma-\alpha q^i)}{(1-\alpha\beta q^{2i})(1-\alpha\beta q^{2i+1})}\,,
    \\
      &  \Phi_i^{0+}=-\Phi_i^{+1,+}-\Phi_i^{-1,+}\,,\\
\end{align*}
\item[(qRII/IV)] $\ox=x, \quad \oalpha=\alpha/q,\quad \obeta=\beta q,\quad \ogamma=\gamma\,,\quad \oN=N$
\begin{align*}
 &  \lambda_{x,\rr}^+=\frac{(1-\alpha q^{x})(\gamma-\alpha q^{N-x})}{\alpha(1-\alpha)}\,,
 \\
  &  \Phi_i^{+1,+}=-\frac{(q^N-q^i)(1-\alpha\beta q^{i+1})(1-\beta\gamma q^{i+1})(1-\beta\gamma q^{i+2})}{(1-\beta\gamma q)(1-\alpha\beta  q^{2i+1})(1-\alpha\beta  q^{2i+2})}\,,
  \\
  %&  \Phi_i^{0+}=\frac{(1-\beta\gamma q^{i+1})(\gamma-\alpha q^i)}{(1-\beta \gamma q)(1-\alpha \beta q^{2i+1})}\left(\frac{(1-\alpha\beta q^{i+1})(1-\alpha\beta q^{N+i+2})}{\alpha(1-\alpha\beta q^{2i+2})} +\frac{\beta\gamma(1-q^i)(q^N-q^{i-1})}{1-\alpha\beta q^{2i}}  \right)
  %\\
    &  \Phi_i^{-1,+}=-\frac{\beta q(1-q^i)(\gamma-\alpha q^{i-1})(\gamma-\alpha q^i)(1-\alpha\beta q^{N+i+1})}{\alpha(1-\beta\gamma q)(1-\alpha\beta q^{2i})(1-\alpha\beta q^{2i+1})}\,,
    \\
      &  \Phi_i^{0+}=\lambda_{0,\rr}^+-\Phi_i^{+1,+}-\Phi_i^{-1,+}\,,\\
\end{align*}
\item[(qRIII/I)] $\ox=x-1, \quad \oalpha=\alpha q,\quad \obeta=\beta/q,\quad \ogamma=\gamma q^2\,,\quad \oN=N$
\begin{align*}
 &  \lambda_{x,\rr}^+=(1-\gamma q^{x+1})(1-q^{N-x+1})\,,
 \\
  &  \Phi_i^{+1,+}=\frac{(1-q^{N-i})(1-\alpha q^{i+1})(1-\alpha q^{i+2})(1-\alpha\beta q^{i+1})(1-\beta\gamma q^{i+1})(1-\beta\gamma q^{i+2})}{(1-\alpha q)(1-\beta\gamma q)(1-\alpha\beta q^{2i+1})(1-\alpha\beta q^{2i+2})}\,,
  \\
  &  \Phi_i^{0+}=\frac{q(1-\alpha q^{i+1})(1-\beta q^i)(1-\beta\gamma q^{i+1})(\alpha q^i-\gamma)}{(1-\alpha q)(1-\beta\gamma q)(1-\alpha\beta q^{2i+1})}\\
 & \qquad\quad  \times\left(
  \frac{q(1-q^{N-i})(1-\alpha \beta q^{i+1})}{1-\alpha \beta q^{2i+2}}
  +\frac{(1-q^{-i})(1-\alpha\beta q^{N+i+1})}{1-\alpha\beta q^{2i}}
  \right)\,,
  \\
    &  \Phi_i^{-1,+}=\frac{q^3(1-q^{-i})(\alpha q^{i-1}-\gamma)(\alpha q^{i}-\gamma)(1-\beta q^{i-1})(1-\beta q^i)(1-\alpha\beta q^{N+i+1})}{(1-\alpha q)(1-\beta\gamma q)(1-\alpha\beta q^{2i})(1-\alpha\beta q^{2i+1})}\,,
    \\
\end{align*}
\item[(qRIII/II)] $\ox=x-1, \quad \oalpha=\alpha q,\quad \obeta=\beta/q,\quad \ogamma=\gamma q\,,\quad \oN=N-1$
\begin{align*}
 &  \lambda_{x,\rr}^+=\frac{(1-\beta\gamma q^{x})(1-\beta q^{N-x})}{\beta}\,,
 \\
  &  \Phi_i^{+1,+}=\frac{(1-q^{N-i})(q^i-q^{N-1})(1-\alpha q^{i+1})(1-\alpha q^{i+2})(1-\alpha\beta q^{i+1})(1-\beta\gamma q^{i+1})}{(1-q^N)(1-\alpha q)(1-\alpha\beta q^{2i+1})(1-\alpha\beta q^{2i+2})}\,,
  \\
  &  \Phi_i^{0+}=\frac{(1-q^{N-i})(1-\alpha q^{i+1})(1-\beta q^i)(1-\alpha\beta q^{N+i+1})}{(1-q^N)(1-\alpha q)(1-\alpha\beta q^{2i+1})}\\
  &\qquad \quad \left(
  \frac{(1-\alpha\beta q^{i+1})(1-\beta\gamma q^{i+1})}{\beta(1-\alpha\beta q^{2i+2})}+
  \frac{(1-q^i)(\gamma-\alpha q^i)}{1-\alpha \beta q^{2i}}
  \right)\,,
  \\
    &  \Phi_i^{-1,+}=\frac{q(1-q^i)(\gamma q^{-i}-\alpha )(1-\beta q^{i-1})(1-\beta q^i)(1-\alpha\beta q^{N+i})(1-\alpha\beta q^{N+i+1})}{\beta(1-q^N)(1-\alpha q)(1-\alpha\beta q^{2i})(1-\alpha\beta q^{2i+1})}\,,
    \\
\end{align*}
\item[(qRIII/IV)] $\ox=x-1, \quad \oalpha=\alpha,\quad \obeta=\beta,\quad \ogamma=\gamma q\,,\quad \oN=N-1$
\begin{align*}
 &  \lambda_{x,\rr}^+=\frac{(1-\alpha q^{x})(\gamma-\alpha q^{N-x})}{\alpha}\,,
 \\
  &  \Phi_i^{+1,+}=\frac{(1-q^{N-i})(q^i-q^{N-1})(1-\alpha q^{i+1})(1-\alpha\beta q^{i+1})(1-\beta\gamma q^{i+1})(1-\beta\gamma q^{i+2})}{(1-q^N)(1-\beta\gamma q)(1-\alpha\beta q^{2i+1})(1-\alpha\beta q^{2i+2})}\,,
  \\
  &  \Phi_i^{0+}=\frac{(1-q^{N-i})(1-\beta\gamma q^{i+1})(1-\alpha\beta q^{N+i+1})(\gamma-\alpha q^i)}{(1-q^N)(1-\beta \gamma q)(1-\alpha\beta q^{2i+1})}\\
  &\qquad \quad \times \left(
  \frac{(1-\alpha q^{i+1})(1-\alpha\beta q^{i+1})}{\alpha(1-\alpha\beta q^{2i+2})}
  +\frac{(1-q^i)(1-\beta q^i)}{1-\alpha\beta q^{2i}}
  \right)\,,
  \\
    &  \Phi_i^{-1,+}=\frac{(1-q^i)(\gamma q^{1-i}-\alpha q)(\gamma-\alpha q^{i-1})(1-\beta q^i)(1-\alpha\beta q^{N+i})(1-\alpha\beta q^{N+i+1})}{\alpha(1-q^N)(1-\beta\gamma q)(1-\alpha\beta q^{2i})(1-\alpha\beta q^{2i+1})}\,,
    \\
\end{align*}
\item[(qRIV/I)] $\ox=x, \quad \oalpha=\alpha q,\quad \obeta=\beta/q,\quad \ogamma=\gamma q\,,\quad \oN=N+1$
\begin{align*}
 &  \lambda_{x,\rr}^+=\frac{(1-\gamma q^{x+1})(1-q^{N-x+1})}{1-q^{N+1}}\,,
 \\
  &  \Phi_i^{+1,+}=\frac{(1-\alpha q^{i+1})(1-\alpha q^{i+2})(1-\alpha\beta q^{i+1})(1-\beta\gamma q^{i+1})}{(1-\alpha q)(1-\alpha \beta q^{2i+1})(1-\alpha\beta q^{2i+2})}\,,
  \\
 % &  \Phi_i^{0+}=\frac{q(1-\alpha q^{i+1})(1-\alpha\beta q^{i+1})(1-\beta q^i)}{(1-\alpha q)(1-\alpha \beta q^{2i+1})}\left(
 % \frac{\alpha q^{i+1}-\gamma}{1-\alpha\beta q^{2i+2}}-\frac{\alpha (1-q^i)}{1-\alpha\beta q^{2i}}
  %\right)\,,
  %\\
    &  \Phi_i^{-1,+}=-\frac{\alpha q^2(1-q^i)(1-\beta q^{i-1})(1-\beta q^i)(\alpha q^i-\gamma)}{(1-\alpha q)(1-\alpha\beta q^{2i})(1-\alpha \beta q^{2i+1})}\,,
    \\
    &  \Phi_i^{0+}=\lambda_{0,\rr}^+-\Phi_i^{+1,+}-\Phi_i^{-1,+}\,,\\
\end{align*}
\item[(qRIV/II)] $\ox=x, \quad \oalpha=\alpha q,\quad \obeta=\beta/q,\quad \ogamma=\gamma\,,\quad \oN=N$
\begin{align*}
 &  \lambda_{x,\rr}^+=\frac{(1-\beta\gamma q^{x})(1-\beta q^{N-x})}{\beta(1-\beta\gamma)}\,,
 \\
  &  \Phi_i^{+1,+}=\frac{(1-\alpha q^{i+1})(1-\alpha \beta q^{i+1})(q^i-q^N)(1-\alpha q^{i+2})}{(1-\alpha q)(1-\alpha \beta q^{2i+1})(1-\alpha \beta q^{2i+2})}\,,
  \\
%  &  \Phi_i^{0+}=\frac{(1-\alpha q^{i+1})(1-\beta q^i)}{(1-\alpha q)(1-\alpha \beta q^{2i+1})}\left(
 % \frac{(1-\alpha\beta q^{i+1})(1-\alpha\beta q^{N+i+2})}{\beta(1-\alpha\beta q^{2i+2})}
 % -\frac{\alpha q(1-q^i)(q^{i-1}-q^N)}{1-\alpha \beta q^{2i}}
 % \right)
  %\\
    &  \Phi_i^{-1,+}=-\frac{\alpha q(1-q^i)(1-\beta q^{i-1})(1-\beta q^i)(1-\alpha \beta q^{N+i+1})}{\beta(1-\alpha q)(1-\alpha \beta q^{2i})(1-\alpha\beta q^{2i+1})}
    \\
      &  \Phi_i^{0+}=\lambda_{0,\rr}^+-\Phi_i^{+1,+}-\Phi_i^{-1,+}\,,\\
\end{align*}
\item[(qRIV/III)] $\ox=x +1, \quad \oalpha=\alpha,\quad \obeta=\beta,\quad \ogamma=\gamma/q\,,\quad \oN=N+1$
\begin{align*}
 &  \lambda_{x,\rr}^+=\frac{(1- q^{-x-1})(1-\gamma q^{x-N-1})}{(1-\beta\gamma )(1-q^{N+1})}\,,
 \\
  &  \Phi_i^{+1,+}=\frac{q^{i-N-1}(1-\alpha q^{i+1})(1-\alpha\beta q^{i+1})}{(1-\alpha\beta q^{2i+1})(1-\alpha\beta q^{2i+2})}\,,
  \\
 % &  \Phi_i^{0+}=-\frac{q^{i-N-1}}{1-\alpha\beta q^{2i+1}}\left(
 % \frac{(1-\alpha q^{i+1})(1-\alpha\beta q^{i+1}) }{1-\alpha\beta q^{2i+2}}+
 % \frac{\alpha(1-q^i)(1-\beta q^i)}{1-\alpha \beta q^{2i}}
  %\right)\,,
  %\\
    &  \Phi_i^{-1,+}=\frac{\alpha q^{i-N-1}(1-q^i)(1-\beta q ^i)}{(1-\alpha\beta q^{2i})(1-\alpha\beta q^{2i+1})}\,.
    \\
    &  \Phi_i^{0+}=-\Phi_i^{+1,+}-\Phi_i^{-1,+}\,,\\
\end{align*}
\end{enumerate}
The $B_2$-contiguity relations for the other types of polynomials can be obtained by the different limits specified at the end of each of the Subsections \ref{sec:qHahn}-\ref{sec:kraw}.

\section{$B'_2$-contiguity relations for $q$-Racah polynomials \label{app:B2p}} 

The $B'_2$-contiguity relations read as follows
\begin{align}
& R_i(x; \rr)= \Phi^{0,+}_{i}\ R_{i}(\ox; \orr)
+\Phi^{-1,+}_{i}\ R_{i-1}(\ox; \orr)
+\Phi^{-2,+}_{i}\ R_{i-2}(\ox; \orr)\,,\\
&\lambda^{-}_{x; \rr}\;  R_i(\ox; \orr)=\Phi^{+2,-}_{i\ }\ R_{i+2}(x; \rr)+\Phi^{+1,-}_{i\ }\ R_{i+1}(x; \rr)+\Phi^{0,-}_{i\ }\ R_{i}(x; \rr)\,.
\end{align}
Let us emphasize that for all relations, one can choose
$\lambda_{x,\rr}^+=1$.
As proved in this paper, all the $B'_2$-contiguity relations for the $q$-Racah polynomials can be obtained from the $A_2$-contiguity relations following the procedure described in Remark \ref{rmk:B2'-from-A2}.

In the following, the complete list of the $B_2$-contiguity relations for $q$-Racah polynomials is given (the names refer to the two $A_2$-contiguity relations used to compute the $B'_2$-contiguity relations):
\begin{itemize}
 \item[(qRI/I')] $\ox=x,\quad \oalpha=\alpha\,,\quad \obeta=\beta q^2\,,\quad \ogamma=\gamma/q^2\,,\quad \oN=N-2\,$
  \begin{align*}
  &  \Phi_i^{0,+}=\frac{(1-q^{i-N})(1-q^{i-N+1})(1-\alpha\beta q^{i+1})(1-\alpha\beta q^{i+2})}{(1-q^{-N})(1-q^{-N+1})(1-\alpha\beta q^{2i+1})(1-\alpha\beta q^{2i+2})}\,,
  \\
  %&  \Phi_i^{-1,+}=\frac{(1-q^{i-N})(1-q^i)(1-\alpha\beta q^{N+i+1})(1-\alpha \beta q^{i+1})(1+q)}{q(1-q^{-N})(1-q^{N-1})(1-\alpha\beta q^{2i})(1-\alpha\beta q^{2i+2})}\,,
    &  \Phi_i^{-2,+}=\frac{(1-q^{i-1})(1-q^i)(1-\alpha\beta q^{N+i})(1-\alpha\beta q^{N+i+1})}{(1-q^{N-1})(1-q^{N})(1-\alpha\beta q^{2i})(1-\alpha\beta q^{2i+1})}\,,\\
     &\Phi_i^{-1,+}=1-\Phi_i^{0,+}-\Phi_i^{-2,+}\,,
\end{align*}
and
\begin{align*}
     &  \lambda^-_{x,\rr}=\frac{(1-\gamma q^{x-1})(1-\gamma q^x)(1-q^{N-x-1})(1-q^{N-x})}{(1-q^{N-1})(1-q^N)}\,,
    \\
     &  \Phi_i^{+2,-}=\frac{(1-\alpha q^{i+1})(1-\alpha q^{i+2})(1-\beta\gamma q^{i+1})(1-\beta\gamma q^{i+2})}{(1-\alpha \beta q^{2i+3})(1-\alpha \beta q^{2i+4})}
     \,,
  \\
    &  \Phi_i^{0,-}=\frac{(1-\beta q^{i+1})(1-\beta q^{i+2})(\alpha q^{i+1}-\gamma)(\alpha q^{i+2}-\gamma)}{q(1-\alpha \beta q^{2i+2})(1-\alpha \beta q^{2i+3})}
    \,,
    \\
    &  \Phi_i^{+1,-}=\lambda^-_{0,\rr}-\Phi_i^{+2,-}-\Phi_i^{0,-}\,,
  \\
\end{align*}

  \item[(qRI/II')] $\ox=x,\quad \oalpha=\alpha\,,\quad \obeta=\beta q^2\,,\quad \ogamma=\gamma/q\,,\quad \oN=N-1$
    \begin{align*}
  &  \Phi_i^{0,+}=\frac{(1-q^{i-N})(1-\beta\gamma q^{i+1})(1-\alpha\beta q^{i+1})(1-\alpha\beta q^{i+2})}{(1-q^{-N})(1-\beta\gamma q)(1-\alpha\beta q^{2i+1})(1-\alpha\beta q^{2i+2})}\,,
  \\
    &  \Phi_i^{-2,+}=-\frac{\beta q(1-q^{i-1})(1-q^i)(1-\alpha\beta q^{N+i+1})(\gamma-\alpha q^{i})}{(1-q^{N})(1-\beta\gamma q)(1-\alpha\beta q^{2i})(1-\alpha\beta q^{2i+1})}\,,\\
    &\Phi_i^{-1,+}=1-\Phi_i^{0,+}-\Phi_i^{-2,+}\,,
\end{align*}
and
\begin{align*}
     &  \lambda^-_{x,\rr}=\frac{(1-\gamma q^x)(1-q^{N-x})(1-\beta\gamma q^{x+1})(1-\beta q^{N+1-x})}{\beta q(1-q^N)(1-\beta \gamma q)}\,,
    \\
     &  \Phi_i^{+2,-}=\frac{(1-\alpha q^{i+1})(1-\alpha q^{i+2})(1-\beta\gamma q^{i+2})(q^{i+1}- q^{N})}{(1-\alpha \beta q^{2i+3})( 1-\alpha \beta q^{2i+4})}
     \,,
  \\
    &  \Phi_i^{0,-}=\frac{(1-\beta q^{i+1})(1-\beta q^{i+2})(\alpha q^{i+1}-\gamma)(1-\alpha\beta q^{N+i+2})}{\beta q(1-\alpha \beta q^{2i+2})(1-\alpha \beta q^{2i+3})}
    \,,
    \\
     &  \Phi_i^{+1,-}=\lambda^-_{0,\rr}-\Phi_i^{+2,-}-\Phi_i^{0,-}\,,
  \\
\end{align*}

   \item[(qRI/III')] $\ox=x-1,\quad \oalpha=\alpha q\,,\quad \obeta=\beta q\,,\quad \ogamma=\gamma\,,\quad \oN=N-2$
  \begin{align*}
  &  \Phi_i^{0,+}= \frac{(1-q^{i-N})(1-q^{N-i-1})(1-\alpha\beta q^{i+1})(1-\alpha\beta q^{i+2})(1-\alpha q^{i+1})(1-\beta\gamma q^{i+1})}{(1-q^{-N})(1-q^{N-1})(1-\alpha q)(1-\beta\gamma q)(1-\alpha \beta q^{2i+1})(1-\alpha \beta q^{2i+2})}\,,
  \\
  &  \Phi_i^{-1,+}=\frac{(1-q^{N-i})(1-\alpha\beta q^{i+1})(1-q^i)(1-\alpha\beta q^{N+i+1})}{(1-\alpha q)(1-\beta\gamma q)(1-q^{N-1})(1-q^N)(1-\alpha\beta q^{2i+1})}\\
  &\qquad\quad\times\left( \frac{(\gamma-\alpha q^{i+1})(1-\beta q^{i+1})}{1-\alpha\beta q^{2i+2}}
  +
  \frac{(1-\beta \gamma q^{i})(1-\alpha q^{i})}{1-\alpha\beta q^{2i}}
  \right)\,,
  \\
    &  \Phi_i^{-2,+}=\frac{(1-q^{i-1})(1-q^i)(1-\alpha\beta q^{N+i})(1-\alpha\beta q^{N+i+1})(\gamma q^{1-i}-\alpha q)(1-\beta q)}{(1-q^{N-1})(1-q^N)(1-\alpha q)(1-\beta\gamma q)(1-\alpha \beta q^{2i})(1-\alpha\beta q^{2i+1})}\,,
\end{align*}
and
\begin{align*}
     &  \lambda^-_{x,\rr}=\frac{(1-\gamma q^x)(1-q^{N-x})(1-q^{-x})(1-\gamma q^{x-N})}{(1-q^{N-1})(1-q^{N})(1-\alpha q)(1-\beta \gamma q)}\,,
    \\
     &  \Phi_i^{+2,-}=\frac{q^{i-N+1}(1-\alpha q^{i+2})(1-\beta\gamma q^{i+2})}{(1-\alpha \beta q^{2i+3})( 1-\alpha \beta q^{2i+4})}
     \,,
  \\
    &  \Phi_i^{0,-}=-
    \frac{q^{i-N+1}(1-\beta q^{i+1})(\alpha q^{i+1}-\gamma)}{(1-\alpha \beta q^{2i+2})(1-\alpha \beta q^{2i+3})}
    \,,
    \\
     &  \Phi_i^{+1,-}=-\Phi_i^{+2,-}-\Phi_i^{0,-}\,,
  \\
\end{align*}

   \item[(qRI/IV')] $\ox=x,\quad \oalpha=\alpha q\,,\quad \obeta=\beta q\,,\quad \ogamma=\gamma/q\,,\quad \oN=N-1$
      \begin{align*}
  &  \Phi_i^{0,+}=\frac{(1-q^{i-N})(1-\alpha\beta q^{i+1})(1-\alpha\beta q^{i+2})(1-\alpha q^{i+1})}{(1-q^{-N})(1-\alpha q)(1-\alpha \beta q^{2i+1})(1-\alpha \beta q^{2i+2})}\,,
  \\
    &  \Phi_i^{-2,+}=-\frac{\alpha q(1-q^{i-1})(1-q^i)(1-\alpha\beta q^{N+i+1})(1-\beta q^i)}{(1-q^N)(1-\alpha q)(1-\alpha\beta q^{2i})(1-\alpha\beta q^{2i+1})}\,,\\
     &  \Phi_i^{-1,+}=1-\Phi_i^{0,+}-\Phi_i^{-2,+}\,,
  \\
\end{align*}
and
\begin{align*}
     &  \lambda^-_{x,\rr}=\frac{(1-\gamma q^x)(1-q^{N-x})(1-\alpha q^{x+1})(\gamma-\alpha q^{N+1-x})}{\alpha q(1-q^N)(1-\alpha q)}\,,
    \\
     &  \Phi_i^{+2,-}=-\frac{(q^N-q^{i+1})(1-\alpha q^{i+2})(1-\beta\gamma q^{i+1})(1-\beta\gamma q^{i+2})}{(1-\alpha \beta q^{2i+3})(1-\alpha\beta q^{2i+4})}
     \,,
  \\
    &  \Phi_i^{0,-}=
    \frac{(1-\beta q^{i+1})(\alpha q^{i+2}-\gamma)(1-\alpha\beta q^{N+i+2})(\gamma-\alpha q^{i+1})}{\alpha q(1-\alpha\beta q^{2i+2})(1-\alpha \beta q^{2i+3})}
    \,,
    \\
     &  \Phi_i^{+1,-}=\lambda^-_{0,\rr}-\Phi_i^{+2,-}-\Phi_i^{0,-}\,,
  \\
\end{align*}

    \item[(qRII/II')] $\ox=x,\quad \oalpha=\alpha \,,\quad \obeta=\beta q^2\,,\quad \ogamma=\gamma\,,\quad \oN=N$
      \begin{align*}
  &  \Phi_i^{0,+}=\frac{(1-\alpha\beta q^{i+1})(1-\alpha\beta q^{i+2})(1-\beta\gamma q^{i+1})(1-\beta\gamma q^{i+2})}{(1-\beta \gamma q)(1-\beta \gamma q^2)(1-\alpha \beta q^{2i+1})(1-\alpha\beta q^{2i+2})}\,,
  \\
   &  \Phi_i^{-2,+}=\frac{\beta^2 q^3(1-q^{i-1})(1-q^i)(\gamma-\alpha q^{i-1})(\gamma-\alpha q^i)}{(1-\beta \gamma q)(1-\beta \gamma q^2)(1-\alpha\beta q^{2i})(1-\alpha\beta q^{2i+1})}\,,\\
     &  \Phi_i^{-1,+}=1-\Phi_i^{0,+}-\Phi_i^{-2,+}\,,
  \\
\end{align*}
and
\begin{align*}
     &  \lambda^-_{x,\rr}=\frac{(1-\beta\gamma q^{x+1})(1-\beta\gamma q^{x+2})(1-\beta q^{N+1-x})(1-\beta q^{N+2-x})}{\beta^2 q^3(1-\beta \gamma q)(1-\beta\gamma q^2)}\,,
    \\
     &  \Phi_i^{+2,-}=\frac{(q^i- q^{N})(q^{i+1}- q^{N})(1-\alpha q^{i+1})(1-\alpha q^{i+2})}{ (1-\alpha \beta q^{2i+3})( 1-\alpha \beta q^{2i+4})}
     \,,
  \\
    &  \Phi_i^{0,-}=\frac{(1-\beta q^{i+1})(1-\beta q^{i+2})(1-\alpha\beta q^{N+i+2})(1-\alpha\beta q^{N+i+3})}{\beta^2 q^3(1-\alpha \beta q^{2i+2})(1-\alpha \beta q^{2i+3})}
\,,
    \\
     &  \Phi_i^{+1,-}=\lambda^-_{0,\rr}-\Phi_i^{+2,-}-\Phi_i^{0,-}\,,
  \\
\end{align*}

       \item[(qRII/III')] $\ox=x-1,\quad \oalpha=\alpha q\,,\quad \obeta=\beta q\,,\quad \ogamma=\gamma q\,,\quad \oN=N-1$
   \begin{align*}
  &  \Phi_i^{0,+}=\frac{(1-\alpha\beta q^{i+1})(1-\alpha\beta q^{i+2})(1-\beta\gamma q^{i+1})(1-\beta\gamma q^{i+2})(1-q^{N-i})(1-\alpha q^{i+1})}{(1-\beta \gamma q)(1-\beta \gamma q^2)(1-q^N)(1-\alpha q)(1-\alpha \beta q^{2i+1})(1-\alpha\beta q^{2i+2})}\,,
  \\
  &  \Phi_i^{-1,+}=\frac{(1-q^i)(1-\alpha\beta q^{i+1})(1-\beta\gamma q^{i+1})(\gamma q^{1-i}-\alpha q)}{(1-\beta\gamma q)(1-\beta \gamma q^2)(1-\alpha\beta q^{2i+1})(1-q^N)(1-\alpha q)}\\
  &\qquad\quad \times \left(
  \frac{(1-\beta q^{i+1})(1-\alpha\beta q^{N+i+2})}{1-\alpha\beta q^{2i+2}}
  -\frac{\beta q^i (1-q^{N-i+1})(1-\alpha q^i)}{1-\alpha\beta q^{2i}}
  \right)\,,
  \\
    &  \Phi_i^{-2,+}=-\frac{\beta q(1-q^{i-1})(1-q^i)(\gamma q^{2-i}-\alpha q)(\gamma-\alpha q^i)(1-\beta q^i)(1-\alpha\beta q^{N+i+1})}{(1-\beta \gamma q)(1-\beta \gamma q^2)(1-q^N)(1-\alpha q)(1-\alpha\beta q^{2i})(1-\alpha\beta q^{2i+1})}\,,
 \end{align*}
and
\begin{align*}
     &  \lambda^-_{x,\rr}=\frac{(1-\beta\gamma q^{x+1})(1- q^{-x})(1-\beta q^{N+1-x})(1-\gamma q^{x-N})}{\beta q(1-\beta \gamma q)(1-\beta\gamma q^2)(1-\alpha q)(1-q^N)}\,,
    \\
     &  \Phi_i^{+2,-}=\frac{q^{i+1-N}(q^i- q^{N-1})(1-\alpha q^{i+2})}{ (1-\alpha \beta q^{2i+3})( 1-\alpha \beta q^{2i+4})}
     \,,
  \\
    &  \Phi_i^{0,-}=-\frac{q^{i-N-1}(1-\beta q^{i+1})(1-\alpha\beta q^{N+i+2})}{\beta (1-\alpha \beta q^{2i+2})(1-\alpha \beta q^{2i+3})}
    \,,
    \\
     &  \Phi_i^{+1,-}=-\Phi_i^{+2,-}-\Phi_i^{0,-}\,,
  \\
\end{align*}

  \item[(qRII/IV')] $\ox=x,\quad \oalpha=\alpha q\,,\quad \obeta=\beta q\,,\quad \ogamma=\gamma\,,\quad \oN=N$
  \begin{align*}
  &  \Phi_i^{0,+}=\frac{(1-\alpha q^{i+1})(1-\beta\gamma q^{i+1})(1-\alpha\beta q^{i+1})(1-\alpha\beta q^{i+2})}{(1-\alpha q)(1-\beta \gamma q)(1-\alpha\beta q^{2i+1})(1-\alpha\beta q^{2i+2})}\,,
  \\
 &  \Phi_i^{-2,+}=\frac{\alpha\beta  q^2(1-q^{i-1})(1-q^i)(1-\beta q^i)(\gamma-\alpha q^i)}{(1-\alpha q)(1-\beta\gamma q)(1-\alpha\beta q^{2i})(1-\alpha\beta q^{2i+1})}\,,\\
    &  \Phi_i^{-1,+}=1-\Phi_i^{0,+}-\Phi_i^{-2,+}\,,
  \\
  \end{align*}
and
\begin{align*}
     &  \lambda^-_{x,\rr}=\frac{(1-\beta\gamma q^{x+1})(1-\beta q^{N+1-x})(1-\alpha q^{x+1})(\gamma-\alpha q^{N+1-x})}{\alpha\beta q^2(1-\alpha q)(1-\beta \gamma q)}\,,
    \\
     &  \Phi_i^{+2,-}=-\frac{(q^i- q^{N})(q^N-q^{i+1})(1-\alpha q^{i+2})(1-\beta\gamma q^{i+2})}{( 1-\alpha \beta q^{2i+3})(1-\alpha\beta q^{2i+4})}
    \,,
  \\
    &  \Phi_i^{0,-}=\frac{(1-\beta q^{i+1})(\gamma-\alpha q^{i+1})(1-\alpha\beta q^{N+i+2})(1-\alpha\beta q^{N+i+3})}{\alpha \beta q^2(1-\alpha\beta q^{2i+2})(1-\alpha \beta q^{2i+3})}
    \,,
    \\
     &  \Phi_i^{+1,-}=\lambda^-_{0,\rr}-\Phi_i^{+2,-}-\Phi_i^{0,-}\,,
  \\
\end{align*}

    \item[(qRIII/III')] $\ox=x-2,\quad \oalpha=\alpha q^2\,,\quad \obeta=\beta\,,\quad \ogamma=\gamma q^2\,,\quad \oN=N-2$
  \begin{align*}
  &  \Phi_i^{0,+}=\frac{(1-q^{N-i-1})(1-q^{N-i})(1-\alpha q^{i+1})(1-\alpha q^{i+2})(1-\beta\gamma q^{i+1})(1-\beta\gamma q^{i+2})}{(1-q^{N-1})(1-q^N)(1-\alpha q)(1-\alpha q^2)(1-\beta \gamma q)(1-\beta \gamma q^2)}
  \\
  & \qquad\quad \times \frac{(1-\alpha\beta q^{i+1})(1-\alpha\beta q^{i+2})}{(1-\alpha\beta q^{2i+1})(1-\alpha\beta q^{2i+2})}\,,
  \\
  &  \Phi_i^{-1,+}=\frac{(1-q^i)(1-q^{N-i})(1-\alpha q^{i+1})(1-\beta\gamma q^{i+1})(\gamma q^{1-i}-\alpha q)(1-\beta q^i)}{(1-q^{N-1})(1-q^N)(1-\alpha q)(1-\alpha q^2)(1-\beta\gamma q)(1-\beta \gamma q^2)}\\
 & \qquad\quad \times \frac{(1-\alpha\beta q^{i+1})(1-\alpha\beta q^{N+i+1})(1+q)}{(1-\alpha\beta q^{2i})(1-\alpha\beta q^{2i+2})}\,,
  \\
    &  \Phi_i^{-2,+}=\frac{q^3(1-q^{i-1})(1-q^i)(\gamma q^{-i}-\alpha)(\gamma q^{1-i}-\alpha)(1-\beta q^{i-1})(1-\beta q^i)}{(1-q^{N-1})(1-q^N)(1-\alpha q)(1-\alpha q^2)(1-\beta\gamma q)(1-\beta \gamma q^2)}
  \\
  & \qquad\quad \times \frac{(1-\alpha\beta q^{N+i+1})(1-\alpha\beta q^{N+i+2})}{(1-\alpha\beta q^{2i})(1-\alpha\beta q^{2i+1})}\,,
  \\
  \end{align*}
and
\begin{align*}
     &  \lambda^-_{x,\rr}= \frac{(1-q^{-x})(1-q^{-x+1})(1-\gamma q^{x-N})(1-\gamma q^{x-N+1})}{(1-\alpha q)(1-\alpha q^2)(1-q^{N-1})(1-q^N)(1-\beta \gamma q)(1-\beta \gamma q^2)} \,,
    \\
     &  \Phi_i^{+2,-}=\frac{q^{2i-2N+2}}{( 1-\alpha \beta q^{2i+3})(1-\alpha \beta q^{2i+4})}\,,
  \\
    &  \Phi_i^{0,-}=\frac{q^{2i-2N+1}}{(1-\alpha \beta q^{2i+2})(1-\alpha \beta q^{2i+3})}
    \,,
    \\
     &  \Phi_i^{+1,-}=-\Phi_i^{+2,-}-\Phi_i^{0,-}\,,
  \\
\end{align*}

      \item[(qRIII/IV')] $\ox=x-1,\quad \oalpha=\alpha q^2\,,\quad \obeta=\beta\,,\quad \ogamma=\gamma q\,,\quad \oN=N-1$
    \begin{align*}
  &  \Phi_i^{0,+}=\frac{(1-q^{N-i})(1-\alpha q^{i+1})(1-\alpha q^{i+2})(1-\alpha \beta q^{i+1})(1-\alpha \beta q^{i+2})(1-\beta\gamma q^{i+1})}{(1-q^N)(1-\alpha q)(1-\alpha q^2)(1-\beta \gamma q)(1-\alpha\beta q^{2i+1})(1-\alpha \beta q^{2i+2})}\,,
  \\
  &  \Phi_i^{-1,+}=\frac{q(1-q^i)(1-\alpha q^{i+1})(1-\beta q^i)(1-\alpha \beta q^{i+1})}{(1-q^N)(1-\alpha q)(1-\alpha q^2)(1-\beta \gamma q)(1-\alpha \beta q^{2i+1})}\\
  &\qquad\quad \times
  \left(
 -\frac{\alpha q(1-q^{N-i})(1-\beta \gamma q^{i+1})}{1-\alpha\beta q^{2i+2}}
 +\frac{(\gamma q^{-i}-\alpha)(1-\alpha\beta q^{N+i+1})}{1-\alpha\beta q^{2i}}\right)\,,
  \\
    &  \Phi_i^{-2,+}=-\frac{\alpha q^3(1-q^{i-1})(1-q^i)(\gamma q^{-i}-\alpha)(1-\beta q^{i-1})(1-\beta q^i)(1-\alpha\beta q^{N+i+1})}{(1-q^N)(1-\alpha q)(1-\alpha q^2)(1-\beta\gamma q)(1-\alpha\beta q^{2i})(1-\alpha\beta q^{2i+1})}\,,
  \end{align*}
and
\begin{align*}
     &  \lambda^-_{x,\rr}=\frac{(1-q^{-x})(1-\gamma q^{x-N})(1-\alpha q^{x+1})(\gamma-\alpha q^{N+1-x})}{\alpha q (1-\alpha q)(1-\alpha q^2)(1-\beta \gamma q)(1-q^N)}\,,
    \\
     &  \Phi_i^{+2,-}=-
     \frac{q^{i-N}(q^N-q^{i+1})(1-\beta\gamma q^{i+2})}{(1-\alpha \beta q^{2i+3})(1-\alpha\beta q^{2i+4})}\,,
  \\
    &  \Phi_i^{0,-}=-
    \frac{q^{i-N-1}(1-\alpha\beta q^{N+i+2})(\gamma-\alpha q^{i+1})}{\alpha (1-\alpha\beta q^{2i+2})(1-\alpha \beta q^{2i+3})}
    \,,
    \\
    &  \Phi_i^{+1,-}=-\Phi_i^{+2,-}-\Phi_i^{0,-}\,,
  \\
\end{align*}

   \item[(qRIV/IV')] $\ox=x,\quad \oalpha=\alpha q^2\,,\quad \obeta=\beta\,,\quad \ogamma=\gamma\,,\quad \oN=N$
     \begin{align*}
  &  \Phi_i^{0,+}=\frac{(1-\alpha q^{i+1})(1-\alpha q^{i+2})(1-\alpha\beta q^{i+1})(1-\alpha \beta q^{i+2})}{(1-\alpha q)(1-\alpha q^2)(1-\alpha\beta q^{2i+1})(1-\alpha\beta q^{2i+2})}\,,\\
     &  \Phi_i^{-2,+}=\frac{\alpha^2 q^3(1-q^{i-1})(1-q^i)(1-\beta q^{i-1})(1-\beta q^i)}{(1-\alpha q)(1-\alpha q^2)(1-\alpha\beta q^{2i})(1-\alpha\beta q^{2i+1})}\,,\\
    & \Phi_i^{-1,+}=1-\Phi_i^{0,+}-\Phi_i^{-2,+}\,,
  \\
 \end{align*}
and
\begin{align*}
     &  \lambda^-_{x,\rr}=\frac{(1-\alpha q^{x+1})(1-\alpha q^{x+2})(\gamma-\alpha q^{N+1-x})(\gamma-\alpha q^{N+2-x})}{\alpha^2 q^3(1-\alpha q)(1-\alpha q^2)}\,,
    \\
     &  \Phi_i^{+2,-}=\frac{(q^N-q^{i})(q^N-q^{i+1})(1-\beta\gamma q^{i+1})(1-\beta\gamma q^{i+2})}{(1-\alpha\beta q^{2i+3})(1-\alpha\beta q^{2i+4})}
    \,,
  \\
    &  \Phi_i^{0,-}=\frac{(1-\alpha\beta q^{N+i+2})(1-\alpha\beta q^{N+i+3})(\gamma-\alpha q^{i+1})(\gamma-\alpha q^{i+2})}{\alpha^2 q^3(1-\alpha\beta q^{2i+2})(1-\alpha\beta q^{2i+3})}
  \,,
    \\
     &  \Phi_i^{+1,-}=\lambda^-_{0,\rr}-\Phi_i^{+2,-}-\Phi_i^{0,-}\,.
  \\
\end{align*}
\end{itemize}
The $B'_2$-contiguity relations for the other types of polynomials can be obtained by the different limits specified at the end of each of the Subsections \ref{sec:qHahn}-\ref{sec:kraw}.

\vspace{1cm}

\textbf{Statements and Declarations}

Competing interests: The authors declare no competing interests.

\bibliographystyle{utphys}
\bibliography{contiguity}

\end{document}